\documentclass[twoside]{article}
\usepackage[preprint]{aistats2024}

\author{
  Damien Scieur \\
  Samsung SAIL, Montreal\\
  \texttt{damien.scieur@gmail.com} \\
}

\usepackage[natbib, style=numeric, maxcitenames=1, uniquelist=false]{biblatex}
\usepackage[T1]{fontenc}
\usepackage{lmodern}
\usepackage{microtype}
\usepackage{appendix}
\usepackage{amssymb}

\usepackage{algorithm}
\usepackage{algpseudocode}

\usepackage{amsmath}

\usepackage{amsthm}
\usepackage{color}
\usepackage{enumitem}
\usepackage{graphicx} 
\usepackage{thm-restate}

\usepackage{hyperref}
\usepackage[nameinlink]{cleveref}

\DeclareMathOperator{\sign}{sign}
\def\defas{\stackrel{\text{def}}{=}}

\DeclareMathOperator*{\argmin}{{arg\,min}}

\newtheorem{assumption}{Assumption}
\newtheorem{proposition}{Proposition}
\newtheorem{lemma}{Lemma}

\newtheorem{corollary}{Corollary}

\algdef{SE}[DOWHILE]{Do}{doWhile}{\algorithmicdo}[1]{\algorithmicwhile\ #1}

\crefformat{equation}{(#2#1#3)}

\providecommand{\customgenericname}{}
\newtheorem{innercustomgeneric}{\customgenericname}
\newcommand{\newcustomtheorem}[2]{%
  \newenvironment{#1}[1]
  {%
   \renewcommand\customgenericname{#2}%
   \renewcommand\theinnercustomgeneric{##1}%
   \innercustomgeneric
  }
  {\endinnercustomgeneric}
}

\newcustomtheorem{requirement}{Requirement}

\crefname{assumption}{Assumption}{Assumptions}
\crefname{innercustomgeneric}{Requirement}{Requirements}
\crefname{theorem}{Theorem}{Theorems}

\definecolor{mydarkblue}{rgb}{0,0.08,0.45}
\hypersetup{
    colorlinks=true,
    linkcolor=mydarkblue,
    citecolor=mydarkblue,
    filecolor=mydarkblue,
    urlcolor=mydarkblue,
    pdfview=FitH,
    breaklinks=true
}

\begin{document}

\runningtitle{Adaptive Quasi-Newton and Anderson Acceleration with Global Convergence Rates}

\twocolumn[
\aistatstitle{Adaptive Quasi-Newton and Anderson Acceleration Framework with Explicit Global (Accelerated) Convergence Rates}
\aistatsauthor{ Damien Scieur }
\aistatsaddress{  Samsung SAIL Montreal } ]


\begin{abstract}
Despite the impressive numerical performance of the quasi-Newton and Anderson/nonlinear acceleration methods, their global convergence rates have remained elusive for over 50 years. This study addresses this long-standing issue by introducing a framework that derives novel, adaptive quasi-Newton and nonlinear/Anderson acceleration schemes. Under mild assumptions, the proposed iterative methods exhibit explicit, non-asymptotic convergence rates that blend those of the gradient descent and Cubic Regularized Newton's methods. The proposed approach also includes an accelerated version for convex functions. Notably, these rates are achieved adaptively without prior knowledge of the function's parameters. The framework presented in this study is generic, and its special cases includes algorithms such as Newton's method with random subspaces, finite-differences, or lazy Hessian. Numerical experiments demonstrated the efficiency of the proposed framework, even compared to the l-BFGS algorithm with Wolfe line-search.
\end{abstract}

\section{Introduction}

Consider the problem of determining the minimizer $x^\star$ of the unconstrained minimization problem
\[
    f^\star \defas f(x^\star) = \min_{x\in\mathbb{R}^d} f(x),
\]
where $d$ is the problem dimension, and the function $f$ has a Lipschitz continuous Hessian.
\begin{assumption}\label{assump:lipchitiz_hessian}
    The function $f(x)$ has a Lipschitz continuous Hessian with constant $L$:
\begin{equation} \label{eq:def_Lipschitz_hess}
    \forall \;\; y,\,z\in\mathbb{R}^d,\quad \|\nabla^2 f(z)-\nabla^2 f(y)\| \leq L\|z-y\|.
\end{equation}
\end{assumption} 
In this study, $\|.\|$ stands for the maximal singular value of a matrix and the $\ell_2$ norm for a vector. Numerous twice-differentiable problems, such as logistic or least-squares regression, satisfy \cref{assump:lipchitiz_hessian}.

The Lipschitz continuity of the Hessian is crucial when analyzing second-order algorithms because it extends the concept of smoothness to second-order. The groundbreaking work by \citet{nesterov2006cubic} revealed the remarkable convergence rate of Newton's method for problems satisfying \cref{assump:lipchitiz_hessian} when augmented with cubic regularization. For instance, if the problem is convex, the accelerated gradient descent typically achieves $O(\frac{1}{t^2})$, whereas accelerated second-order methods achieve $O(\frac{1}{t^3})$. Recent advancements achieved even faster convergence rates, up to $\mathcal{O}(\frac{1}{t^{7/2}})$ using hybrid methods \citep{monteiro2013accelerated,carmon2022recapp}, or the direct acceleration of second-order methods \citep{nesterov2008accelerating,gasnikov2019near,kovalev2022first}.

However, second-order methods are not scalable, particularly for the high-dimensional problems common in machine learning. The limitation is that exact second-order methods require the solution of a linear system involving the Hessian $\nabla^2 f$. This has motivated alternative approaches that balance the efficiency of second-order methods with the scalability of first-order methods, such as quasi-Newton methods or Anderson/nonlinear acceleration methods (which are equivalent to quasi-Newton methods \cite{fang2009two}). Due to space limitation, the results of this study are presented under the prism of quasi-Newton methods, but their links with Anderson acceleration is explained extensively in \cref{sec:link_existing_methods}.

Quasi-Newton (qN) methods minimize differentiable functions by iteratively updating an approximate Hessian matrix using previous gradients, effectively balancing scalability and efficiency. This approach makes them highly suitable for large-scale optimization problems across diverse fields. For instance, l-BFGS is a widely used and effective optimization method for unconstrained functions (e.g., \texttt{fminunc} from Matlab) and is often considered state-of-the-art \citep{aggrawal2021hessian}.

Despite the impressive performance of quasi-Newton methods and nonlinear acceleration schemes, the following long-standing question \textbf{has remained unanswered for over 50 years}.
\begin{center}
\textit{What are the nonasymptotic global convergence rates of quasi-Newton and Anderson acceleration methods?}
\end{center}

This question is challenging: Over the years, extensive research has catapulted the popular l-BFGS algorithm to an exceptionally high level of efficiency, as attested by various studies (e.g., \citep{liu1989limited,venter2010review,morales2002numerical}). However, its theoretical convergence guarantees are notably lacking and do not accurately reflect its actual numerical performance. Therefore, additional numerical improvements or obtaining fast rates without arming the numerical convergence may be increasingly difficult or infeasible.


\subsection{Contributions}

\paragraph{Theoretical guarantees} This study presents generic updates that are novel quasi-Newton methods or nonlinear acceleration schemes with cubic regularization that meet the following requirements.
\begin{enumerate}
    \item The assumptions for the theoretical analysis are simple and verifiable (\cref{sec:assump}).
    \item The algorithm is suitable for large-scale problems, as for a fixed memory budget $N\leq d$, its per-iteration cost is linear in the dimension.
    \item The algorithm exhibits \textbf{explicit, global, and nonasymptotic convergence rates} that interpolate between those of the first and second order methods (\cref{sec:rate_convergence,sec:comparison_rate}):
    \begin{itemize}[leftmargin=5ex]
        \item Convergence on Nonconvex problems (\cref{thm:rate_nonconvex}): $\min_{i\leq t} \|\nabla f(x_i) \| \leq O(t^{-\frac{2}{3}}+t^{-\frac{1}{3}})$,
        \item (Star-)convex problems (\cref{thm:rate_starconvex,thm:rate_randomconvex}): $f(x_t)-f^\star \leq O(t^{-2}+t^{-1})$,
        \item Accelerated rate on convex problems (\cref{thm:rate_acc_sketch}): $f(x_t)-f^\star \leq O(t^{-3}+t^{-2})$.
    \end{itemize}
    \item The algorithm\textbf{ is adaptive to the problem’s constants} (\cref{alg:type1,alg:accelerated_algo}): both accelerated and classical methods require only an initial estimate of the Lipchitz constant (\cref{sec:backtracking_line_search}).
\end{enumerate}

\paragraph{Novel Analysis} To address these points, this study explores a novel paradigm, \textit{rethinking from scratch} the framework underlying qN methods (\cref{sec:detail_construction_algo}).

\paragraph{Numerical Efficiency} The algorithm outperform l-BFGS in many scenarios (\cref{sec:numerical_experiments}).

\paragraph{Practicability} A particular focus have been put on making the method \textit{simple}, \textit{generic} and \textit{adaptive}, to make it suitable for practical applications. The method is simple to implement, and requires \textit{fewer} hyperparameters than classical qN schemes.

\paragraph{Generic Framework} The framework can be instantiated as many kind of previously known methods (\cref{sec:special_cases}), and recovers the cubic regularization of Newton's method in its most extreme case.

\subsection{Limitations in Current qN Schemes} 

In most classical qN methods, a \textbf{(Wolfe) line search algorithm} (often in addition to other techniques such as secant equation filtering or re-scaling) is required to ensure global convergence. Without such a line search, the performance of qN methods is poor, even in a simple quadratic case in two dimensions \cite{powell1986bad}.

Nevertheless, some previous work already attempted to determine rates for qN methods (or to derive new ones), but often violates at least one of the previous points: \textbf{1)} the analysis requires non-verifiable assumptions, \textbf{2)} the algorithm is not suitable for large-scale problems as the per-iteration cost is at least $O(d^2)$,
\textbf{3)} the rates are local or do not interpolate between first and second order rates, \textbf{4)} the algorithm depends on potentially unknown parameters. See \cref{sec:related_work} for a detailed literature review.

\textbf{Violates 1:}  For instance, the ARC method \citep{cartis2011adaptive,cartis2011adaptive2} or proximal qN methods \citep{wei2004superlinear,scheinberg2016practical} show accelerated rates for quasi-Newton under similar assumptions. However, the authors state that the convergence rate is derived under a nonverifiable assumption and typically their rates do not rely on or exploit the accuracy of second-order approximations.

\textbf{Violates 2:} By using online algorithms and the Monteiro-Svaiter acceleration technique, \citep{jiang2023accelerated} achieves accelerated rates $O(\min\{\frac{1}{t^2}, \frac{1}{t^{2.5}}\})$ for qN methods, but require the storage and inversion of a $d\times d$ matrix, which does not scale well in high dimension, and does not perform well compared to BFGS.

\textbf{Violates 3:}  Recent research on quasi-Newton updates has unveiled explicit and nonasymptotic rates of convergence \cite{rodomanov2021greedy,rodomanov2021rates,rodomanov2021new,lin2022explicit,lin2021greedy}. Nonetheless, these analyses suffer from several significant drawbacks, as they are \textit{local}, full-memory (hence require storing a $d\times d$ matrix and the per-iteration cost is $O(d^2)$) and sometimes requires access to the Hessian matrix.

\textbf{Violates 4:} \citet{kamzolov2023accelerated} introduced an adaptive regularization technique combined with cubic regularization, but the method relies on the knowledge of $L$ in \cref{assump:lipchitiz_hessian}.

\section{Rethinking Quasi-Newton Methods} \label{sec:detail_construction_algo}

The starting point is the cubic upper bound on the objective function $f$ and the upper bound on the gradient variation derived from \cref{assump:lipchitiz_hessian} \cite{nesterov2006cubic}:
\begin{align}
    & \textstyle \|\nabla f(y)\hspace{-0.25ex}-\hspace{-0.25ex}\nabla f(x)-\nabla^2f(x)(y-x)\| \leq \frac{L}{2} \|y-x\|^2, \label{eq:ineq_secant}\\
    & \textstyle  f(y) \leq f(x) +\nabla f(x)(y-x)  \nonumber\\
    &\qquad + \frac{1}{2}(y-x)^T\nabla^2 f(x)(y-x)+ \frac{L}{6}\|y-x\|^3,\label{eq:ineq_function}
\end{align}
which holds for all $x,y\in\mathbb{R}^d$. Minimizing \cref{eq:ineq_function} over $y$ yields the cubic regularization of Newton's method \cite{nesterov2006cubic}.

The main steps in deriving this algorithm are as follows: 
\textbf{1)} The minimization will be constrained to a subspace of dimension $N\leq d$, reducing the per-iteration computation cost. 
\textbf{2)} The Hessian (in the subspace) will be approximated using differences of gradients.
\textbf{3)} From the previous points, an upper bound for the objective function and the gradient norm will be constructed, leading to a Type I and Type II method.
\textbf{4)} To ensure convergence, the update of the subspace will be chosen such that it spans the gradient (deterministic) or spans a portion of the gradient in expectation (random subspace).

\subsection{First Ingredient: Subspace Minimization} 

Minimizing the upper bound \cref{eq:ineq_function} is expensive in high dimensions because it requires an eigenvalue decomposition of the Hessian $\nabla^2 f(x)$ \citep{nesterov2006cubic}. Rather, let $D_t$ be an $d\times N$ matrix of directions (the construction of $D_t$ is described in \cref{sec:direction_requirements}). By constraining the update $x_{t+1}-x_t$ in the columns span of $D_t$, that is,
\begin{equation}\label{eq:def_next_iterate}
    x_{t+1} = x_t + D_t\alpha_t,
\end{equation}
where $\alpha_t$ is a vector of $N$ coefficients, the minimization problem is simplified to
\begin{align*}
    \alpha_{t} =  \argmin_{\alpha\in\mathbb{R}^N} & \;\;f(x_t) + \nabla f(x_t)D_t\alpha\\
    & \;\; \textstyle + \frac{1}{2}(D_t\alpha)^T\nabla^2 f(x_t)D_t\alpha + \frac{L}{6}\|D_t\alpha\|^3.
\end{align*}
Hence, the complexity of the eigenvalue decomposition of $D_t^T\nabla^2 f(x)D_t$ is now $O(N^2d + N^3)$.

\subsection{Second Ingredient: Multisecant Approximation of the Hessian} 

Typically, quasi-Newton methods approximate the Hessian using the properties of the \textit{secant equation}:
\[
    \nabla^2 f(x_i)(x_i-x_{i-1}) \approx \nabla f(x_i)-\nabla f(x_{i-1}).
\]
Usually, those updates are performed recursively, that is, by updating an approximation of the Hessian one secant equation at a time. 

This study approximates the Hessian using the secant equations simultaneously by using a finite-difference approximation between two sets of points $Y_t$ and $Z_t$.

More formally, let the matrix of directions $D_t$ and the matrix of normalized gradient differences $G_t$ be
\begin{align}\label{def:matrices_2}
    D_t & = \left[\frac{y^{(t)}_1-z^{(t)}_1}{\|y^{(t)}_1-z^{(t)}_1\|_2},\ldots, \frac{y^{(t)}_N-z^{(t)}_N}{\|y^{(t)}_N-z^{(t)}_N\|_2}\right],\\
    G_t & = \left[\ldots,\frac{\nabla f(y^{(t)}_i)-\nabla f(z^{(t)}_i)}{\|y^{(t)}_i-z^{(t)}_i\|_2},\ldots\right].\nonumber
\end{align}
and let the matrices $Y_t$, $Z_t$ be defined as:
\begin{align}\label{def:matrices}
    Y_t = [y_1^{(t)},\ldots, y_N^{(t)}],\quad Z_t = [z_1^{(t)},\ldots, z_N^{(t)}].
\end{align}
A naïve approach can be $Y_t = [x_{t-N},\ldots, x_{t-1}]$ and $Z_t = [x_{t-N+1},\ldots, x_{t}]$ (this will \textbf{not} be the case in this study; see \cref{sec:direction_requirements}). Intuitively, the matrix $G_t$ is a finite-difference approximation of the Hessian matrix product $\nabla^2f(x_t)D_t$. More precisely, the next theorem states a bound on the approximation error of this product as a function of the \textit{error vector} $\varepsilon_t$:
\begin{align} \label{eq:error_vector}
    \varepsilon_t \defas [e^{(t)}_1,\ldots,e^{(t)}_N], \; e^{(t)}_i \defas \|y_i^{(t)}-z_i^{(t)}\| + 2\|z^{(t)}_i-x_t\|
\end{align}
\vspace{-2ex}
\begin{restatable}{theorem}{thmboundsecantalpha}\label{thm:bound_secant_alpha} Let the function $f$ satisfy \cref{assump:lipchitiz_hessian}. Let matrices $D_t,\,G_t$ be defined as in \cref{def:matrices} and vector $\varepsilon$ as in \cref{eq:error_vector}. 
Subsequently, for all $w\in\mathbb{R}^d$ and $\alpha\in\mathbb{R}^N$
\begin{align}
    \left| w^T(\nabla^2f(x_t)D_t-G_t)\alpha\right| &\leq \textstyle \frac{L\|w\|}{2} |\alpha|^T\varepsilon_t, \label{eq:bound_hessian_scalar}\\
    \|w^T(\nabla^2f(x_t)D_t-G_t)\| &\leq  \textstyle \frac{L\|w\|}{2} \|\varepsilon_t\|. \label{eq:bound_hessian_dalpha}
\end{align}
\end{restatable}
\paragraph{Proof sketch} The detailed proof can be found in \cref{sec:missing_proofs}. The main idea of the proof is as follows. The superscript$~^{(t)}$ has been removed for clarity. From \cref{eq:ineq_secant} with $y=y_i$, $z=z_i$, then \cref{assump:lipchitiz_hessian}, 
\begin{align*}
    & \frac{\|\nabla f(y_i)-\nabla f(z_i)-\nabla^2f(x_t)(y_i-z_i)\|}{\|y_i-z_i\|}\\  & \leq  \frac{L}{2}\|y_i-z_i\| + \|\nabla^2f(x_t)-\nabla^2f(z_i)\| \leq \frac{L}{2}e_i.
\end{align*}
    
The \textit{first} term in $e_i$ bounds the error of \cref{eq:ineq_secant}, whereas the \textit{second} term originates from the distance between  \cref{eq:ineq_secant} and the current point $x_t$ where the Hessian is estimated. Subsequently, it suffices to combine the inequalities with the coefficients $\alpha_t$ to obtain \cref{thm:bound_secant_alpha}.

\subsection{Third Ingredient: Objective Function and Gradient Norm Upper bounds}

As the approximation error between $\nabla^2 f(x_t) D_t$ and $G_t$ can be explicitly bounded, carefully replacing the term $\nabla^2 f(x_t) D_t\alpha_t$ in \cref{eq:ineq_function,eq:ineq_secant} with $G_t\alpha_t$, along with an appropriate regularization, leads to \textbf{Type I} and \textbf{Type II} bounds.

\begin{restatable}{theorem}{thmupperboundalgo} \label{thm:upper_bound_algo} Let the function $f$ satisfy \cref{assump:lipchitiz_hessian}. Let $x_{t+1}$ be defined as in \cref{eq:def_next_iterate}; $D_t,\,G_t$ be defined as in \cref{def:matrices} and $\varepsilon_t$ be defined as in \cref{eq:error_vector}. Subsequently, $\forall\alpha\in\mathbb{R}^N$,
\begin{align}
    & \textstyle  f(x_{t+1}) \leq f(x_t) + \nabla f(x_t)^TD_t\alpha + \frac{\alpha^TH_t\alpha}{2} + \frac{L\|D_t\alpha\|^3}{6}, \tag{Type I bound}\label{eq:type1_bound}\\
    & \textstyle \|\nabla f(x_{t+1})\| \leq \left\|\nabla f(x_t) + G_t\alpha\right\| + \frac{L}{2}\Big( |\alpha|^T \varepsilon_t + \|D_t\alpha\|^2\Big), \tag{Type II bound} \label{eq:type2_bound}
\end{align}
where $H_t \defas  (G_t^TD_t+D_t^TG_t + \mathrm{I}L\|D_t\|\|\varepsilon_t\|)/2$.
\end{restatable}
The proof can be found in \cref{sec:missing_proofs}. Minimizing \cref{eq:type1_bound,eq:type2_bound} leads to \cref{alg:type1,alg:type2}, respectively, whose constant $L$ is replaced by a parameter $M$ determined by a backtracking line search. Type I methods often refer to algorithms that aim to minimize the function value $f(x)$, whereas type II methods minimize the gradient norm $\|\nabla f(x)\|$ \citep{fang2009two,zhang2020globally,canini2022direct}.

\paragraph{Solving the sub-problems} In \cref{alg:type1,alg:type2}, the coefficients $\alpha$ are computed by solving a minimization sub-problem in $O(N^3+Nd)$, where in practice $N$ is smaller than $d$ (see all the details in \cref{sec:solving_subproblem}). In \cref{alg:type1}, the subproblem can be formulated as a convex problem in two variables, by using an eigenvalue decomposition of the matrix $H \in\mathbb{R}^{N\times N}$  \citep{nesterov2006cubic}, while for \cref{alg:type2}, the subproblem can be cast into a linear-quadratic problem of $O(N)$ variables and constraints that can be solved efficiently with SDP solvers.

\subsection{Fourth Ingredient: Direction Update Rule} \label{sec:direction_requirements}

\begin{algorithm}[t]
\caption{"Orthogonal forward estimate only"}\label{alg:forward_estimate}
\begin{algorithmic}[1]
\Require First-order oracle for $f$, step-size $h$, matrices $D_{t-1}$, $G_{t-1}$, $Y_{t-1}$, $Z_{t-1}$, new point $x_t$.
\State \textbf{If} \# columns of $D_{t-1}$, $G_{t-1}$, $Y_{t-1}$, $Z_{t-1}$ is larger than $N$, \textbf{then} remove their first column.
\State Compute $g_t = \nabla f(x_t)$, then compute $d_t = -\frac{\tilde d}{\|\tilde d\|}$, where $\tilde d = g_t - D_{t-1}(D_{t-1}^Tg_t)$.
\State Compute $x_{t+\frac{1}{2}} = x_t + h d_t$, the \textit{orthogonal forward estimate}.
\State Update $Y_t = [Y_{t-1}, x_{t+\frac{1}{2}}]$, $Z_t = [Z_{t-1}, x_t]$, $\qquad$ \mbox{$D_t = [D_{t-1}, d_t]$}, $G_t = \cref{def:matrices_2}$, $\varepsilon = \cref{eq:error_vector}$ .
\State \Return $\nabla f(x_{t})$, $D_{t}$, $G_{t}$, $Y_{t}$, $Z_{t}$, $\varepsilon_t$.
\end{algorithmic}
\end{algorithm}

\begin{algorithm}[t]
\caption{"Orthogonal random dir." (example)}\label{alg:ortho_random}
\begin{algorithmic}[1]
\Require First-order oracle for $f$, step-size $h$, memory $N$, new point $x_t$.
\State Generates $N$ random orthonormal directions, for example, $[D_t,] = \texttt{qr}(\texttt{Rand}(d,N))$.
\State Create $Z_t = [x_t,\ldots, x_t]$, $Y_t = Z_t + hD_t$, then update $G_t = \cref{def:matrices_2}$, $\varepsilon = [h,\ldots,h]$ .
\State \Return $\nabla f(x_{t})$, $D_{t}$, $G_{t}$, $Y_{t}$, $Z_{t}$, $\varepsilon_t$.
\end{algorithmic}
\end{algorithm}

One critical theoretical property in the analysis is the alignment of the gradient $\nabla f(x_t)$ with the directions in $D_t$. This section presents updates that ensure good theoretical properties of $D_t$ (see summary in \cref{tab:summary_update_rule}).

Below are some assumptions for updating $Y_t,\,Z_t,\,D_t$, which are called \textbf{requirements}. While not overly restrictive, naïve methods, such as keeping only the previous iterations, will not satisfy these requirements.

All convergence results rely on \emph{one} of these conditions on the projector onto $\textbf{span}(D_t)$,
\begin{align}
    \textstyle P_t \defas D_t(D_t^TD_t)^{-1}D_t^T. \label{def:p_projector}
\end{align}
\begin{requirement}{1a}\label{assump:random_projector}
    For all $t$, projector $P_t$ of stochastic matrix $D_t$ satisfies $\mathbb{E}[P_t] = \frac{N}{d}\textbf{I}$.
\end{requirement}
\begin{requirement}{1b} \label{assump:projector_grad}
    For all $t$, $P_t \nabla f(x_t) = \nabla f(x_t)$.
\end{requirement}
The first condition guarantees that $D_t$ partially spans the gradient $\nabla f(x_t)$ in expectation because $\mathbb{E}[P_t\nabla f(x_t)] = \frac{N}{d}\nabla f(x_t)$. The second condition requires the possibility of moving toward the current gradient when taking the step $x_t+D_t\alpha_t$.

The norm of the relative error must be bounded.

\begin{requirement}{2}\label{assump:bounded_epsilon}
    For all $t$, $(\|\varepsilon_t\|/\|D_t\|) \leq \delta $.
\end{requirement}
The \cref{assump:bounded_epsilon} is also nonrestrictive, as it simply prevents the secant equations at $y_i-z_i$ and $z_i-x_i$ from diverging significantly. Generally, $\delta$ satisfies the (very) crude bound $\delta \leq O(\|x_0-x^\star\|)$.

Finally, the condition number of the matrix $D$ must also be bounded. The following section provides explicit updates for ensuring that this condition is satisfied.
\begin{requirement}{3}\label{assump:bounded_conditionning}
    For all $t$, the condition number $\kappa_{D_t} \defas \sqrt{\|D_t^TD_t\|\|(D_t^TD_t)^{-1}\|}$ is bounded as $\kappa_{D_t} \leq \kappa$.
\end{requirement}

\begin{table}[t]
    \centering
    \begin{tabular}{llccc}
        \hline
                            & Complexity & \# Grads & $\kappa$ & $\delta$  \\
        \hline
        Forw. est.          & $O(Nd)$       & 2     & 1 & $O(R)$\\
        Random              & At least $O(d)$ & $N+1$ & 1 & $O(h)$\\
        Pruning             & $O(Nd+N^2)$   & 2     & ? & $O(R)$\\
        Orthogonal.         & $O(N^2d)$     & $N+1$ & 1 & $O(h)$ \\
        \hline
    \end{tabular}
    \caption{Comparison between different updates rules: "forward estimates only," "orthogonal random directions," pruning or orthogonalization. If the gradient computation is costly or if $N$ is large, the forward estimate is probably the best method as it only requires the computation of two new gradients, $\nabla f(x_t)$ and $\nabla f(x_{t+\frac{1}{2}})$. Otherwise, the orthogonalization or random directions might be the methods of choice given their small constants $\delta$ and $\kappa$.}
    \label{tab:summary_update_rule}
\end{table}

\subsubsection{"Orthogonal Forward Estimate Only" Update Rule (Recommended)}

The \textit{"orthogonal forward estimate only"} update maintains $D_t$ orthonormal, that is, $D_t^TD_t = I$ for all $t$, while ensuring that $\nabla f(x_t)$ belongs to the span of columns of $D_t$ (see \cref{alg:forward_estimate}). These conditions are ensured thanks to an intermediate iterate $x_{t+\frac{1}{2}}$ which is used to estimate $\nabla^2 f(x_t)\nabla f(x_t)$; this is called the \textbf{orthogonal forward estimate}:
\[
    x_{t+\frac{1}{2}} \hspace{-0.5ex} = \hspace{-0.5ex} x_t - hd_t, \; d_t\hspace{-0.5ex} = \hspace{-0.5ex}\frac{\nabla f(x_t)- \tilde D_{t-1} (\tilde D_{t-1}^T \nabla f(x_t))}{\|\nabla f(x_t)- \tilde D_{t-1} (\tilde D_{t-1}^T \nabla f(x_t))\|}, 
\]
where $h>0$ is a small step size and $\tilde D_{t-1}$ is simply the matrix $D_{t-1}$ whose first column has been removed if its number of columns equals $N$. This estimate is a gradient descent step projected onto the orthogonal space of $\textbf{span}(\tilde D_{t-1})$, and is inexpensive because the orthogonality of $D_t$ is maintained over the iterations.

After computing the forward estimate, the matrices $Y_t$ and $Z_t$ are updated as the moving history of the previous forward iterates and previous iterates:
\[
    Y_t = [x_{t-N+\frac{3}{2}},\ldots,x_{t+\frac{1}{2}}],\qquad Z_t = [x_{t-N+1},\ldots,x_{t}].
\]
Matrices $D_t$ and $G_t$ then follow \cref{def:matrices_2}. See \cref{alg:forward_estimate} for a detailed implementation. This method has several advantages; it ensures good theoretical performance, particularly because $\kappa=1$ (see \cref{thm:forward_estimate_properties}), at the cost of only one extra gradient evaluation.

\begin{restatable}{theorem}{thmforwardestimateproperties}\label{thm:forward_estimate_properties}
    The \textit{``orthogonal forward estimate only''} update described in \cref{alg:forward_estimate} satisfies \cref{assump:bounded_conditionning,assump:projector_grad} with $\kappa = 1$.
\end{restatable}

\subsubsection{"Random Orthogonal Directions"}

The "random orthogonal direction" update (see \cref{alg:ortho_random}) generates a batch of $N$ random orthogonal directions at each iteration, such that
\[
    \mathbb{E}[D_tD_t^T] = \frac{N}{d}I.
\]
Subsequently, it remains to update the matrices $Y_t,\, Z_t, G_t$,
\[
    Z_t = [\underbrace{x_t,\ldots, x_t}_{N \text{ times}}], \quad Y_t = Z_t + hD_t, \quad G_t = \cref{def:matrices_2}.
\]
For instance, $D_t$ could be the $Q$ matrix of a \texttt{qr} decomposition of a random $N\times d$ matrix (complexity: $O(N^2 d)$), or even simpler, an aggregation of random canonical vectors (e.g., see \cite{hanzely2020stochastic}, where the complexity is $O(Nd)$). The major advantages of this approach are that $\kappa = 1$ and $\delta = \sqrt{N}\cdot h$. However, $N$ additional gradient calls are required to computes $G_t$.

\subsubsection{Other Matrix Updates: Pruning or Orthogonalization}

It is possible to create other matrix updates, e.g., the \textit{iterates only} (stores the last forward estimate and previous iterates) or \textit{greedy} (stores all previous forward estimates \textit{and} iterates) strategies, as detailed below:
\begin{align}
\begin{cases}
    Y_{t} &= [x_{t+\frac{1}{2}}, x_t, x_{t-1},\ldots, x_{t-N+2}], \nonumber\\
    Z_{t} & = [x_t, x_{t-1},x_{t-2}\ldots, x_{t-N+1}] 
\end{cases}\tag{Iterates only}\\
\begin{cases}
    Y_{t} &= [x_{t+\frac{1}{2}}, x_{t},x_{t-\frac{1}{2}},\ldots, x_{t-\frac{N+2}{2}}],\nonumber\\
    Z_{t} & = [x_{t},x_{t-\frac{1}{2}},x_{t-1},\ldots, x_{t-\frac{N+1}{2}}] 
\end{cases}\tag{Greedy}
\end{align}
where, this time, $x_{t+\frac{1}{2}}\defas x_t - h\nabla f(x_t)$, with $h$ small. However, the directions in $D_t$ are not orthogonal, hence,  $\kappa$ in \cref{assump:bounded_conditionning} may be large \cite{tyrtyshnikov1994bad,scieur2016regularized}. Nevertheless, the condition number can be controlled via pruning or orthogonalization.

\paragraph{Pruning.} It is sufficient to verify the condition number of $D_t$ and then prune the columns of $Y_t,\, Z_t,\, D_t,\,\text{and } G_t$ until $\kappa_{D_t}$ is sufficiently small, for instance, until $\kappa_{D_t}\leq 10^3$. The advantage is that this method requires only one extra gradient computation $\nabla f(x_{t+\frac{1}{2}})$ to construct $G_t$.

\paragraph{Orthogonalization} From the matrices $Y_t,\, Z_t$, the matrix $D_t$ is computed as $D_t = \texttt{qr}(Z_t-Y_t)$. Subsequently, the rest of the procedure follows steps $2$ and $3$ from \cref{alg:ortho_random}. It ensure the orthogonality of $D_t$, but requires $N$ extra gradients to evaluate $G_t$.

The pruning strategy is more cost-effective than orthogonalization but sacrifices its control over the history size. The orthogonalization technique resembles the "random orthogonal directions" rule but potentially offers more relevant directions than random ones.


\begin{algorithm}[t]
\caption{Generic iterative type I method}\label{alg:type1_iterative}
\begin{algorithmic}
\Require First-order oracle $f$, initial iterate and smoothness $x_0,\,M_0$, number of iterations $T$.
\For{$t=0,\,\ldots,\, T-1$}
\State Update $Y_t$, $Z_t$, $D_t$, $G_t$, and $\varepsilon_t$ (see \cref{sec:direction_requirements}).
\State $x_{t+1}, M_{t+1} \gets \texttt{[Alg. \ref{alg:type1}]} (f,G_t,D_t,\varepsilon_t,x_{t}, \frac{M_{t}}{2})$
\EndFor
\State \Return $x_T$
\end{algorithmic}
\end{algorithm}
\begin{algorithm}[t]
\caption{Type I subroutine with backtracking line-search}\label{alg:type1}
\begin{algorithmic}[1]
\Require First-order oracle for $f$, matrices $G,\,D$, vector $\varepsilon$, iterate $x$, initial smoothness $M_0$.
\State Initialize $M \gets \frac{M_0}{2}$
\Do
    \State $M\gets 2M\;\;$ and $\;\; H \gets \frac{G^TD+D^TG}{2} + \mathrm{I}_N\frac{M\|D\|\|\varepsilon\|}{2}$
    \State $\alpha^\star \gets \min\limits_\alpha f(x) + \nabla f(x)^T\hspace{-0.5ex}D\alpha + \frac{\alpha^T\hspace{-0.5ex}H\alpha}{2} + \frac{M\|D\alpha\|^3}{6}$
    \State $x_+ \gets x+D\alpha$
\doWhile{ 
\[ \textstyle f(x_+)\geq f(x) + \nabla f(x)^TD\alpha^\star + \frac{[\alpha^\star]^TH\alpha^\star}{2}  + \frac{M\|D\alpha^\star\|^3}{6}\]} 
\State \Return $x_+$, $M$
\end{algorithmic}
\end{algorithm}
\begin{algorithm}[t]
\caption{Type II subroutine with backtracking line-search}\label{alg:type2}
\begin{algorithmic}
\State Same as \cref{alg:type1}, but minimize and check the upper bound \cref{eq:type2_bound} instead of \cref{eq:type1_bound} on lines 4 and 6.
\end{algorithmic}
\end{algorithm}

\subsection{Miscellaneous}
\paragraph{Link with qN and Anderson acceleration}  The \cref{alg:type1,alg:type2} are strongly related to known quasi-Newton methods and Anderson acceleration technique, see \cref{sec:link}.

\paragraph{Backtracking Line-Search} The smoothness parameter is replaced by $M_t$, found by backtracking \citep{nesterov2006cubic}. The parameter $M_0$ is estimated by finite-difference, see \cref{sec:backtracking_line_search}.

\section{Convergence Rates for Iterative Type I Methods} \label{sec:rate_convergence}

This section analyzes the convergence rates of the methods that use \cref{alg:type1} as a subroutine; see \cref{alg:type1_iterative}. An analysis of methods that use \cref{alg:type2} is left for future work.

\subsection{Assumptions} \label{sec:assump}

This section lists the important assumptions regarding the function $f$. Subsequent results require a bound on the radius of the sublevel set $\{x:f(x)\leq f(x_0)\}$.
\begin{assumption}\label{assump:bounded_radius}
The radius of the sub-level set \mbox{$\{x:f(x)\leq f(x_0)\}$} is bounded by $R<\infty$.
\end{assumption}
Some results require $f$ to be star convex or convex to ensure convergence toward $f(x^\star)$.
\begin{assumption}\label{assump:star_convex}
    The function $f$ is star convex if, for all $x \in\mathbb{R}^d$ and $\forall \tau \in[0,1]$,
    \[
        f((1-\tau) x + \tau x^\star) \leq (1-\tau) f(x) + \tau f(x^\star).
    \]
\end{assumption}
\begin{assumption}\label{assump:convex}
    The function $f$ is convex if, for all $y,\,z\in\mathbb{R}^d$, $f(y)\geq f(z) + \nabla f(z)(y-z)$.
\end{assumption}

\subsection{Rates of Convergence}

When $f$ satisfies \cref{assump:lipchitiz_hessian}, \cref{alg:type1_iterative} ensures a consistent minimal decrease in function at every step.
\begin{restatable}{theorem}{thmminimaldecrease}\label{thm:minimal_decrease} 
    Let $f$ satisfy \cref{assump:lipchitiz_hessian}. Subsequently, at each iteration $t\geq 0$, \cref{alg:type1_iterative} starting at $x_0$ with $M_0>0$ achieves
    \begin{align}
        & \textstyle f(x_{t+1}) \leq f(x_{t})-\frac{M_{t+1}}{12}\|x_{t+1}-x_t\|^3, \\ 
        \text{with }&  M_{t+1}<\max\left\{2L\;;\; \frac{M_0}{2^t} \right\}.\nonumber
    \end{align}\label{eq:minimal_decrease}
    Moreover, the total function evaluation is bounded by $ 2t + \log_2\left(\frac{M_0}{L}\right)$.
\end{restatable}

Under some mild assumptions, \cref{alg:type1_iterative} converges to a critical point for nonconvex functions and converges to an optimum when the function is star-convex.

\begin{restatable}{theorem}{thmratenonconvex}\label{thm:rate_nonconvex} 
    Let $f$ satisfy \cref{assump:lipchitiz_hessian} and assume that $f$ is bounded below by $f^\star$. Let \cref{assump:projector_grad,assump:bounded_epsilon,assump:bounded_conditionning} hold, and $M_t \geq M_{\min}$. Subsequently, \cref{alg:type1_iterative} starting at $x_0$ with $M_0>0$ achieves
    \begin{align*}
         \min_{i=1,\,\ldots,\, t} \|\nabla f(x_{i})\| \leq &\max\Big\{ \frac{3L}{t^{2/3}} \left(12\frac{f(x_0)-f^\star}{M_{\min}}\right)^{2/3} \; ; \\ 
         &\left(\frac{C_1}{t^{1/3}}\right)\left(12\frac{f(x_0)-f^\star}{M_{\min}}\right)^{1/3} \Big\},
    \end{align*}
    $\textstyle C_1 = \delta L\left(\frac{\kappa + 2\kappa^2}{2} \right) + \max_{i \in [0,t]} \|(I-P_i)\nabla^2 f(x_i)P_i\|.$
\end{restatable}
\begin{restatable}{theorem}{thmratestarconvex}\label{thm:rate_starconvex} 
    Let $f$ satisfy \cref{assump:lipchitiz_hessian,assump:bounded_radius,assump:star_convex}. Let \cref{assump:projector_grad,assump:bounded_epsilon,assump:bounded_conditionning} hold. Then, for $t\geq 1$, \cref{alg:type1_iterative} starting at $x_0$ with $M_0>0$ achieves
    \begin{align*}
        f(x_{t})-f^\star &\leq  6\frac{f(x_0)-f^\star}{t(t+1)(t+2)}\\
        & +\frac{1}{(t+1)(t+2)}\frac{L(3R)^3}{2}  + \frac{1}{t+2}\frac{C_2(3R)^2}{4},
    \end{align*}
    $C_2 \defas \delta L\frac{ \kappa +  2\kappa^2}{2} +\max_{i \in [0,t]} \|\nabla^2 f(x_i)- P_i\nabla^2 f(x_i)P_i\|.$
\end{restatable}

The next theorem shows that when random directions (that satisfy \cref{assump:random_projector}) are used, $f(x_t)$ also converges in expectation to $f(x^\star)$ when $f$ is convex.

\begin{restatable}{theorem}{thmraterandomconvex}\label{thm:rate_randomconvex} 
    Assume $f$ satisfy \cref{assump:lipchitiz_hessian,assump:bounded_radius,assump:convex}. Let \cref{assump:random_projector,assump:bounded_epsilon,assump:bounded_conditionning} hold. Then in expectation over the matrices $D_i$ and for $t\geq 1$, \cref{alg:type1_iterative} starting at $x_0$ with $M_0>0$ achieves 
    \begin{align*}
        \mathbb{E}_{D_t}[f(x_{t}) - f^\star] \leq & \frac{1}{1+\frac{1}{4}\left[\frac{N}{d}t\right]^3}(f(x_0)-f^\star) \\
        & + \frac{1}{\left[\frac{N}{d}t\right]^2} \frac{L(3R)^3}{2} +  \frac{1}{\left[\frac{N}{d}t\right]}\frac{C_3(3R)^2}{2},
    \end{align*}
    $C_3 \defas \textstyle \delta L \frac{\kappa +  2\kappa^2}{2} +\frac{(d-N)}{d}\max_{i\in[0,t]}\|\nabla^2 f(x_i)\|$.
\end{restatable}

\paragraph{Accelerated method} Due to space limitation, the accelerated \cref{alg:type1_iterative} is presented in \cref{sec:accelerated_algorithm_full}, see \cref{alg:subroutine_acc,alg:accelerated_algo}. Indeed, while the algorithm is theoretically intriguing, it underperforms numerically, likely because it trades-off its adaptivity for better worst-case convergence rates.

\begin{restatable}{theorem}{thmrateaccsketch}\label{thm:rate_acc_sketch} 
    Assume $f$ satisfies \cref{assump:lipchitiz_hessian,assump:bounded_radius,assump:convex}. Let \cref{assump:projector_grad,assump:bounded_epsilon,assump:bounded_conditionning} hold. Subsequently, for $t\geq 1$, the accelerated \cref{alg:accelerated_algo} starting at $x_0$ with $M_0>0$ achieves
    \begin{align*}
        f(x_t)-f^\star \leq &  C_4 \frac{(3R)^2}{(t+3)^2} + 9\max\left\{M_0\;;\; 2L\right\} \left(\frac{3R}{t+3}\right)^3\\
         & + \frac{\frac{\tilde\lambda^{(1)}R^2}{2} + \frac{\tilde \lambda^{(2)}R^{3}}{6}}{(t+1)^3}.
    \end{align*}
    \begin{align*}
        & \textstyle \tilde \lambda^{(1)} = \frac{\delta}{2}\left(L\kappa+M_1\kappa^2\right) + \|\nabla^2 f(x_0)-P_0\nabla^2 f(x_0)P_0\|, \\
        & \tilde \lambda^{(2)} = M_1+L,\\
        & C_4 = 30\kappa_D\left(\delta\max\{4L,M_0\} + \max_{i\leq t} \|(I-P_i)\nabla f(x_i)P_i)\|\right) 
    \end{align*}
\end{restatable}

\subsection{Interpretation, Comparison With First-Order Methods and Special Cases} 
\label{sec:interp_comp_special}
The rates presented in \cref{thm:rate_nonconvex,thm:rate_starconvex,thm:rate_randomconvex,thm:rate_acc_sketch} combine the rates of the cubic regularized Newton's method and gradient descent (or coordinate descent, as in \cref{thm:rate_randomconvex}) for functions with Lipschitz-continuous Hessian. As $C_1, C_2$, $C_3$, and $C_4$ decrease, the rates approach those of second-order methods. For simplicity in this section, those constants are denoted as $C$, and $\kappa$ is assumed to be equal to $1$.

\subsubsection{Interpretation and Comparison} 

The constant $C$ quantifies the estimation error of $D_t^T\nabla^2\hspace{-0.5ex} f(x_t)D_t$ by $H_t$ in \eqref{eq:type1_bound} into two terms:
\[
    \textstyle C \leq O\big(\delta L + \textstyle\max_{i\leq t}\|(I-P_i)\nabla^2 f(x_i)\|\Big).
\]
The first term is the error caused by approximating $\nabla^2 f(x)D_t$ by $G_t$ and the second is the subspace approximation error of $\nabla^2 f(x_t)$ in the span of the columns of $D_t$. This approximation is more explicit in $C_3$, where increasing $N$ reduces the constant to $N=d$.

The rate associated with the constant $C$ is a perturbed version of the rate of first order methods (\cref{sec:comparison_rate}): if the function has Lipchitz-continuous gradients with constant $\mathcal{L}$ (\cref{sec:comparison_rate}, \cref{eq:lipschitz_gradient}), in the worst case over $P_i$, the constant $C$ is bounded as $C \leq O(\delta L + \mathcal{L})$. 
Hence, the rates of \cref{thm:rate_nonconvex,thm:rate_starconvex,thm:rate_acc_sketch} are perturbed versions of the rates of gradient descent and accelerated gradient descent, respectively,  whereas \cref{thm:rate_randomconvex} is a perturbed version of the rate of coordinate descent. The perturbation is not surprising, as this study \textbf{did not} assume that $f$ has a Lipchitz continuous gradient.

\subsubsection{Special Cases} \label{sec:special_cases}

The framework presented in this study can be reduced to some known methods, such as the cubic regularization of Newton's method. This section explores some special/extreme cases of this study's framework when the update rule reads $Y_t = Z_t + hD_t$ for some $D_t$ that satisfy the requirements listed in \cref{sec:direction_requirements}. Consider the following three extreme cases.
\begin{enumerate}
    \item[a)] The stepsize $h$ tends to 0, hence $y_i \rightarrow z_i$ and
    \[
        G_t \rightarrow [\ldots, \nabla^2 f(z_i^{(t)})d_i^{(t)},\ldots]_{i=1\ldots N}.
    \]
    This requires a second-order oracle, but makes $\delta$ smaller, that is, $\delta = O(\|x_t-z_i^{(t)}\|)$.
    \item[b)] The forward estimates are all centered in $x_t$, hence $Z_t=[x_t,\ldots, x_t]$. Then, at each iteration $t$, computing $G_t$ requires $N$ additional gradients, but makes $\delta$ smaller, that is,  $\delta = O(h)$.
    \item[c)] The algorithm's memory equals $d$ (full memory version), hence $P_t = I_d$ for all $t$. This raises the complexity of each iteration to $O(d^3)$, but this makes $(I-P_t)\nabla f(x_t) = 0$.
\end{enumerate}

The combinations of these special cases are summarized in \cref{tab:special_case}.

\begin{table*}
    \centering
    \begin{tabular}{lccc}
         & Satisfies \textbf{a)} & Satisfies \textbf{b)} & Satisfies \textbf{c)} \\
         \hline 
         Limited Memory QN with guarantees (\textbf{this paper}) & & & \\
         Quasi-Newton with Second Order Oracle \citep{rodomanov2021rates,rodomanov2021new,rodomanov2021greedy} & \checkmark &  & \\
         Subspace Newton with finite-difference (\textbf{this paper}) &  & \checkmark &  \\
         Full-Memory QN with guarantees (\textbf{this paper}, \citep{jiang2023accelerated}) &  &  & \checkmark \\
         Random Subspace \citep{doikov2018randomized,gower2019rsn,hanzely2020stochastic} & \checkmark & \checkmark & \\
         Cubic Newton with finite-difference \citep{grapiglia2022cubic} &  & \checkmark & \checkmark \\
         Cubic Newton with Lazy Hessian \citep{doikov2023second} & \checkmark &  & \checkmark \\
         Cubic Newton \citep{nesterov2006cubic,nesterov2008accelerating} & \checkmark & \checkmark & \checkmark\\
         \hline 
    \end{tabular}
    \caption{By considering the extreme cases \textbf{a)}, \textbf{b)}, and \textbf{c)} from \cref{sec:special_cases}, the algorithm presented in this study reduces to known methods, up to the Cubic regularization of Newton's method in the most extreme setting.}
    \label{tab:special_case}
\end{table*}

\section{Numerical Experiments}\label{sec:numerical_experiments}
This section compares the methods generated by this study's framework to the l-BFGS algorithm from \texttt{minFunc} \citep{schmidt2005minfunc}, see \cref{fig:test}. Additional experiments are described in \cref{sec:numerical_app}. The tested methods were type-I iterative algorithms (\cref{alg:type1_iterative} using the techniques from \cref{sec:direction_requirements}). The step size of the forward estimation was set to $h=10^{-9}$, and the condition number $\kappa_{D_t}$ was maintained below $\kappa=10^9$ using iterate-only and greedy techniques. The accelerated \cref{alg:accelerated_algo} is used only with the \textit{ forward estimate-only} technique. These methods are evaluated on a logistic regression problem on the Madelon UCI dataset \citep{guyon2003design}.

Regarding the number of iterations, the greedy orthogonalized version outperformed the others owing to the orthogonality of the directions (resulting in a condition number of one) and the meaningfulness of the previous gradients/iterates. However, in terms of gradient oracle calls, the recommended method, \textit{orthogonal forward estimate only}, achieved the best performance by balancing the cost per iteration (only two gradients per iteration) and efficiency (small and orthogonal directions, reducing theoretical constants). Surprisingly, the accelerated method underperforms, likely owing to its tightened theoretical analysis reducing its inherent adaptivity.

\section{Conclusion and Future work}
This study introduces a generic framework for developing novel quasi-Newton and Anderson/nonlinear acceleration schemes that offer a global convergence rate in various scenarios, including accelerated convergence on convex functions, with minimal assumptions.

The proposed approach requires an additional gradient step for the \textit{forward estimate}, as discussed in Section \ref{sec:direction_requirements}. However, this forward estimate is crucial for enabling the algorithm's adaptivity.

Studying the special case $N=d$ (although unsuitable for large-scale problems) could highlight super-linear convergence rates in future research. Moreover, using the average-case analysis framework from existing literature \citep{pedregosa2020acceleration,scieur2020universal,domingo2020average,cunha2021only,paquette2022halting} can also improve the constants in \cref{thm:rate_starconvex,thm:rate_nonconvex} to match those in \cref{thm:rate_randomconvex}. Furthermore, exploring the convergence rates of the type-2 method is also worthwhile.

Ultimately, the results presented in this study open new avenues of research. It may also provide a potential foundation for investigating additional properties of existing quasi-Newton methods. This may even lead to the discovery of convergence rates for adaptive cubic-regularized BFGS variants.

\paragraph{Main limitation} Despite its strong theoretical convergence rates and its ability to recover cubic regularization in the limit, the algorithm's complexity rises to $O(d^3)$ when $N$ approaches $d$ owing to the cubic minimization subproblem. This contrasts with current full-memory qN methods, which have a complexity of $O(d^2)$ but offer local superlinear convergence rates \citep{rodomanov2021greedy,rodomanov2021new,rodomanov2021rates}. Solving the cubic regularized model is an active research topic \citep{carmon2018analysis,carmon2019gradient,jiang2021accelerated,gao2022approximate,jiang2022cubic}, and its study is out of the scope of this paper. Nevertheless, it is \textit{very likely} that future researchs may reduce the complexity to $O(d^2)$ by updating the solution of the subproblem with low-rank perturbations, or by solving inexactly the cubic subproblem like in \citep{jiang2023accelerated}.

\begin{figure}
    \centering
    \includegraphics[width=0.43\textwidth]{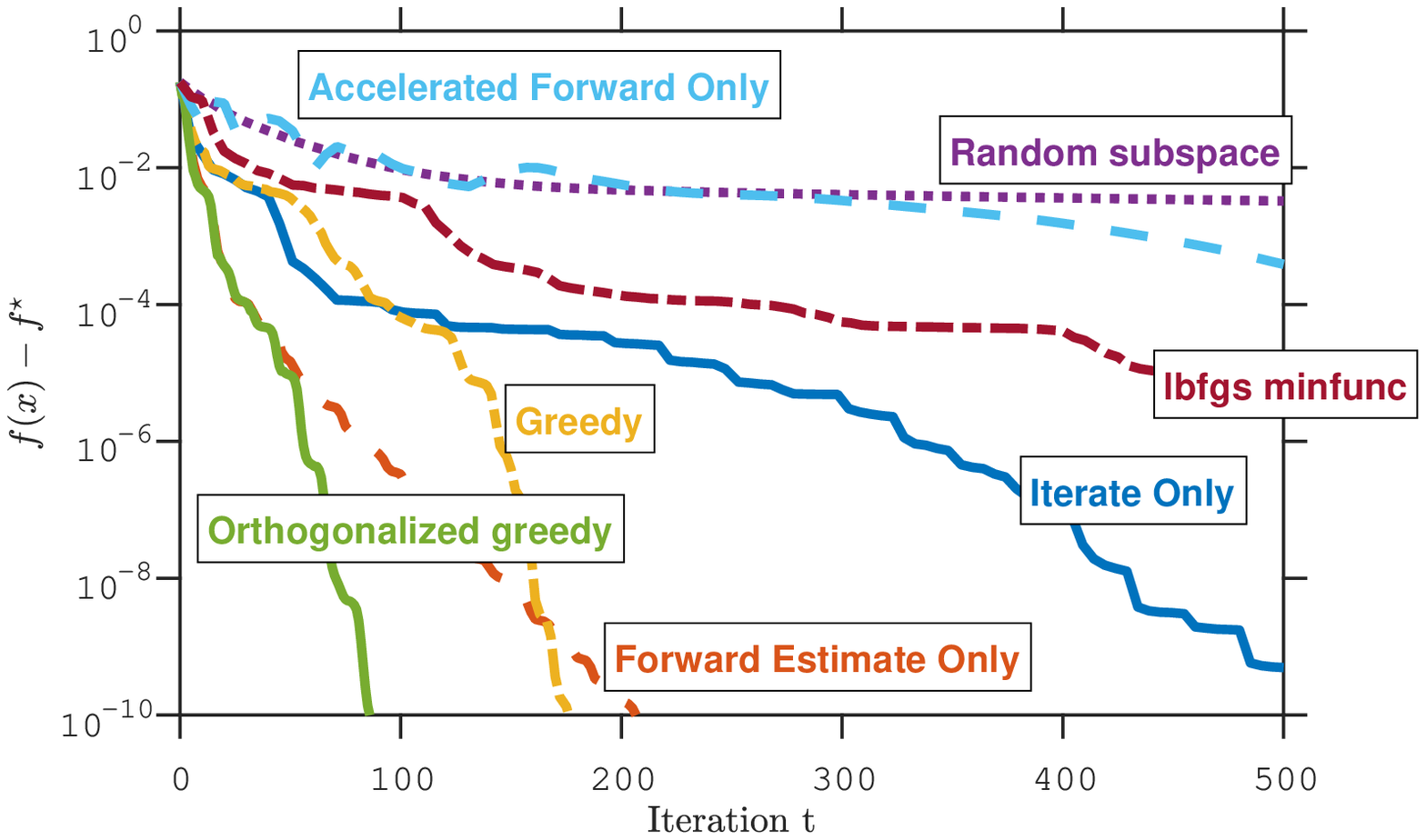}
    \includegraphics[width=0.43\textwidth]{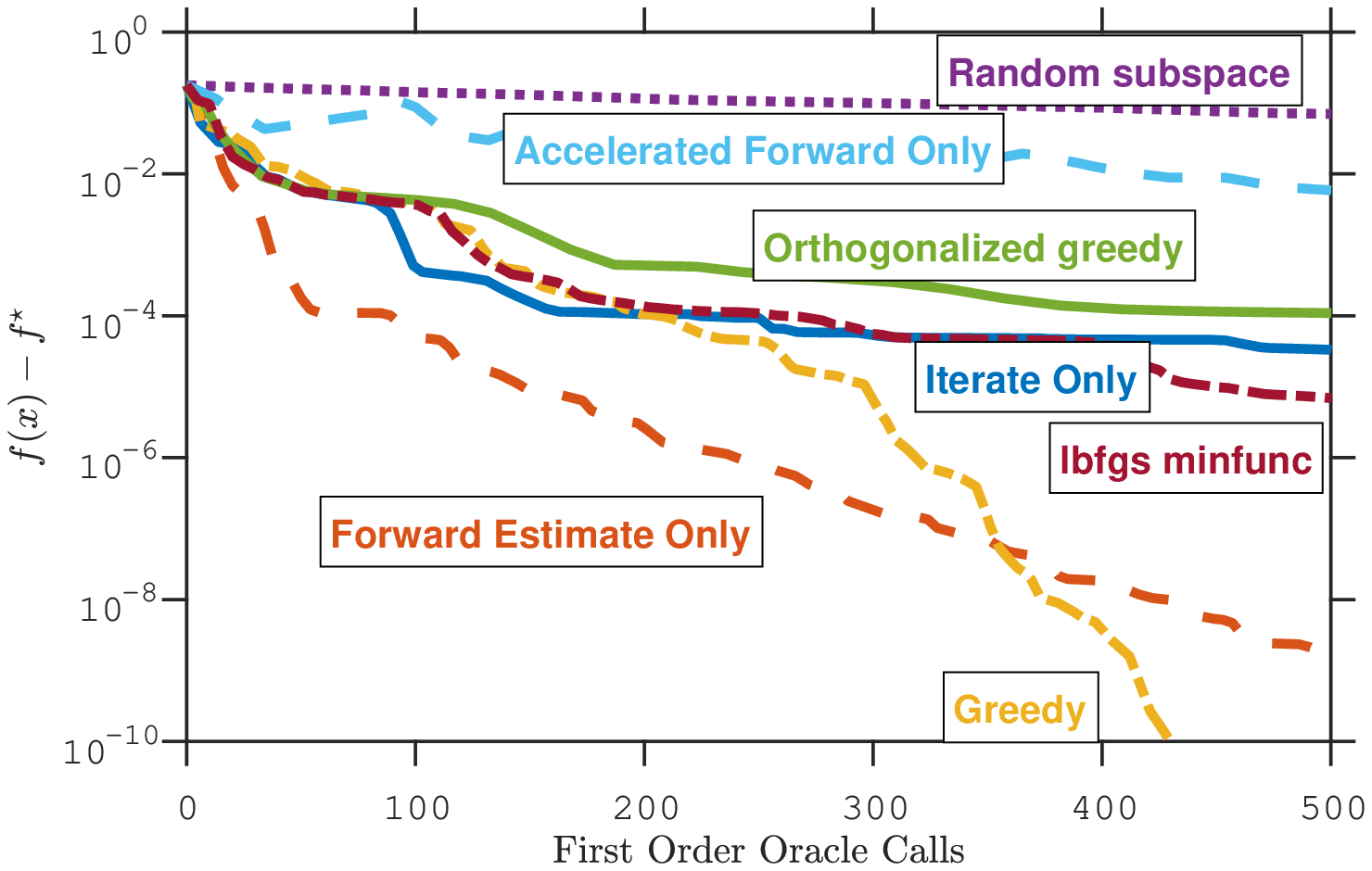}
    \caption{Comparison between the type-1 methods proposed in this study and the optimized implementation of l-BFGS from \texttt{minFunc} \citep{schmidt2005minfunc} with default parameters, except for the memory size. All methods use a memory size of $N=25$.}
    \label{fig:test}
\end{figure}

\clearpage

\clearpage
\printbibliography

\clearpage

\appendix
\onecolumn

\clearpage
\tableofcontents
\clearpage
\section*{Supplementary Materials}
\section{Accelerated Algorithm} \label{sec:accelerated_algorithm_full}

This section introduces \cref{alg:accelerated_algo}, an accelerated variant of \cref{alg:type1_iterative} for convex functions, designed using the estimate sequence technique from \citep{nesterov2008accelerating}. It consists in iteratively building a function $\Phi_t(x)$, that reads 
\[
    \textstyle \Phi_t(x) = \frac{1}{ \sum_{i=0}^{t} b_i}\left( \sum_{i=0}^{t} b_i \left(f(x_i) + \nabla f(x_i)(x-x_i)\right) + \lambda^{(1)}_{t}\frac{\|x-x_0\|^2}{2} + \lambda^{(2)}_{t}\frac{\|x-x_0\|^3}{6}\right).
\]
The parameters $b_i\geq 0$, $\lambda^{(1)}_{t}$, $\lambda_t^{(2}$ and the iterates $X_t$ are designed by theory to ensure the following properties,
\[
    B_t f(x_t) \leq \min_x \phi_t(x),\qquad \phi(x) \leq B_t f(x) + \frac{\tilde \lambda^{(1)}+\lambda^{(1)}_t}{2} \|x-x_0\|^2 + \frac{\tilde \lambda^{(2)}+\lambda^{(2)}_t}{6} \|x-x_0\|^3,
\]
where $B_t = \sum_{i=0}^t b_i$ and $\tilde \lambda^{(1)}$, $\tilde \lambda^{(2)}$ are constants determined by the theory.

Once the parameters are set, the accelerated algorithm operates as follow: \begin{enumerate}
    \item The accelerated algorithm combines linearly $v_t$, the optimum of $\Phi_t$, and the previous iterate $x_t$.
    \item It uses a slight modified version of \cref{alg:type1}, see \cref{alg:subroutine_acc}.
    \item There is a distinction between small and large step sizes, identifying which $\lambda$ needs to be updated. The step size is considered "large" if it resembles a cubic-Newton step.
\end{enumerate}

\begin{algorithm}
\caption{Type-I subroutine with backtracking for the accelerated method}\label{alg:subroutine_acc}
\begin{algorithmic}
\Require First-order oracle $f$, matrices $G,\,D$, vector $\varepsilon$, iterate $x$, initial smoothness parameter $M_0$
\State ~
\State Initialize $M \gets \frac{M_0}{2}$, $\texttt{ExitFlag} \gets \texttt{None}$
\State Define $\gamma_M \defas \frac{\kappa_D}{\|D\|}\left( \frac{3}{2}\|\varepsilon\| + 2\frac{(\|I-P)G\|}{M} \right)$
\Do{}
\State $M\gets 2\cdot M\;\;$ and $\;\; H_\gamma \gets \frac{G^TD+D^TG}{2} + D^TD\frac{M\gamma_M}{2}$
    \State $\alpha^* \gets \argmin_\alpha f(x) + \nabla f(x)^TD\alpha + \frac{1}{2} \alpha^TH_\gamma\alpha + \frac{M\|D\alpha\|^3}{6}$
    \State $x_+ \gets x+D\alpha$
    \If{ $\frac{2}{3^{3/4}} \frac{\|\nabla f(x_+)\|^{3/2}}{\sqrt{M}}\leq -\nabla f(x_+)^TD\alpha $ }
    \State $\texttt{ExitFlag} \gets \texttt{LargeStep}$
    \EndIf
    \If{ $\frac{\|\nabla f(x_+)\|^2}{M(\gamma_M+\|D\alpha\|)}\leq -\nabla f(x_+)^TD\alpha $ \texttt{And} $ \|D\alpha\| \leq (\sqrt{3}-1)\gamma_M $}
    \State $\texttt{ExitFlag} \gets \texttt{SmallStep}$
    \EndIf
\doWhile{\texttt{ExitFlag} is \texttt{None}}
\State ~
\State \Return $x_+$, $\alpha$, $M$, $\gamma_M$, \texttt{ExitFlag}
\end{algorithmic}
\end{algorithm}

\begin{algorithm}
\caption{Adaptive Accelerated Type-I Iterative Algorithm}\label{alg:accelerated_algo}
\begin{algorithmic}
\Require First-order oracle $f$, initial iterate and smoothness $x_0,\,M_0$, number of iterations $T$.
\State $\lambda^{(1)}_0\gets 0$, $\lambda^{(2)}_0\gets 0$
\State Initialize $G_0$, $D_0$, $\varepsilon_0$ (see \cref{sec:direction_requirements})
\State $\{x_{1},\, M_1\} \gets \texttt{[\cref{alg:type1}]} (f,G_0,D_0,\varepsilon_0,x_{0}, M_0)$
\State Initialize $\ell_0^{(0)} = f(x_1)$, $\quad \ell_0^{(1)} = 0$
\State
\For{$t=1,\,\ldots,\, T-1$}
    \State Update $G_t$, $D_t$, $\varepsilon_t$ (see \cref{sec:direction_requirements})
    \State Set $b_t \gets \frac{(t+1)(t+2)}{2}$, $B_t \gets \frac{t(t+1)(t+2)}{6}$, $\beta_t \gets \frac{3}{t+3}$.
    \State Update $\ell_t^{(0)} \gets \ell_{t-1}^{(0)} + b_{t-1} [f(x_t) - \nabla f(x_t)^T x_t]$, $\quad\ell_t^{(1)} \gets \ell_{t-1}^{(1)} + b_{t-1}\nabla f(x_t)$
    \State
    \Do
        \State $\texttt{ValidBound} \gets \texttt{True}$
        \State Set $v_t \gets \argmin_v \phi_t(v)$ (See \cref{prop:solution_minimizer}).
        \State Let $y_t \gets \frac{3}{t+3} v_t + \frac{t}{t+3} x_t$
        \[
        \{x_{t+1}, \alpha_{t},\, M_{t+1},\, \gamma_t, \texttt{ExitFlag}\} \gets \texttt{[Alg.\ref{alg:subroutine_acc}]} (f,G_t,D_t,\varepsilon_t,y_{t}, \frac{M_t}{2})
        \]
        \State
        \State \texttt{\%\% Check if the next $\phi$ is still a lower bound for $B_tf(x_{t+1})$}
        \State Define $\phi_{+}(x) = \phi_{t}(x) + b_{t}[f(x_{t+1} + \nabla f(x_{t+1})(x-x_{t+1})]$.
        \State Set $v_{+} \gets \argmin_v \phi_{+}(v)$ (See \cref{prop:solution_minimizer}).
        \State
        \If{$\Phi_{+}(v_{+}) \leq B_t f(x_{t+1})$}  $\quad$ \texttt{\%\% Parameters adjustment if needed}
            \State $\texttt{ValidBound} \gets \texttt{False}\quad$  \texttt{\%\% Unsuccessful iteration:} $\phi_{t+1}(v_{t+1})\geq f(x_{t+1})$.
            \If{\texttt{ExitFlag} is \texttt{LargeStep}}
                \State \textbf{If } $\lambda_t^{(2)} =0$ \textbf{ then } $\displaystyle \lambda_t^{(2)}\gets \frac{4}{\sqrt{3}}\frac{b_{t+1}^{3}}{B^2_t}M_{t+1}$. \textbf{Else, }$\lambda_t^{(2)} \gets 2\lambda_t^{(2)}$.
            \Else \texttt{ \%\% Exitflag is SmallStep}
                \State \textbf{If} $\lambda_t^{(1)} =0$ \textbf{ then } $\displaystyle\frac{b_{t+1}^2}{B_t}M_{t+1}\left(\gamma_t+\|D_t\alpha_t\|\right)$. \textbf{Else, } $\lambda_t^{(1)} \gets 2\lambda_t^{(1)}$.
            \EndIf
        \Else 
            \State $\{\lambda_{t+1}^{(1)} ,\lambda_{t+1}^{(2)} \} \gets \{\lambda_{t}^{(1)} ,\lambda_{t}^{(2)} \}$ \quad \texttt{\%\% Successful iteration}
        \EndIf
    \doWhile{\texttt{ValidBound is False}}
\EndFor
\State \Return $x_T$
\end{algorithmic}
\end{algorithm}

\begin{proposition} \label{prop:solution_minimizer}
    Let $v_t$ be the minimizer of 
    \[
        \phi_t(v) = \ell_t^{(0)} + \left[\ell_t^{(1)}\right]^Tv + \frac{\lambda_t^{(1)}}{2} \|v-x_0\|^2 + \frac{\lambda_t^{(2)}}{6}\|v-x_0\|^3.
    \]
    where $\lambda_t^{(1)} \geq 0 $, $\lambda_t^{(2)}\geq 0$. Let $r_t = \|v_t-x_0\|$. Then,
    \begin{align*}
        r_t& = \textstyle \|v_t-x_0\| =\begin{cases}
        0 & \texttt{if } \lambda_t^{(1)} = \lambda_t^{(2)} = 0\\
        \frac{\|\ell_t^{(1)}\|}{\lambda_t^{(1)}} & \texttt{if }  \lambda_t^{(1)}>0 \texttt{ and }\lambda_t^{(2)} = 0\\
        \frac{-\lambda^{(1)}_t + \sqrt{[\lambda^{(1)}_t]^2+2\lambda^{(2)}_t\|\ell_k\|}}{\lambda^{(2)}_2} & \texttt{if }\lambda_t^{(2)} >0\\
    \end{cases}\\
    v_t & \textstyle = \argmin \Phi_t(x) = x_0-r_t\frac{\ell_t^{(1)}}{\|\ell_t^{(1)}\|}
    \end{align*}
\end{proposition}

\clearpage

\section{Related work} \label{sec:related_work}

\subsection{Inexact, Subspace, and Stochastic Newton Methods} 

Instead of explicitly computing the Hessian matrix and the Newton step, inexact methods compute an approximation using sampling \citep{antonakopoulos2022extra}, inexact Hessian computation \citep{ghadimi2017second,doikov2022second}, or random subspaces \citep{doikov2018randomized,gower2019rsn,hanzely2020stochastic}. These approaches substantially reduce per-iteration costs without significantly compromising the convergence rate. The convergence speed in such cases often represents an interpolation between the rates observed in gradient descent methods and (cubic) Newton's method.

\subsection{Nonlinear and Anderson Acceleration} 

Nonlinear acceleration techniques, including Anderson acceleration \citep{anderson1965iterative}, have a long standing history \citep{brezinski1970application,brezinski1971algorithme,gekeler1972solution}. Driven by their promising empirical performance, they recently gained interest in their convergence analysis \citep{sidi1988extrapolation,ford1988recursive,sidi1991efficient,jbilou1991some,sidi1998upper,sidi2008vector,walker2011anderson,toth2015convergence,scieur2016regularized,sidi2017minimal,sidi2017vector,brezinski2018shanks,scieur2018online,brezinski2020shanks,scieur2020regularized}. In essence, Anderson acceleration is an optimization technique that enhances convergence by extrapolating a sequence of iterates using a combination of previous gradients and corresponding iterates. Comprehensive reviews and analyses of these techniques can be found in notable sources such as \citep{jbilou1991some,brezinski1991extrapolation,jbilou1995analysis,jbilou2000vector,brezinski2019genesis,d2021acceleration}. However, these methods do not generalize well outside quadratic minimization and their convergence rate can only be guaranteed asymptotically when using a line-search or regularization techniques \citep{sidi1986convergence,sidi1988convergence,scieur2016regularized}.

\subsection{Quasi-Newton Methods} 

Quasi-Newton schemes are renowned for their exceptional efficiency in continuous optimization. These methods replace the exact Hessian matrix (or its inverse) in Newton's step with an approximation updated iteratively during the method's execution. The most widely used algorithms in this category include DFP \citep{davidon1991variable,fletcher1963rapidly} and BFGS \citep{shanno1970conditioning,goldfarb1970family,fletcher1970new,broyden1970convergence,broyden1970convergenceb}. Most of the existing convergence results predominantly focus on the asymptotic super-linear rate of convergence \citep{stachurski1981superlinear,griewank1982local,byrd1987global,byrd1989tool,conn1991convergence,engels1991local,yabe1996local,wei2004superlinear,yabe2007local}. However, recent research on quasi-Newton updates has unveiled explicit and non-asymptotic rates of convergence \cite{rodomanov2021greedy,rodomanov2021rates,rodomanov2021new,lin2022explicit,lin2021greedy}. Nonetheless, these analyses suffer from several significant drawbacks, such as assuming an infinite memory size and/or requiring access to the Hessian matrix. These limitations fundamentally undermine the essence of quasi-Newton methods, typically designed to be Hessian-free and maintain low per-iteration cost through their low memory requirement and low-rank structure. 

\subsection{Close Related Work}

\subsubsection{(Accelerated) Quasi-Newton with Secant Inexactness}

Recently, \citet{kamzolov2023accelerated} introduced an adaptive regularization technique combined with cubic regularization, with global, explicit (accelerated) convergence rates for any quasi-Newton method. Based on the secant inexactness inequality, the technique introduces a quadratic regularization whose parameter is found by a backtracking line search. However, this algorithm relies on prior knowledge of the Lipschitz constant specified in \cref{assump:lipchitiz_hessian}. Unfortunately, the paper does not provide an adaptive method to find jointly the Lipschitz constant as well, as it is \textit{a priory} too costly to know which parameter to update. This aspect makes the method impractical in real-world scenarios. 

\subsubsection{ARC: Adaptive Regularization algorithm using Cubics}

In \citep{cartis2011adaptive,cartis2011adaptive2} is proposed a generic framework for inexact cubic regularized Newton's steps,
\[
    x_{t+1} = \min_x f(x_t) + \nabla f(x_t) (x-x_t) + \frac{1}{2}(x-x_t)H_t(x-x_t) + \frac{M_t}{6}\|x-x_t\|^3,
\]
where $H_t$ is assumed to be an approximation of the Hessian $\nabla^2 f(x_t)$. However, the theoretical analysis presents numerous problems, in particular, the assumption that the norm of the current step bounds the approximation
\[
    \| \nabla^2 f(x_t)-H_t\| \leq C\|x_{t+1}-x_{t}\|,
\]
for some constant $C$. Follow up works, such as \citep{wang2018note}, relaxed this assumption into
\[
    \| \nabla^2 f(x_t)-H_t\| \leq C\|x_{t}-x_{t-1}\|,
\]
which is much weaker since it can be verified while computing the step $x_{t+1}$. Nevertheless, those are assumptions on the matrix $H_t$, but those works do not explicitly construct such a matrix. Even worse - the assumption might not be met in practice, especially if $H_t$ is a subspace estimation of the matrix $\nabla^2 f(x_t)$.

\subsubsection{Proximal Quasi-Newton Methods}

The work of \citep{scheinberg2016practical,ghanbari2018proximal} combined qN methods with proximal schemes and provided sublinear and accelerated convergence rates. However, the rates in \citep{scheinberg2016practical} are based on a technical assumption \citep[Assumption 2]{scheinberg2016practical}, for which the authors commented that "\textit{Exploring different conditions on the Hessian approximations that ensure Assumption 2 is a subject of a separate study}", and acknowledge in their conclusion that "\textit{Our framework does not rely on or exploit the accuracy of second-order information, and hence we do not obtain fast local convergence rates.}" 

In a follow-up work, \citep{ghanbari2018proximal} proposed accelerated convergence rates under similar assumptions. However, the authors acknowledge the following: "\textit{In our numerical results, we construct $H_k$ via L-BFGS and ignore condition $\sigma_{k+1}H_{k+1}\preceq \sigma_k H_k$, since enforcing it in this case causes a very rapid decrease in $\sigma$. It is unclear, however, if a practical version of Algorithm 5, based on L-BFGS Hessian approximation, can be derived, which may explain why the accelerated version of our algorithm does not represent any significant advantage.}" In addition, their theoretical convergence results are based on an upper bound on the sequence $\sigma_k$, which current qN schemes cannot ensure.

\subsubsection{Proximal Extragradient Quasi-Newton Methods with Online Estimation}

Based on the technique in \citep{jiang2023online}, \citep{jiang2023accelerated} developed a novel quasi-Newton method with the global accelerated rate of convergence of $O(\min\{\frac{1}{t^2};\frac{\sqrt{d\log t}}{t^{2.5}}\})$. The main ideas are as follows: the authors used the framework of inexact proximal method from \citep{monteiro2013accelerated}, used an online algorithm to estimate the Hessian, and then solved a linear system involving this approximation using conjugate gradients. 

The paper focuses on a different regime than this study: \citep{jiang2023accelerated} explicitly show that it is possible to break the $O(\frac{1}{t^2})$ barrier for first order methods using full memory qN methods but this implies storing a full $d\times d$ matrix, and using it in a linear system, leading to per-iteration complexities of at least $O(d^2)$. 

From a practical point of view, the algorithm requires numerous hyperparameters such as $\alpha_1$, $\alpha_2$, $\beta$,$\ldots$, whose impact on the efficiency is rather unclear. Moreover, numerically, the algorithm improves over Nesterov's acceleration but is slower than l-BFGS on toy experiments.
\clearpage
\section{Known rates of convergence and Comparison} \label{sec:comparison_rate}

\subsection{(Accelerated) Gradient Descent}

This section study the rate of gradient decent when function is smooth (i.e., has Lipschitz continuous gradients): 
\begin{equation}
    \label{eq:lipschitz_gradient}
    f(y) \leq f(x) + \nabla f(x)(y-x) + \frac{\mathcal{L}}{2}\|y-x\|^2,
\end{equation}
Note that the class of functions considered in this paper is \textit{not} the class of smooth functions. However, if the function satisfies \cref{assump:lipchitiz_hessian}, the Lipchits constant can be bounded as
\begin{equation}
    \mathcal{L} \leq \|\nabla^2 f(x) \| + L R \qquad \text{for all $x \in \{x:f(x)\leq f(x_0)\}$.} \label{eq:estimation_smoothness}
\end{equation}
The rates of plain gradient descent and its accelerated version read \citep{nesterov2004introductory} (after replacing $\mathcal{L}$)
\begin{align}
    \min_{0\leq i \leq t} \|\nabla f(x_i)\| & \leq \sqrt{\frac{\left[\|\nabla^2 f(x) \| + L R\right](f(x_0)-f^\star)}{t+1}}, & \text{(plain, non-convex)} \label{eq:rate_gradient_nonconvex}\\
    f(x_t)-f(x^\star) & \leq \left[\|\nabla^2 f(x) \| + L R\right] \frac{2}{t+4}R^2,  & \text{(plain, convex)} \label{eq:rate_gradient_convex}\\
    f(x_t)-f(x^\star) & \leq \left[\|\nabla^2 f(x) \| + L R\right] \frac{4}{(t+2)^2}R^2.  & \text{(accelerated)} \label{eq:rate_gradient_accelerated}
\end{align}

\subsection{(Accelerated) Cubic Regularized Newton's Method} 

When the function has a Lipschitz-continuous Hessian, the cubic regularized Newton method and its accelerated version converge with the following rates \citep{nesterov2006cubic,nesterov2008accelerating,hanzely2020stochastic}:
\begin{align}
    \min_{0\leq i \leq t} \|\nabla f(x_i)\| & \leq \frac{16L}{9} \left(\frac{3(f(x_0)-f^\star)}{2tM_{\min}}\right)^{2/3}, & \text{(plain, non-convex)} \label{eq:rate_cubic_nonconvex}\\
    f(x_t)-f(x^\star) & \leq  \quad 9L\frac{R^3}{(t+4)^2}, &\text{(plain, convex)} \label{eq:rate_cubic_convex}\\
    \mathbb{E}[f(x_t)]-f(x^\star) & \leq \left(\frac{d-N}{N}\right) \frac{\mathcal{L}(3R)^2}{2t} + \left(\frac{d}{N}\right)^2 \frac{L(3R)^3}{3t^2} + O\left(\frac{1}{t^3}\right), &\text{(Random Subspace, convex)} \label{eq:rate_cubic_stoch}\\
    f(x_t)-f(x^\star) & \leq L\frac{14R^3}{t(t+1)(t+2)}. &\text{(accelerated)} \label{eq:rate_cubic_accelerated}
\end{align}

\subsection{Relation Between Parameters}

Given that this paper does not make the assumption of Lipchitz-continuous gradients, it becomes necessary to establish connections between various quantities to facilitate the comparison of rates. To streamline the notation, all numeric constants are substituted with the big $O$ notation, and the subsequent equations are derived for the "orthogonal forward estimate only" update rule, hence $\|D\|=1$ and $\kappa=1$.

\paragraph{Relation between $\delta$ and $R$.} The constant $\delta$ represents the upper bound on the relative error (see \cref{assump:bounded_epsilon}):
\[
    \forall t,\;\; \frac{\|\varepsilon_t\|}{\|D_t\|} \leq \delta.
\]
For a fixed memory, and assuming $h$ small, since $\varepsilon$ is the norm between iterates, $\delta$ is upper-bounded as 
\begin{equation}\label{eq:approx_delta}
    \delta \leq O(R).
\end{equation}

\paragraph{Relation between the different $C_i$ and $\mathcal{L}$} The $C_1, C_2$, and $C_4$ in \cref{thm:rate_nonconvex,thm:rate_starconvex,thm:rate_acc_sketch} quantifies the estimation error of $D_t^T\nabla^2\hspace{-0.5ex} f(x_t)D_t$ by $H_t$ in \eqref{eq:type1_bound} into two terms:
\[
    C_i \leq O\big(\delta L + \textstyle\max_{i\leq t}\|(I-P_i)\nabla^2 f(x_i)\|\Big).
\]
The first term is the error caused by approximating $\nabla^2 f(x)D_t$ by $G_t$, and the second is the subspace approximation error of $\nabla^2 f(x_t)$ in the span of the columns of $D_t$. 

Intuitively, the constants $C_i$ can be seen as an approximation of an upper bound on $\mathcal{L}$ in a neighborhood of size $\delta$. This is similar to \cref{eq:estimation_smoothness} but the norm of the Hessian is taken in a subspace, hence the $C_i$'s are smaller. Indeed, using \cref{eq:approx_delta}, in the worst case, if all iterates satisfies $\|x_i-x^\star\|<R$, 
\begin{equation}\label{eq:approx_c}
    C_i = O(RL + \textstyle\max_{i\leq t}\|(I-P_i)\nabla^2 f(x_i)\|).
\end{equation}

\paragraph{Other updates} Note that \cref{eq:approx_delta,eq:approx_c} are valid only for the \textit{"orthogonal forward estimate only"} update rule. If the random orthogonal forward estimate, or the orthogonalization of the "greedy" or "iterates only" update rules were used, the results would have been
\[
    \delta = O(h),\qquad C_i = O(hL + \textstyle\max_{i\leq t}\|(I-P_i)\nabla^2 f(x_i)\|),
\]
where $h$ is small. However, the comparison with gradient descent or Newton's method wouldn't have been fair as the orthogonalization update rules requires $N$ additional gradient calls.

\subsection{Comparing rates of convergence}

\paragraph{Non convex}

The rate from \cref{thm:rate_nonconvex} reads
 \begin{align*}
     \min_{i=1,\,\ldots,\, t} \|\nabla f(x_{i})\| \leq &\max\Big\{ \frac{3L}{t^{2/3}} \left(12\frac{f(x_0)-f^\star}{M_{\min}}\right)^{2/3} \;;\; \left(\frac{C_1}{t^{1/3}}\right)\left(12\frac{f(x_0)-f^\star}{M_{\min}}\right)^{1/3} \Big\},
\end{align*}
where $\textstyle C_1 = \frac{3\delta L}{2} + \max_{i \in [0,t]} \|(I-P_i)\nabla^2 f(x_i)P_i\|.$ In the case where $C_1$ is small, the rate matches exactly \cref{eq:rate_cubic_nonconvex}. In the other case, using the approximation from \cref{eq:approx_c},
\begin{align*}
    \min_{i=1,\,\ldots,\, t} \|\nabla f(x_{i})\| \leq \left(\frac{O(RL + \textstyle\max_{i\leq t}\|(I-P_i)\nabla^2 f(x_i)\|)}{t^{1/3}}\right)\left(12\frac{f(x_0)-f^\star}{M_{\min}}\right)^{1/3} 
\end{align*}
which differs significantly from \cref{eq:rate_gradient_nonconvex}, as the rate is $O(\frac{1}{\sqrt{t}})$. However, this might be an artifact of the theoretical analysis, since the function was not assumed to be smooth.

\paragraph{Star convex}

After using the approximation from \cref{eq:approx_c}, the rate from \cref{thm:rate_starconvex} reads
\begin{align}
    & f(x_{t})-f^\star \leq O\left(\frac{f(x_0)-f^\star}{t^3}\right)+O\left(\frac{LR^3}{t^2}\right) + O\left(\frac{[RL+\textstyle\max_{i\leq t}\|(I-P_i)\nabla^2 f(x_i)\|]R^2}{t}\right)
\end{align}
The term in $t^{-2}$ is \textit{exactly} the one from \cref{eq:rate_cubic_convex}, while the term is $t^{-1}$ has the same dependency in $R^3$ compared to \cref{eq:rate_gradient_convex}. However, $\|(I-P)\nabla^2 f(x_i)\|$ could be much smaller than $\|\nabla^2 f(x)\|$.

\paragraph{Convex with random coordinates or random subspace}

The rate from \cref{thm:rate_randomconvex} reads
\begin{align*}
    \mathbb{E}_{D_t}[f(x_{t}) - f^\star] \leq \frac{1}{1+\frac{1}{4}\left[\frac{N}{d}t\right]^3}(f(x_0)-f^\star)  + \frac{1}{\left[\frac{N}{d}t\right]^2} \frac{L(3R)^3}{2} +  \frac{1}{\left[\frac{N}{d}t\right]}\frac{[O(\delta L)  +\frac{(d-N)}{d}\max\limits_{i\in[0,t]}\|\nabla^2 f(x_i)\|](3R)^2}{2}.
\end{align*}
The rate is similar to \cref{eq:rate_cubic_stoch}, up to an additional $O(\delta L/t)$ term. This extra term comes from the estimation of the Hessian with finite difference, while the method presented in \citep{hanzely2020stochastic} uses exact Hessian-vector products.

\paragraph{Convex, accelerated rates}

After using the approximation from \cref{eq:approx_c}, and ignoring the terms $\tilde \lambda^{(1)}$, $\tilde \lambda^{(2)}$ for clarity, the rate from \cref{thm:rate_acc_sketch} reads 
\begin{align*}
    f(x_t)-f^\star \leq [RL+\max_{i=0\ldots t}\|(I-P_i)\nabla f(x_i)\|] \frac{(3R)^2}{(t+3)^2} + 9\max\left\{M_0\;;\; 2L\right\} \left(\frac{3R}{t+3}\right)^3 
\end{align*}
The rate is exactly a combination of \cref{eq:rate_cubic_accelerated} and \cref{eq:rate_gradient_accelerated}, but the constant ascociated to the $1/t^2$ rate is smaller in practice: \cref{eq:approx_delta} is a conservative bound and $\|(I-P_i)\nabla^2 f(x)\|\leq \|\nabla^2 f(x)\|$.

\clearpage

\section{Link with quasi-Newton and Anderson/Nonlinear Acceleration}\label{sec:link_existing_methods}

This section presents the fundamentals of Anderson/nonlinear acceleration (\cref{sec:anderson}), quasi-Newton schemes (\cref{sec:qn}), and their relationship with the method proposed in this paper (\cref{sec:link}).

\subsection{Anderson Acceleration and Nonlinear Acceleration} \label{sec:anderson}

Anderson acceleration, also known as nonlinear acceleration, is a powerful technique that enhances the convergence speed of fixed point iterations and optimization algorithms. Initially developed for solving linear systems, Anderson acceleration has gained popularity due to its effectiveness in accelerating iterative methods, including the ones in optimization. The method leverages previous iterations to construct an improved estimate of the objective function's minimizer.

The Anderson acceleration algorithm employs the following approximation to compute weights:
\[
    \nabla f\left(\sum_{i=0}^N \beta_i x_i\right) \approx 
     \sum_{i=0}^N\beta_i \nabla f(x_i), \;\; \sum_{i=0}^N \beta_i = 1.
\]
When the function $f$ is quadratic, this approximation becomes an equality. The underlying idea is as follows: since the optimum satisfies $\nabla f(x^\star) = 0$, 
\[
     \sum_{i=0}^N\beta_i \nabla f(x_i) \approx 0 \;\; \Rightarrow \nabla f\left(\sum_{i=0}^N \beta_i x_i\right)\approx 0 \quad \Rightarrow \sum_{i=0}^N \beta_i x_i \approx x^\star.
\]
The Anderson acceleration steps are thus given by
\[
    x_{t+1} = \sum_{i=0}^N \beta_i^\star x_{t-i+1},\quad \beta^\star = \argmin_\beta\|\sum_{i=0}^N\beta_i \nabla f( x_{t-i+1})\|^2
\]

Over the past decades, the ideas behind Anderson acceleration have been refined. For example, the constraint can be eliminated by considering the step $x_{t+1}-x_t$ instead:
\begin{align*}
x_{t+1}-x_t & = \left(\sum_{i=0}^N \beta_i x_{t-i+1}\right) - x_t\\
& = \sum_{i=0}^N \tilde \beta_i x_{t-i+1}.
\end{align*}
The vector $\tilde \beta_i$ has the property that its sum equals zero. Hence, it can be rewritten as
\begin{align*}
x_{t+1}-x_t & = \sum_{i=1}^N \alpha_i (x_{t-i+1} - x_{t-i}) \\
\alpha & = \argmin_\alpha \left\|\nabla f(x_t)+\sum_{i=1}^N \alpha_i (\nabla f(x_{t-i+1})-\nabla f(x_{t-i}))\right\|
\end{align*}
where $\alpha\in\mathbb{R}^N$ has no constraint. By writing $d_i = x_{t-i+1}-x_{t-i}$, $g_i =\nabla f(x_{t-i+1})-\nabla f(x_{t-i}) $, and $D=[d_{t},\ldots, d_{t-N+1}], \, G = [g_{t},\,\ldots,\, g_{t-N+1}]$, the step becomes
\begin{align*}
x_{t+1}-x_t & = D_t\alpha, \quad \alpha = \argmin_\alpha \|\nabla f(x_t)+G_t\alpha\|.
\end{align*}
However, this version of Anderson acceleration is non-convergent because there is no contribution from $\nabla f(x_t)$ in the step $x_{t+1}-x_t$. The most popular solution to this problem is introducing a \textit{mixing parameter} that combines gradient steps, resulting in the following expression:
\begin{align}
x_{t+1} & = x_t - h\nabla f(x_t) + (D-hG)\alpha, \quad \alpha = \argmin_\alpha \|\nabla f(x_t)+G\alpha\|. \label{eq:anderson_acc} \tag{AA Type II}
\end{align}

Following a similar idea, recent works have introduced a type I variant of the algorithm \citep{fang2009two,walker2011anderson,zhang2020globally,canini2022direct} that minimizes the function value instead of the gradient norm:
\begin{equation}
x_{t+1} = x_t - h\nabla f(x_t) + (D-hG)\alpha,\quad \alpha = \argmin f(x_t) + \nabla f(x_t)D_t\alpha + \frac{1}{2}\alpha^TD_t^TG_t\alpha, \label{eq:anderson_type_1} \tag{AA Type I}
\end{equation}

By incorporating regularization \citep{scieur2016regularized,canini2022direct}, globalization techniques \citep{zhang2020globally}, or performing a line search on the parameter $h$, the algorithm converges towards $x^\star$.

\subsection{Single-secant and Multisecant Quasi-Newton Methods}  \label{sec:qn}
Quasi-Newton methods, such as the Broyden-Fletcher-Goldfarb-Shanno (BFGS) method, approximate the Hessian matrix to solve unconstrained optimization problems efficiently. These methods avoid the expensive computation of the exact Hessian by using iterative updates based on previous iterates and gradients of the objective function.

This section focuses on other commonly used quasi-Newton methods: the Davidon-Fletcher-Powell (DFP) and Broyden type-1 and type-2 updates.

\subsubsection{The Ideas Behind Single-Secant and Multisecant Hessian Approximation}

In quasi-Newton methods, the Hessian approximation is updated using the \textit{secant equation}, which relates the gradients and Hessian at two different points. For a twice continuously differentiable function, the secant equation is given by:
\[
\nabla f(y) - \nabla f(x) = \nabla^2 f(\xi)(y - x),
\]
where $\xi$ is a point on the line segment connecting $x$ and $y$. This equation serves as the basis for updating the Hessian approximation.

Based on this remarkable identity, quasi-Newton methods update an approximation of the Hessian $B_t$ or its inverse $H_t$ such that the approximation satisfies
\[
    \nabla f(x_{t}) - \nabla f(x_{t-1}) = B_t(x_{t}-x_{t-1}),\quad H_t\left(\nabla f(x_{t}) - \nabla f(x_{t-1})\right) = x_{t}-x_{t-1}.
\]
What distinguishes the different updates is how to fix the remaining degrees of freedom. For instance, the simple SR-1 method updates $H_t$ such that
\begin{equation}\label{eq:qn_update_single}
    \min_{H}\|H-H_{t-1}\|_F \quad : H=H^T,\;\; H\left(\nabla f(x_{t}) - \nabla f(x_{t-1})\right) = x_{t}-x_{t-1}.
\end{equation}
Those methods are called \textit{single-secant} as they update $H_t$ only one secant equation at a time. Hence, in general, $H_t$ only satisfies the latest secant equation.

Multisecant updates, on the other hand, approximate the Hessian using a batch of secant equations. By introducing matrices $D_t = [x_{t-N+1}-x_{t-N}, \ldots, x_{t}-x_{t-1}]$ and $G_t = [\nabla f(x_{t-N+1})-\nabla f(x_{t-N}), \ldots, \nabla f(x_{t})-\nabla f(x_{t-1})]$, the multisecant updates satisfy
\[
G_t = B_t D_t, \quad \text{or}\quad  H_t G_t = D_t.
\] 
Unfortunately, when imposing symmetry, it is impossible to satisfy multiple secants at a time \citep{schnabel1983quasi}. However,it is possible to enforce symmetry while approximating the secant equation in a least square sense \citep{scieur2019generalized,scieur2021generalization}.

When symmetry is not imposed, the solution for $B_t$ and $H_t$ can be obtained as:
\begin{equation} \label{eq:multisecant_qn_formula}
    B_t = G_t [D_t]^\dagger + B_0(I - D_tD_t^\dagger), \quad H_t = D_t [G_t]^\dagger + H_0(I - G_tG_t^\dagger),    
\end{equation}

where $B_0$ and $H_0$ are the initial approximations, and $[A]^\dagger$ denotes the pseudo-inverse of matrix $A$. Different choices of pseudo-inverse lead to different methods.

The inversion of $B_t$ can be computed using the Woodbury matrix identity, which provides an efficient way to compute the inverse. The update for $B_t^{-1}$ is given by:

\[
B_t^{-1} = B_0^{-1}\left(I - G_t\left(D_t^\dagger B_0^{-1}G_t\right)^{-1}D_t^\dagger B_0^{-1}\right) + D_t\left(D_t^\dagger B_0^{-1}G_t\right)^{-1}D_t^\dagger B_0^{-1}.
\]

This update is equivalent to the update for $H_t$, given that 
\begin{equation}
    B_0^{-1} = H_0, \quad\text{and} \quad G_t^\dagger = \left(D_t^\dagger B_0^{-1}G_t\right)^{-1}D_t^\dagger B_0^{-1}.
\end{equation}

In summary, quasi-Newton methods update the Hessian approximation using the secant equation. Single-secant methods update the approximation using the secant equation one by one, while multisecant methods use a batch of secant equations. The choice of updating strategy and pseudo-inverse affects the behavior of the method.

\subsubsection{Davidon-Fletcher-Powell (DFP) Formula}

The DFP formula is a Quasi-Newton update rule used to iteratively refine an approximation of the inverse Hessian matrix. It is defined as follows:

\begin{equation}
H_{t} = H_{t-1} + \frac{d_t d_t^T}{d_t^T g_t} - \frac{H_{t-1} g_t g_t^T H_{t-1}}{g_t^T H_{t-1} g_t},
\end{equation}

In the above equation, $g_t = \nabla f(x_{t}) - \nabla f(x_{t-1})$ represents the difference in gradients, and $d_t = x_{t} - x_{t-1}$ denotes the difference in parameter values. The DFP formula updates the matrix $H_t$ using a rank-two matrix such that it remains symmetric and positive definite.





\subsubsection{Multisecant Broyden Methods}

The multisecant Broyden methods utilize the update equation from \cref{eq:multisecant_qn_formula}, where $A^\dagger$ is chosen as the Moore-Penrose pseudo-inverse of $A$, given by $A^\dagger = (A^TA)^{-1}A$. In this equation, $B_0$ and $H_0$ are scaled identity matrices. After simplification, the two types of updates can be expressed as follows:

\begin{align}
B_{t}^{-1} & = D_t\left(D_t^\dagger G_t\right)^{-1}D_t^\dagger + B_0^{-1}\left(I-G_t\left(D_t^\dagger G_t\right)^{-1}D_t^\dagger \right), \label{eq:broyden_type_1}\\
H_t & = D_t(G_t^TG_t)^{-1}G_t^T+H_0\left(I-G_t\left(G_t^TG_t\right)^{-1}G_t^T\right) \label{eq:broyden_type_2}.
\end{align}

Both updates are quite similar, differing mainly in the choice of the pseudo-inverse of the matrix $G$.

\subsubsection{Link with Anderson Acceleration}

The connection between quasi-Newton methods and Anderson Acceleration is strong, as, for instance, Broyden methods and Anderson acceleration are equivalent. To illustrate this, let's closely examine the update of $\alpha$ in \cref{eq:anderson_type_1}:
\begin{align*}
& x_{t+1} =  x_t - h\nabla f(x_t) + (D_t-hG_t)\alpha,\quad \alpha = \argmin f(x_t) + \nabla f(x_t)D_t\alpha + \frac{1}{2}\alpha^TD_t^TG_t\alpha \\
\Leftrightarrow & x_{t+1} =  x_t - h\nabla f(x_t) +(D_t-hG_t)\alpha,\quad \alpha : D_t^T\nabla f(x_t) + D_t^TG_t\alpha = 0\\
\Leftrightarrow & x_{t+1} =  x_t - h\nabla f(x_t) + (D_t-hG_t)\alpha,\quad \alpha : \alpha =- (D_t^TG_t)^{-1}D_t^T\nabla f(x_t)\\
\Leftrightarrow & x_{t+1} =  x_t - h\nabla f(x_t) - (D_t-hG_t) (D_t^TG_t)^{-1}D_t^T\nabla f(x_t).\\
\Leftrightarrow & x_{t+1} =  x_t -\left(D_t(D_t^TG_t)^{-1}D_t^T+h\left(I- G_t(D_t^TG_t)^{-1}D_t^T\right)\right)\nabla f(x_t)
\end{align*}

The above step is precisely the quasi-Newton step $x_{t+1}=x_t-B_t^{-1}\nabla f(x_t)$, where $B_t^{-1}$ corresponds to the Broyden update given by Equation \ref{eq:broyden_type_1}, with $B_0^{-1} = hI$. A similar reasoning can be applied to Equation \ref{eq:broyden_type_2}.

When considering the single-secant updates, following the same reasoning as in Section 3 leads to the same conclusion for the SR-1 and DFP updates.

This result is expected since the approximations $H_t$ or $B_t^{-1}$ satisfy the single or multisecant equation:
\[
    H_t G_t = D_t.
\]
This indicates that the matrix $H_t$ maps vectors from the span of previous gradients to the span of previous directions. This observation justifies the construction in \cref{eq:def_next_iterate}.

\subsection{Links with Algorithms \ref{alg:type1} and \ref{alg:type2}}  \label{sec:link}

Both Algorithms \ref{alg:type1} and \ref{alg:type2} can be viewed as quasi-Newton and Anderson/nonlinear acceleration schemes. The update formulas are
\begin{align}
    & \min_\alpha f(x_t) + \nabla f(x_t)^TD_t\alpha + \frac{\alpha^TH_t\alpha}{2} + \frac{M\|D_t\alpha\|^3}{6},  \quad H_t  \defas  \frac{G_t^TD_t+D_t^TG_t + \mathrm{I}M\|D_t\|\|\varepsilon_t\|}{2}.
    \tag{Type I}\\
     & \min_\alpha \left\|\nabla f(x_t) + G_t\alpha\right\| + \frac{M}{2}\Big( \sum_{i=1}^N |\alpha_i| [\varepsilon_t]_i + \|D_t\alpha\|^2\Big),\tag{Type II}
\end{align}
The resemblance with Anderson/nonlinear acceleration is strong, as the objective function is similar. If the function is quadratic, $L=0$ and therefore $M$ can also be set to $0$; hence, the coefficients $\alpha$ are \textit{exactly} the type I and type II Anderson steps \cref{eq:anderson_type_1,eq:anderson_acc}.

The same idea holds when compared to quasi-Newton methods. In both cases, the optimal solution $\alpha^\star$ can be written implicitly:
\begin{align}
    & \alpha^\star = -\left(H_t + \frac{MD_t^TD_t \|D_t\alpha^\star\|}{6}\right)^{-1}D_t^T\nabla f(x_t), 
    \tag{Type I - solution}\\
    & \alpha^\star = - \left(G_t^TG_t+\tilde M D_t^TD_t\right)^{-1} \left(G_t^T \nabla f(x) + \frac{\tilde M\|\varepsilon_t\|}{2}\partial(|\alpha^\star|)\right),\tag{Type II - solution}
\end{align}
where $\tilde M \defas\|\nabla f(x_t) + G_t\alpha\|M $ and $\partial(|\alpha^\star|)$ is a subgradient of $|\alpha^*|$. The step then reads
\begin{align}
    x_{t+1} & = x_t + D\alpha^\star \tag{Generic step}\\
    x_{t+1} & = x_t - D_t\left(H_t + \frac{MD_t^TD_t \|D_t\alpha^\star\|}{6}\right)^{-1}D_t^T\nabla f(x_t), 
    \tag{Type I - step}\\
    x_{t+1} &= x_t - D_t \left(G_t^TG_t+\tilde M D_t^TD_t\right)^{-1} \left(G_t^T \nabla f(x) + \frac{\tilde M\|\varepsilon_t\|}{2}\partial(|\alpha^\star|)\right),\tag{Type II - step}
\end{align}
Type I is a quasi-Newton step with a symmetrization of $G^TD$ and a regularization. In contrast, the type II step can be seen as a quasi-Newton method with a regularization on $G^\dagger$, with a correction term on the gradient. Therefore the Hessian approximation reads
\[
    B_t^{-1} = D_t\left(H_t + \frac{MD_t^TD_t \|D_t\alpha^\star\|}{6}\right)^{-1}D^T, \quad H_t = D_t \left(G_t^TG_t+\tilde M D_t^TD_t\right)^{-1}G_t^T.
\]

Again, when the objective function is quadratic, $L=0$ and therefore $M=0$. Moreover, when $f$ is quadratic, the matrix multiplication $D^TG$ satisfies $D^TG+G^TD = 2D^TG$ as $D^TG$ becomes symmetric. Hence,
\begin{align}
    x_{t+1} & = x_t - D_t\left(D_t^TG_t\right)^{-1}D_t^T\nabla f(x_t), 
    \tag{Type I - quadratic}\\
    x_{t+1} &= x_t - D_t \left(G_t^TG_t\right)^{-1} G_t^T \nabla f(x_t),\tag{Type II quadratic}
\end{align}
The steps are \textit{exactly} the type I and type II multisecant Broyden methods from \cref{eq:broyden_type_1,eq:broyden_type_2}, with the only difference that there is no initialization $H_0$ or $B_0$. 
\clearpage
\section{Solving the sub-problems} \label{sec:solving_subproblem}

\paragraph{Solving the Type 1 Subproblem} The Type 1 subproblem is a well-studied problem that involves minimizing a specific objective function. A method proposed by \citep{nesterov2006cubic} has proven to be efficient for solving this problem. The method utilizes eigenvalue decomposition on a matrix to find the optimal solution. In this paper, the matrix involved in this problem is relatively small, therefore eigenvalue decomposition is not a concern even for large-scale problems. The subproblem aims to determine the norm of the solution, and this can be achieved through solving one nonlinear equation using bisection or secant method.

\paragraph{Solving the Type 2 Subproblem} The Type 2 subproblem can be formulated as a Second-Order Cone Program (SOCP). The objective function of this subproblem consists of three terms: a norm term, a sum of absolute values term, and a quadratic term. The norm term can be transformed using singular value decomposition, and the sum of absolute values term can be expressed as with linear constraints. The quadratic term can be simplified using a rotated quadratic cone. By utilizing these techniques, the Type 2 subproblem can be effectively solved using existing SOCP solvers.

\subsection{Solving the Type 1 Subproblem}
The Type 1 subproblem can be expressed as follows:

\[
    \min_\alpha \nabla f(x) D\alpha + \frac{1}{2} \alpha^TH\alpha + \frac{M}{6}\|D\alpha\|^3,
\]

where $H$ is symmetric but not necessarily positive definite. This problem has been well-studied, and \citep{nesterov2006cubic} proposed an efficient method to solve it using eigenvalue decomposition on the matrix $H$. Although eigenvalue decomposition may be challenging for large-scale problems, it is not a concern here since $H\in\mathbb{R}^{N\times N}$, with a relatively small $N$ (e.g., $N=25$ in the experiments).

In essence, the subproblem involves determining the norm of the solution $r = \|\alpha\|$. This can be accomplished through a simple bisection on the following system of nonlinear equations:

\begin{equation} \label{eq:subproblem_type_1_plain}
    \left(H+\frac{MD^TDr}{2}I\right)\alpha = -D^t\nabla f(x),\quad \|\alpha\| = r,\quad r\geq -\lambda_{\min}(H).
\end{equation}

Interestingly, this problem is equivalent to the following formulation, as shown in Proposition \ref{prop:equivalence_type1_subproblem}:

\begin{equation}\label{eq:subproblem_type_1_preco}
    \left(\Lambda+\frac{Mr}{2}I\right)\tilde\alpha = -V^T(D^TD)^{-1/2}D^t\nabla f(x),\; \|\alpha\| = r,\; r\geq -\lambda_{\min}(H), \; \tilde \alpha = V^T(D^TD)^{1/2}\alpha,
\end{equation}

which involves the eigenvalue decomposition $(D^TD)^{-1/2}H(D^TD)^{-1/2} = V\Lambda V^T$.

\begin{proposition} \label{prop:equivalence_type1_subproblem}
    Problems \cref{eq:subproblem_type_1_plain} and \cref{eq:subproblem_type_1_preco} are equivalent.
\end{proposition}

\begin{proof}
    The first step is to split $D^TD = (D^TD)^{1/2}(D^TD)^{1/2}$ and then employ an eigenvalue decomposition on $(D^TD)^{-1/2}H(D^TD)^{-1/2} = V\Lambda V^T$ (where $V$ is orthonormal due to the symmetry of the matrix):
    \begin{align*}
        & \left(H+\frac{MD^TDr}{2}I\right)\alpha = -D^t\nabla f(x) \\
        \Leftrightarrow&  (D^TD)^{1/2}\left((D^TD)^{-1/2}H(D^TD)^{-1/2}+\frac{Mr}{2}I\right)(D^TD)^{1/2}\alpha = -D^t\nabla f(x)\\
        \Leftrightarrow&  (D^TD)^{1/2}V\left(\Lambda+\frac{Mr}{2}I\right)V^T(D^TD)^{1/2}\alpha = -D^t\nabla f(x)\\
        \Leftrightarrow&  \left(\Lambda+\frac{Mr}{2}I\right)V^T(D^TD)^{1/2}\alpha = -V^T(D^TD)^{-1/2}D^t\nabla f(x)\\
        \Leftrightarrow&  \left(\Lambda+\frac{Mr}{2}I\right)\tilde\alpha = -V^T(D^TD)^{-1/2}D^t\nabla f(x).
    \end{align*}
\end{proof}

Once the eigenvalue decomposition is performed, the subproblem \cref{eq:subproblem_type_1_preco} becomes relatively simple since it involves solving a diagonal system of equations for a fixed value of $r$. The main objective is to find an interval $[r_{\min}, r_{\max}]$ that encompasses the optimal value $r=\|\alpha\|$. Once this interval is identified, a straightforward bisection or secant method can be employed to obtain the optimal solution.

\paragraph{Finding initial bounds} Starting with $r_{\min} = \max\{0, -\lambda_{\min(H)}\}$ and $r_{\max} = \max\{2r_{\min}, 1\}$, 
\[
    \text{do } r_{\max} \gets 2r_{\max} \quad \text{while }  \left\|\tilde \alpha\right\| \geq r_{\max}.
\]
where $\tilde\alpha = -\left(\Lambda+\frac{Mr_{\max}}{2}I\right)^{-1}V^T(D^TD)^{-1/2}D^t\nabla f(x)$. Increasing $r_{\max}$ increases the regularization, hence reduces the norm of $\tilde \alpha$.

\paragraph{Finding $\alpha$} After $r^\star$ has been found such that $|r^\star-\|\tilde \alpha\||$ is sufficiently small, the best $\alpha$ is simply
\[
    \alpha = (D^TD)^{-1/2}V\tilde \alpha =  -(D^TD)^{-1/2}V \left(\Lambda+\frac{Mr^\star}{2}I\right)^{-1}V^T(D^TD)^{-1/2}D^t\nabla f(x).
\]
In the case where the diagonal matrix is not invertible, which happens when $r^\star = r_{\min}$, it suffices to use the pseudo-inverse instead. 

Note that $D^TD$ is an $N\times N$ matrix, where $N$ is small, therefore, computing its inverse is inexpensive. Moreover, when $D$ is orthogonal, $D^TD=I$, therefore there is no need to invert it. In addition, $(\Lambda+\frac{Mr^\star}{2}I)^{-1}$ can be computed in $O(N)$ complexity since the matrix is diagonal.

\subsection{Solving the Type 2 Subproblem}

The Type 2 subproblem is given by:

\begin{equation}\label{eq:subproblem_type_2}
    \min_{\alpha} \underbrace{\left\|\nabla f(x) + G\alpha\right\|}_{\textbf{(a)}} + \frac{L}{2}\Big( \underbrace{\sum_{i=1}^N |\alpha_i| \varepsilon_i}_{\textbf{(b)}} + \underbrace{\|D\alpha\|^2}_{\textbf{(c)}}\Big).
\end{equation}

Although it may not be immediately apparent, this subproblem can be formulated as a Second-Order Cone Program (SOCP) with $O(N)$ variables and constraints.

\subsubsection{Fundamentals of SOCP}

SOCP solvers handle the following conic problems:

\begin{align}
    & \min_{x,t_i,\omega_i} c_0x + \sum_{i} c_i [t_i;\omega_i] \quad \text{subject to} \nonumber\\
    & A_0x + \sum_{i=1}^{k}A_i[t_i;\omega_i] =b  \label{eq:socp} \tag{SOCP Standard Matrix Form}\\
    & x \geq 0 \nonumber \\
    & (t_i,\omega_i) \in \mathcal{K}_i \;\; \Leftrightarrow t_i \geq \|\omega_i\|, \;\; t\geq 0. \nonumber
\end{align}

Here, $k$ represents the number of cones, and the cone $\mathcal{K}$ refers to the second-order cone, also known as the \textit{Lorenz} cone.

A useful transformation is the \textit{rotated quadratic cone}, defined as follows:

\begin{align*}
    [a,b,c] \in \mathcal{K}_q  \quad \Leftrightarrow \quad 2ab \geq \|c\|^2.
\end{align*}

The rotated quadratic cone can be reformulated as a second-order cone using a linear transformation:

\[
   \text{if} \quad \begin{bmatrix}
        a \\ b \\ c
    \end{bmatrix} = \begin{bmatrix}
        \frac{1}{\sqrt{2}} & \frac{1}{\sqrt{2}} & 0  \\
        \frac{1}{\sqrt{2}} & - \frac{1}{\sqrt{2}} & 0 \\
        0 & 0 & I_K
    \end{bmatrix}
    \begin{bmatrix}
        t \\ \omega^{(0)} \\ \omega
    \end{bmatrix} \quad \text{then} \quad (t, [\omega^{(0)};\,\omega])\in\mathcal{K} \quad \Leftrightarrow \quad [a,b,c] \in \mathcal{K}_q.
\] 

Thanks to this transformation, the rotated quadratic cone can be included in SOCP solvers.
\subsubsection{SOCP Formulation of the Type 2 Subproblem}

The SOCP of \cref{eq:subproblem_type_2} is composed of the three terms $\textbf{a}$, $\textbf{b}$, and $\textbf{c}$.

\paragraph{Term (a)} Let $U_G\Sigma_G V^T_G$ be the singular value decomposition of $G$. Write $P_G = U_GU_G^T$ as the projector onto the columns of $G$. Then,

\begin{align*}
    \left\|\nabla f(x) + R\alpha\right\| &= \left\|P_G\nabla f(x) + P_GG\alpha + (I-P_G)\nabla f(x)\right\| \\
    &= \sqrt{\left\|P_G\nabla f(x) + R\alpha\right\|^2 + \|(I-P_G)\nabla f(x)\|^2} \\
    &= \sqrt{\left\|U_{G}\left(U_{G}^T\nabla f(x) + \Sigma_{G} V^T_G\alpha\right)\right\|^2 + \|(I-P_G)\nabla f(x)\|^2} \\
    &= \sqrt{\left\|U_{G}^T\nabla f(x) + \Sigma_{G} V^T_G\alpha\right\|^2 + \|(I-P_G)\nabla f(x)\|^2}
\end{align*}

Let the vector $\omega_1 = \left[U_{G}^T\nabla f(x) + \Sigma_{G} V\alpha;\, \|(I-P_G)\nabla f(x)\|\right]$. Hence,

\[
    \left\|\nabla f(x) + G\alpha\right\| = \min_{t_1,\,\alpha,\,\omega_1} t_1 : (t_1, \omega_1) \in \mathcal{K}_L,\quad \omega_1 = \left[U_{G}^T\nabla f(x) + \Sigma_{G} V\alpha;\, \|(I-P_G)\nabla f(x)\|\right].
\]

\paragraph{Term (b)} This term is standard in linear programming. Let $\alpha = \alpha_+ - \alpha_-$, with $\alpha_+, \alpha_- \geq 0$,

\[
    \sum_{i=1}^N |\alpha_i| \varepsilon_i = \sum_{i=1}^N (\alpha_+ + \alpha_-)\varepsilon_i.
\]

\paragraph{Term (c)} Let $U_D\Sigma_D V^T_D$ be the singular value decomposition of $D$. Using the rotated cone, the constraint can be written as

\[
    2 t_3 b \geq \| U_D \Sigma_D V_D \alpha \|^2 = \| \Sigma_D V_D \alpha \|^2 ,\quad b = \frac{1}{2}.
\]

Using the transformation into a Lorenz cone, this is equivalent to

\[
    \begin{bmatrix}
        1 & 0 & 0 \\
        0 & 1 & 0 \\
        0 & 0 & \Sigma_D V_D^T
    \end{bmatrix}
    \begin{bmatrix}
    t_3 \\ b \\ \alpha
    \end{bmatrix}
    =
 \begin{bmatrix}
        \frac{1}{\sqrt{2}} & \frac{1}{\sqrt{2}} & 0 \\
        \frac{1}{\sqrt{2}} & - \frac{1}{\sqrt{2}} & 0 \\
        0 & 0 & I_{k}
    \end{bmatrix}
    \begin{bmatrix}
    t_2\\
    \omega^{(0)}_2\\
    \omega_2
    \end{bmatrix},\quad b = \frac{1}{2},\quad (t_2, \,[ \omega^{(0)}_2,\,\omega_2])\in\mathcal{K}.
\]

\textbf{Simplification.} Note that, since $b=\frac{1}{2}$, the value can be immediately replaced. Same idea with $t_3$: the constraint is written as

\[
    t_3 = \frac{t_2 + \omega^{(0)}_2}{\sqrt{2}},\quad t_3 \geq 0.
\]

Since, by construction, $t_2 \geq \omega^{(0)}_2$ and $t_2 \geq 0$, $t_3$ always satisfies the condition, which means both $t_3$ and its constraint can be removed. The constraints thus simplify into

\[
    \begin{bmatrix}
    \frac{1}{2} \\
    0
    \end{bmatrix}
    +
    \begin{bmatrix}
        0 \\
        \Sigma_D V_D^T
    \end{bmatrix}
    \begin{bmatrix}
    \alpha
    \end{bmatrix}
    =
    \begin{bmatrix}
        \frac{1}{\sqrt{2}} & - \frac{1}{\sqrt{2}} & 0 \\
        0 & 0 & I_{k}
    \end{bmatrix}
    \begin{bmatrix}
    t_2\\
    \omega^{(0)}_2\\
    \omega_2
    \end{bmatrix},\quad (t_2, \,[ \omega^{(0)}_2,\,\omega_2])\in\mathcal{K}.
\]

\paragraph{Final formulation}

Gathering all terms, the final SOCP formulation reads

\begin{align*}
    \text{minimize}\quad&  t_1 + \frac{L}{2}\left( (\alpha_++\alpha_-)^T\varepsilon + t_2\right)\\
    \text{subject to}\quad & \omega_1 = \left[U_G^T\nabla f(x) + \Sigma_G V_G^T\alpha\;;\; \|(I-P_G)\nabla f(x)\|\right],\\
    & \alpha_+,\alpha_-\geq 0\\
    & \alpha = \alpha_+-\alpha_-\\
    &  \begin{bmatrix}
        \textbf{0}_{1\times N}  & -\frac{1}{\sqrt{2}} & \frac{1}{\sqrt{2}} & 0 \\
        \Sigma_D V_D^T & \textbf{0}_{N\times 1} & \textbf{0}_{N\times 1} & -I_{N}
    \end{bmatrix}
    \begin{bmatrix}
    \alpha\\
    t_2\\
    \omega^{(0)}_2\\
    \omega_2
    \end{bmatrix}
    =
    \begin{bmatrix}
    0 \\ -\frac{1}{2} \\ \textbf{0}_{N\times 1}
    \end{bmatrix}
    \\
    & (t_1,\,\omega_1)\in \mathcal{K}, \quad (t_2,[\omega^{(0)}_2;\omega_2])\in \mathcal{K}_L, \quad t_2\geq 0.
\end{align*}

\paragraph{Standard matrix formulation}

The SOCP can be written under the standard matrix form \cref{eq:socp}. Let the variables

\[
    \alpha_+,\, \alpha_- \geq 0, \quad (t_1,\omega_1) \in \mathcal{K}_1,\quad (t_2,[\omega_2^{(0)} \omega_2]) \in  \mathcal{K}_2 ,
\]

where $t_1$, $t_2$, and $\omega_2^{(0)}$ are scalars, $\omega_2$, $\alpha_+$, and $\alpha_-$ are vectors of size $N$, and $\omega_1$ is a vector of size $N+1$. The SOCP matrices read

\begin{align*}
    & c_0 = \begin{bmatrix}
         \frac{L\varepsilon^T}{2} & \frac{L\varepsilon^T}{2}
    \end{bmatrix}\quad 
     c_1 = \begin{bmatrix}
        1 & \textbf{0}_{1\times N+1}
    \end{bmatrix}
    \quad 
    c_2 = \begin{bmatrix}
        \frac{L}{2\sqrt{2}} & \frac{L}{2\sqrt{2}} & \textbf{0}_{1\times N}
    \end{bmatrix}\\
    & A_0 = \begin{bmatrix}
        -\Sigma_G V_G^T & \Sigma_G V_G^T\\
        \textbf{0}_{2\times N} & \textbf{0}_{2\times N}\\
        \Sigma_DV_D^T & -\Sigma_DV_D^T
    \end{bmatrix}\\
    & A_1 = \begin{bmatrix}
        \textbf{0}_{N+1\times 1} & I_{N+1\times N+1}\\
        \textbf{0}_{N+1\times 1} & \textbf{0}_{N+1\times N+1}\\
    \end{bmatrix}\\
    & A_2 = \begin{bmatrix}
        \textbf{0}_{N+1\times 1} & \textbf{0}_{N+1\times 1} & \textbf{0}_{N+1\times N}\\
        -\frac{1}{\sqrt{2}} & \frac{1}{\sqrt{2}} & \textbf{0}_{1\times N}\\
        \textbf{0}_{N\times 1} & \textbf{0}_{N\times 1} & -I_{N\times N}\\
    \end{bmatrix}\\
    & b = \begin{bmatrix}
        \nabla f(x)^TU_G & \|(I-P_R)\nabla f(x)\| & -\frac{1}{2} & \textbf{0}_{N\times 1}
    \end{bmatrix}^T.
\end{align*}

This completes the SOCP formulation of the type 2 subproblem.
\clearpage
\section{Additional Numerical Experiments} \label{sec:numerical_app}

This section presents additional numerical experiments.

\paragraph{Methods} The methods compared are the type 1 and type 2 steps with the following strategies: \textit{Iterate only}, \textit{Forward estimate only}, \textit{Greedy} (refer to \cref{sec:direction_requirements}), and the accelerated type 1 method with the strategy \textit{forward estimate only}. The batch methods are not included as they perform poorly regarding the number of Oracle calls. The baseline is the l-BFGS method from \texttt{minFunc} \citep{schmidt2005minfunc}.

\paragraph{Method parameters} In all experiments, the memory of the methods is set to $N=25$ and the $h$ for the forward estimates is set to $h=10^{-9}$. The parameters of the l-BFGS are left untouched except for the memory. The initial point is $x_0 = \nabla f(0_d)$.

\paragraph{Functions} The minimized problems are square loss with cubic regularization, logistic loss with small quadratic regularization, and the generalized Rosenbrock function. The regularization parameter of the square loss is set to $1e-3$ times the norm of the Hessian, and the regularization of the logistic loss is set to $1e-10$ times the square norm of the feature matrix.

\paragraph{Dataset} The datasets for the square and the logistic loss are Madelon \citep{guyon2003design}, Sido0 \citep{guyon2008design}, and Marti2 \citep{guyon2008design} datasets.

\paragraph{Post-processing} The dataset matrix is normalized by its norm, then a vector of ones is concatenated to the data matrix.

\subsection{Initial Parameter for the Backtracking Line search} \label{sec:backtracking_line_search}

The backtracking line search was used in all experiments. The estimation of the initial value $M_0$ (see \cref{eq:m0}) is based on the following observation. Since the function satisfies \cref{assump:lipchitiz_hessian},
\[
    \| \nabla f(y)-\nabla f(x) - \nabla^2 f(x)(y-x)) \| \leq \frac{L}{2} \|y-x\|^2,
\]
for some $x,\,y$. Hence, the parameter $L$ can be estimated as
\[
    L\approx 2\frac{\| \nabla f(y)-\nabla f(x) - \nabla^2 f(x)(y-x)) \|}{\|y-x\|^2}.
\]
Now, define
\begin{align*}
    s_h \defas h\nabla f(x_0),\qquad 
\end{align*}
for some small $h$ and $\tau >1$, and let $x = x_0$ and $y=x_0+s_{\tau h}$. Indeed, if $h$ is small, then 
\[
    \tau \left[ \nabla f(x_0+s_{h})-\nabla f(x_0) \right] \approx \tau \nabla^2 f(x)s_h = \nabla^2 f(x)s_{\tau h}.
\]
Therefore, 
\[
    \| \nabla f(x_0+s_{\tau h})-\nabla f(x_0) - \tau \left[ \nabla f(x_0+s_{h})-\nabla f(x_0) \right]  \| \approx \| \nabla f(x_0+s_{\tau h})-\nabla f(x_0) - \nabla^2 f(x)s_{\tau h}  \|,
\]
and hence, the Lipchitz constant can be estimated as
\begin{align} \label{eq:m0}
    M_0 = \frac{2}{\|s_{\tau h}\|^2}\| \nabla f(x_0+s_{\tau h})-\nabla f(x_0) - \tau \left[ \nabla f(x_0+s_{h})-\nabla f(x_0) \right]  \|.
\end{align}
In the experiments, $h$ is the same as the algorithm, and $\tau = 10$. Various choices of $\tau,\,h$ have been tested without significantly impacting the numerical convergence.

\newpage

\subsection{Scalability w.r.t. Dimension and Memory}

\begin{figure}[h!t]
    \centering
    \includegraphics[width=0.37\textwidth]{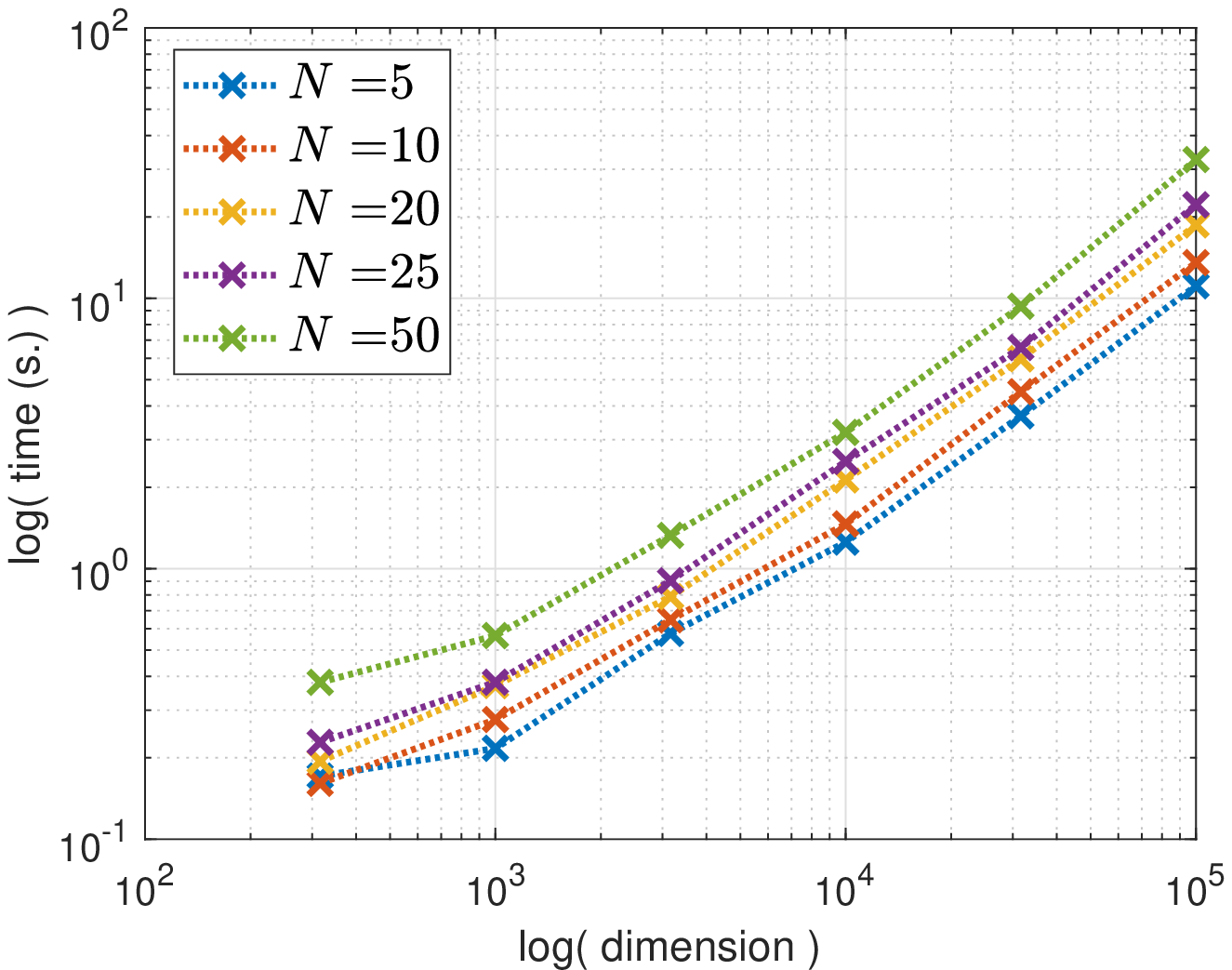}
    \includegraphics[width=0.37\textwidth]{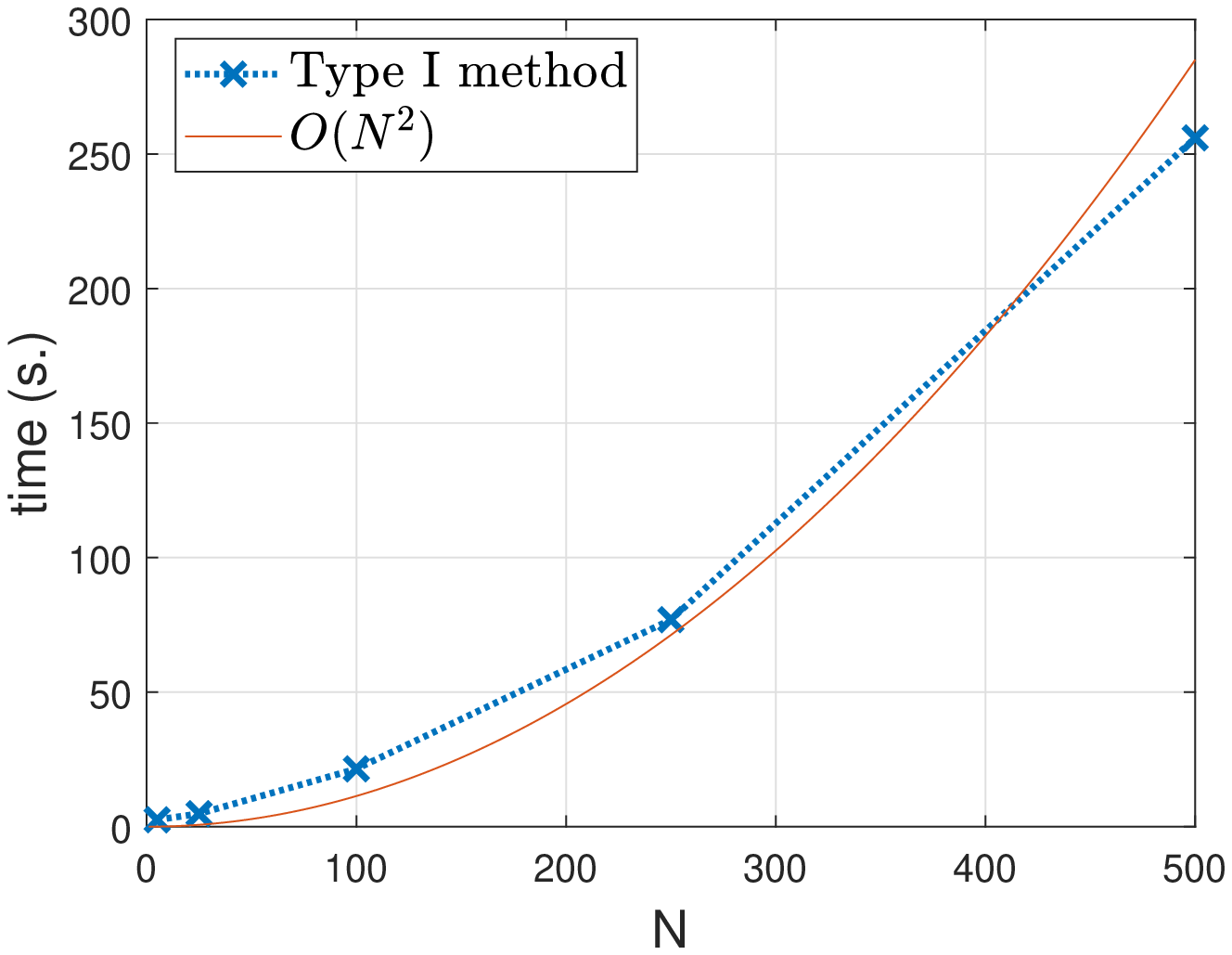}
    \caption{Scaling of the Type 1 method with the "orthogonal forward estimates only" updates rules w.r.t. $N$ and $d$ to minimize a random logistic regression function. As predicted by the theory, the scaling is linear in the dimension and quadratic w.r.t. $N$. The proposed method is suitable for large-scale problems, as it can quickly solve problems with $d\approx 10^6$.}
    \label{fig:scaling}
\end{figure}

\begin{figure}[h!t]
    \centering
    \includegraphics[width=0.3\textwidth]{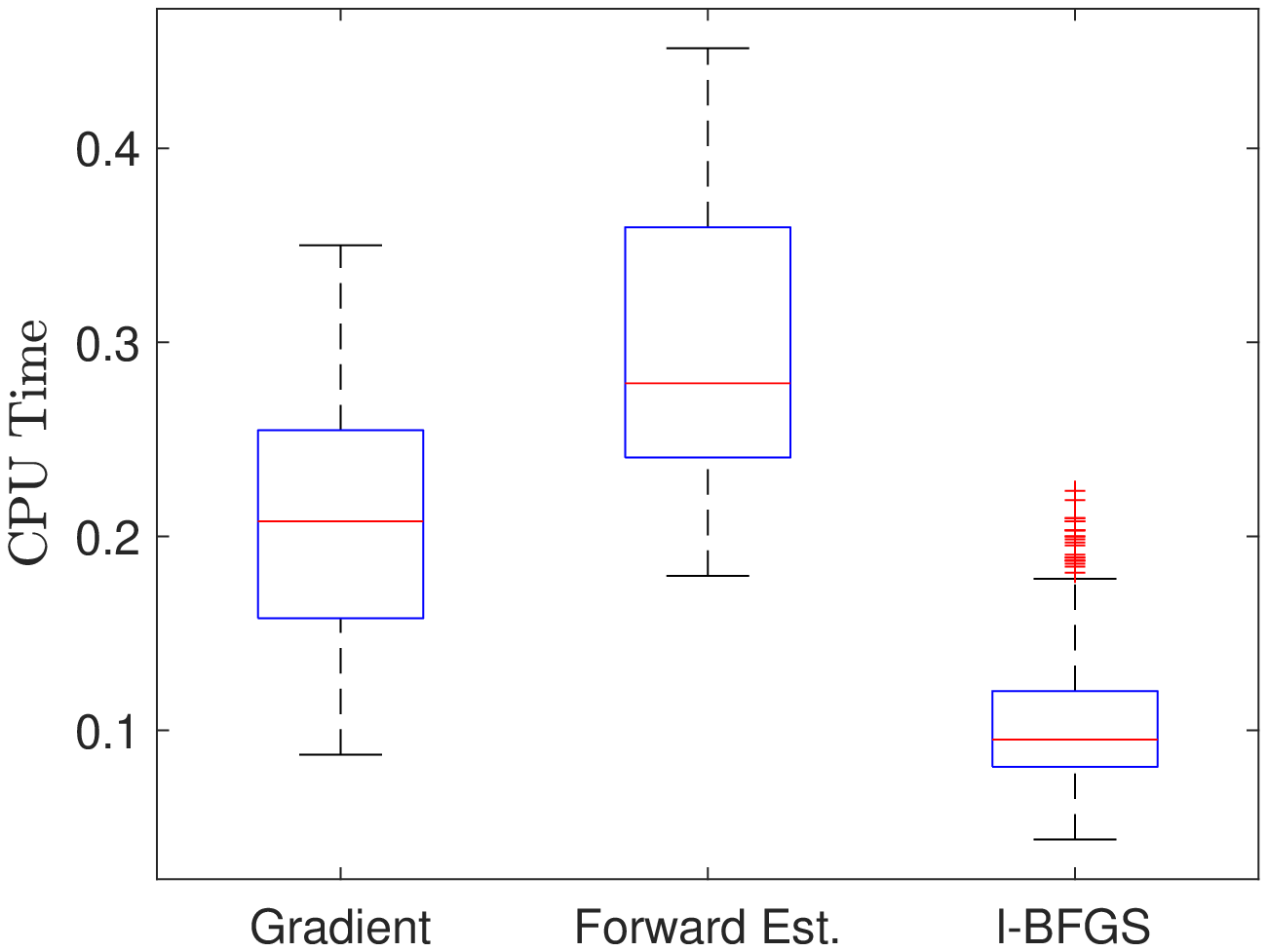}
    \includegraphics[width=0.3\textwidth]{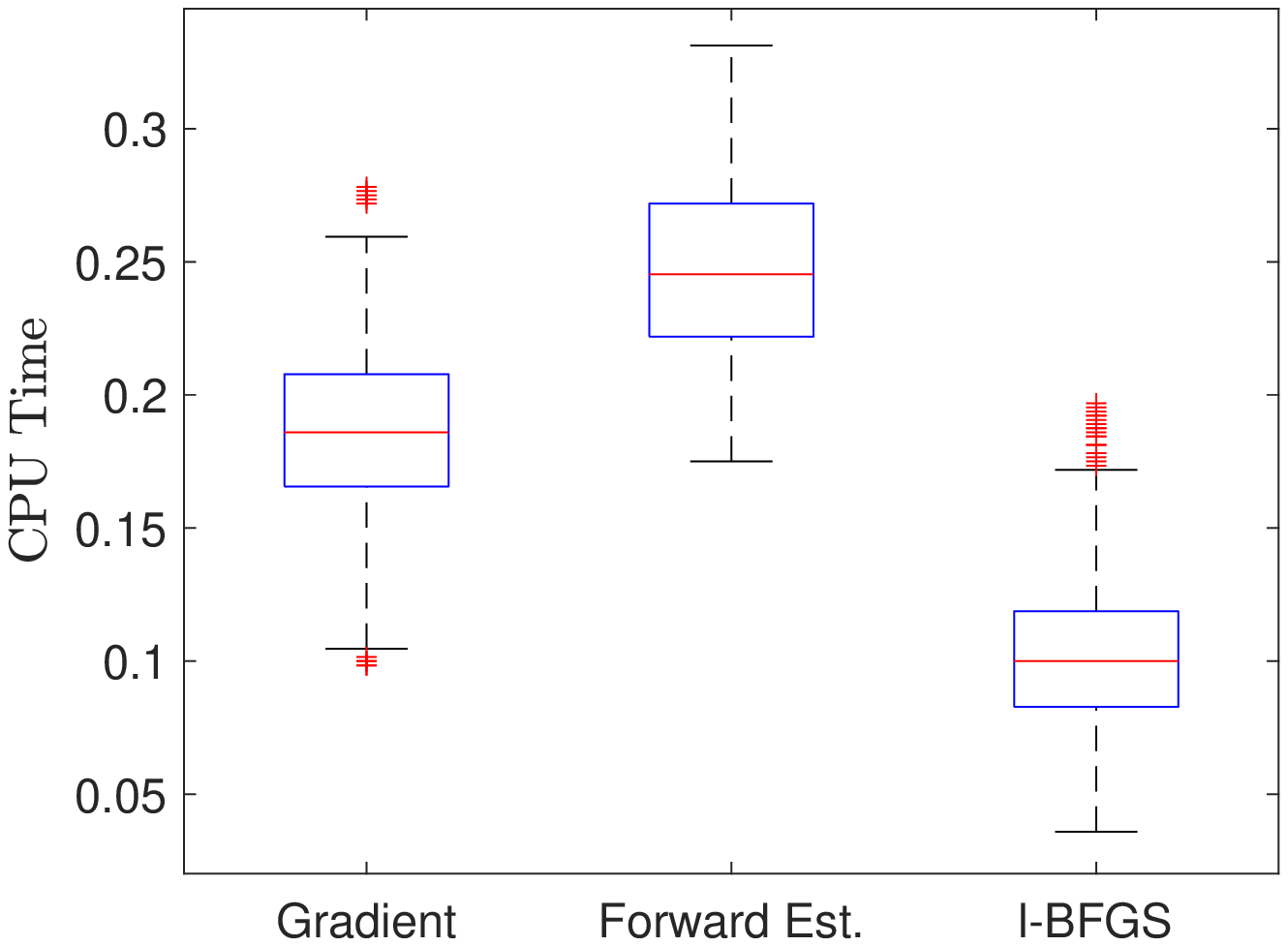}
    \includegraphics[width=0.3\textwidth]{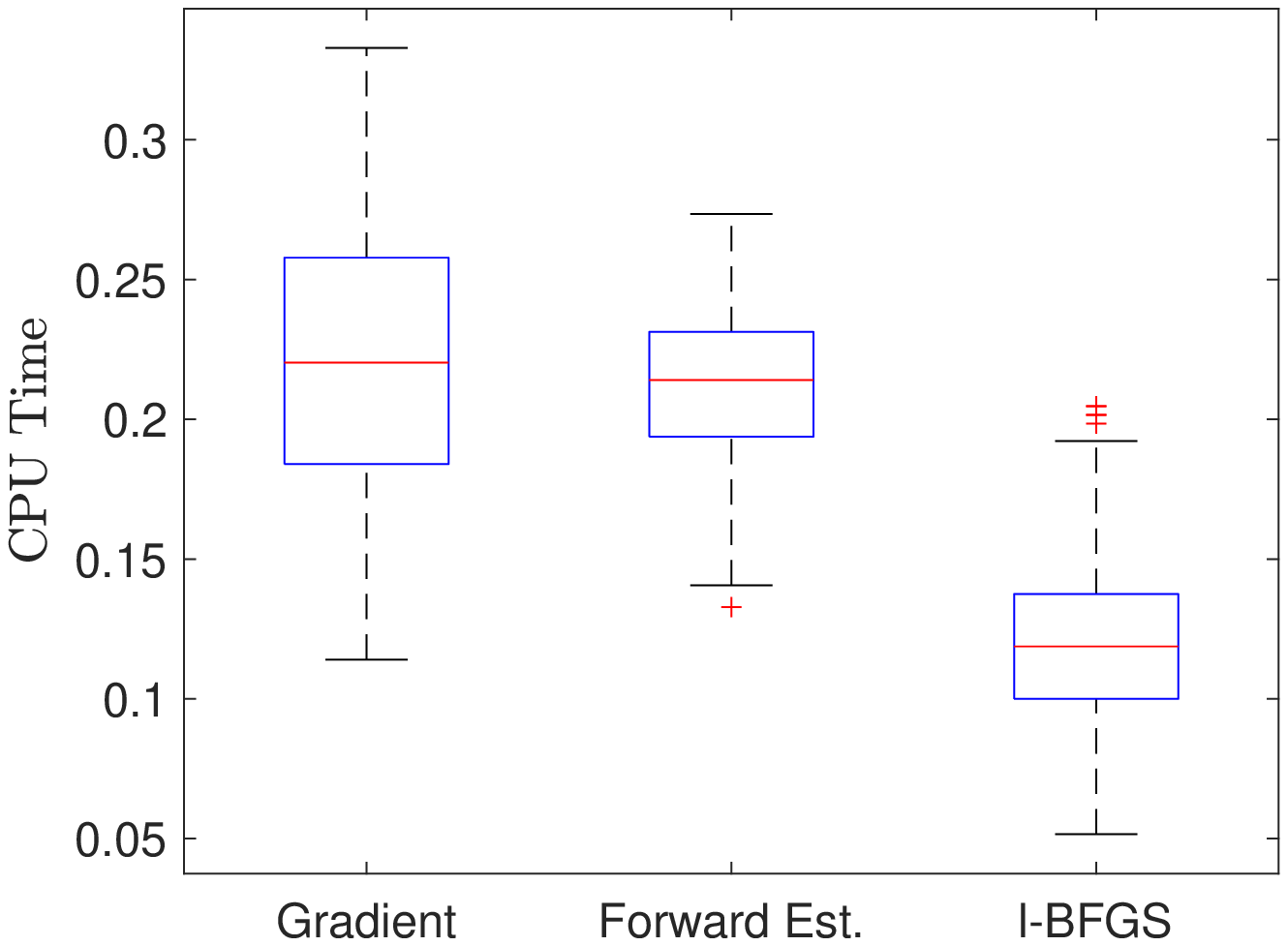}
    \caption{Distribution of the per-iteration time for three methods. The memory parameter of l-BFGS and the type I method is set to (left to right) $N=5,25,100$. The time required by the l-BFGS algorithm increases slightly when $N$ grows, and the per-iteration computation time is approximately two times faster than the type I method. Surprisingly, the total computation time of the type-1 method remains constant for different $N$ because the condition in the backtracking line search is more often satisfied. Note that the $\times 2$ factor between l-BFGS and the type 1 method is expected since the type 1 method requires at least 2 gradient calls.}
    \label{fig:hist}
\end{figure}

\subsection{Influence of h}

\begin{figure}[h!t]
    \centering
    \includegraphics[width=0.37\textwidth]{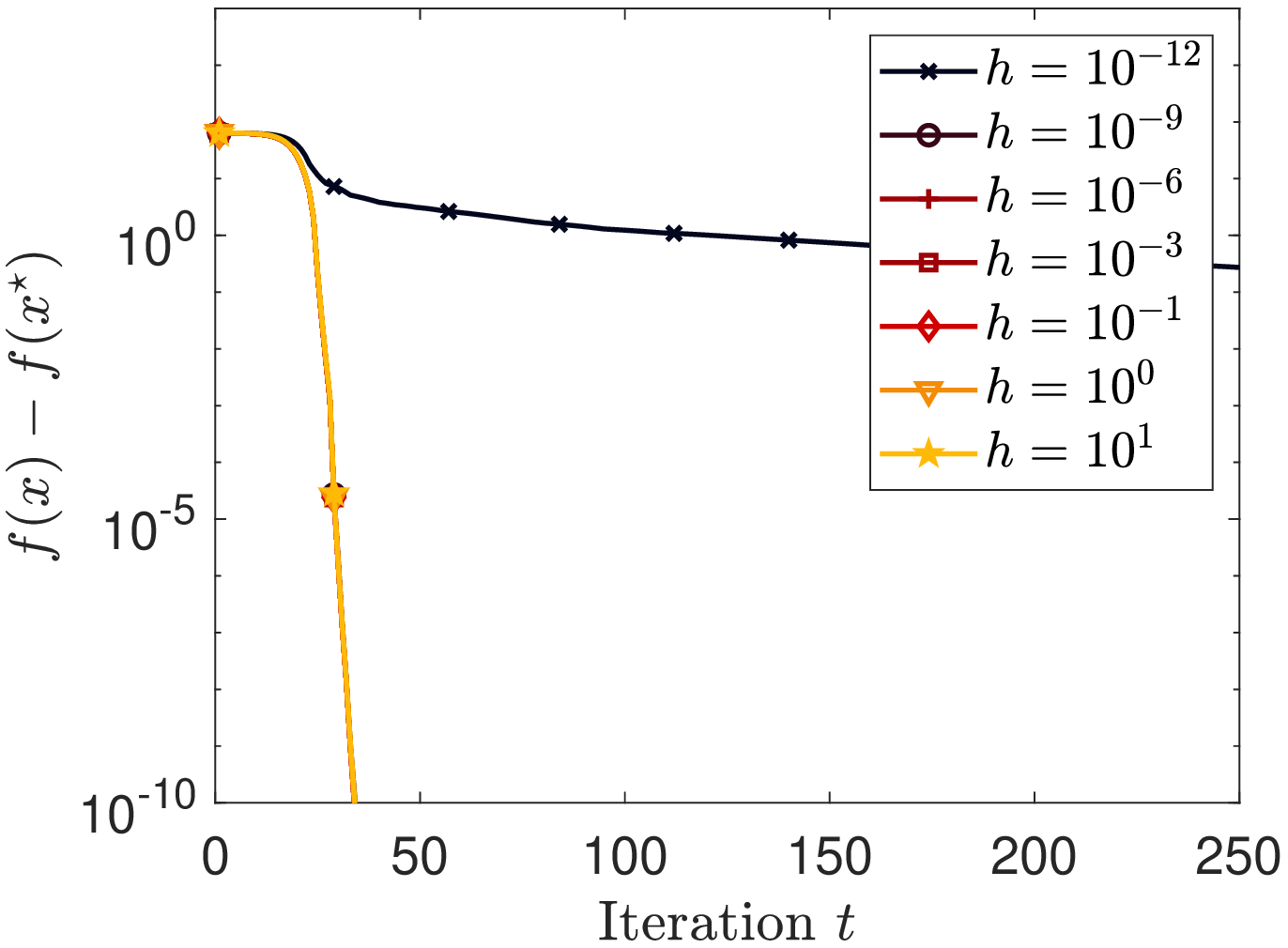}
    \includegraphics[width=0.37\textwidth]{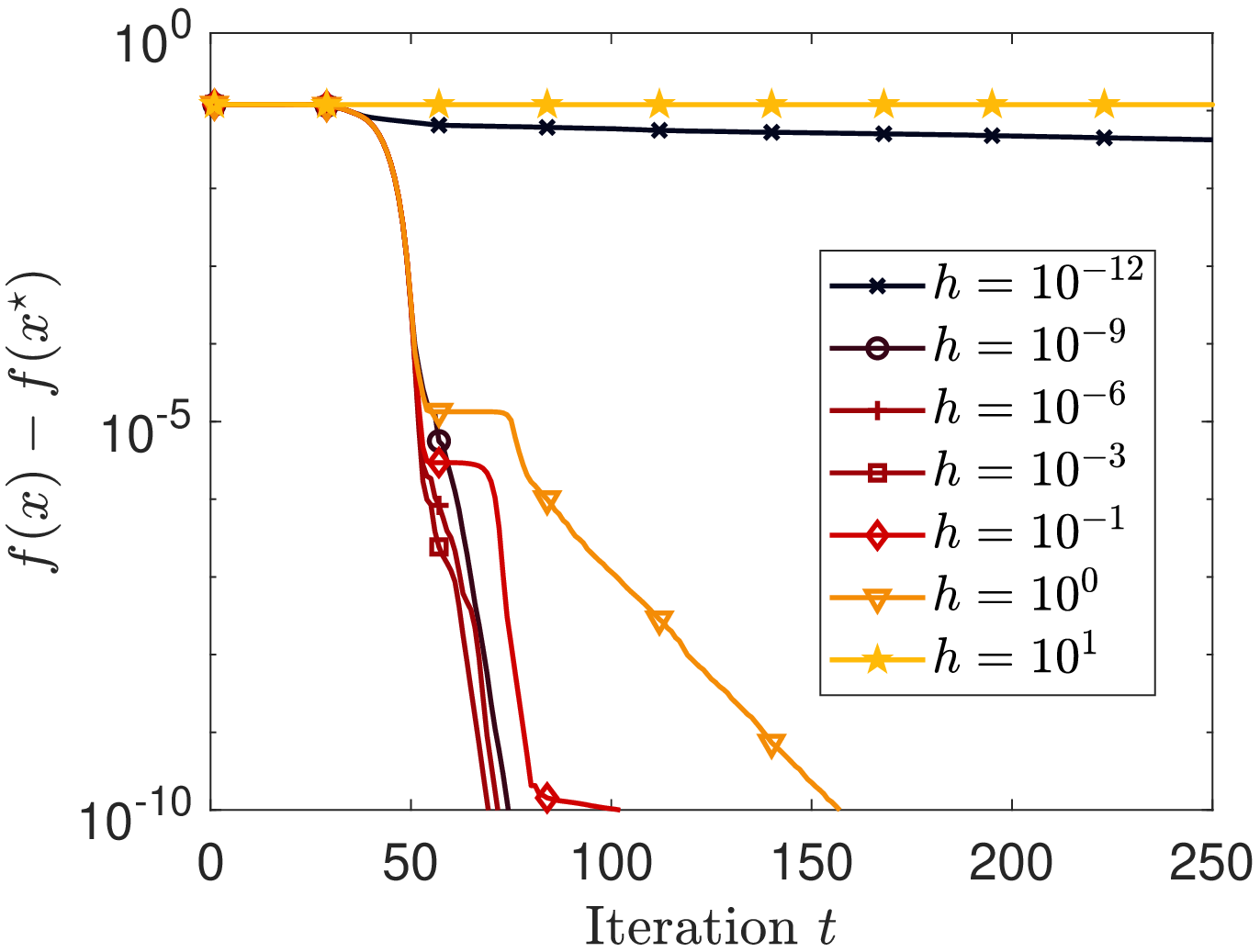}
    \caption{Influence of the step size $h$ to compute the forward estimate $x_{+\frac{1}{2}}$ in the "orthogonal forward estimates only" updates rules on the Madelon dataset to minimize a (left) quadratic and (right) a logistic loss. The range of acceptable $h$ is rather large. For instance, this range is $[10^{-9}, 10^{-1}]$ when minimizing the logistic loss.}
    \label{fig:dependency_h}
\end{figure}

\newpage

\subsection{Impact of the memory parameter N}

\begin{figure}[h!t]
    \centering
    \includegraphics[width=0.37\textwidth]{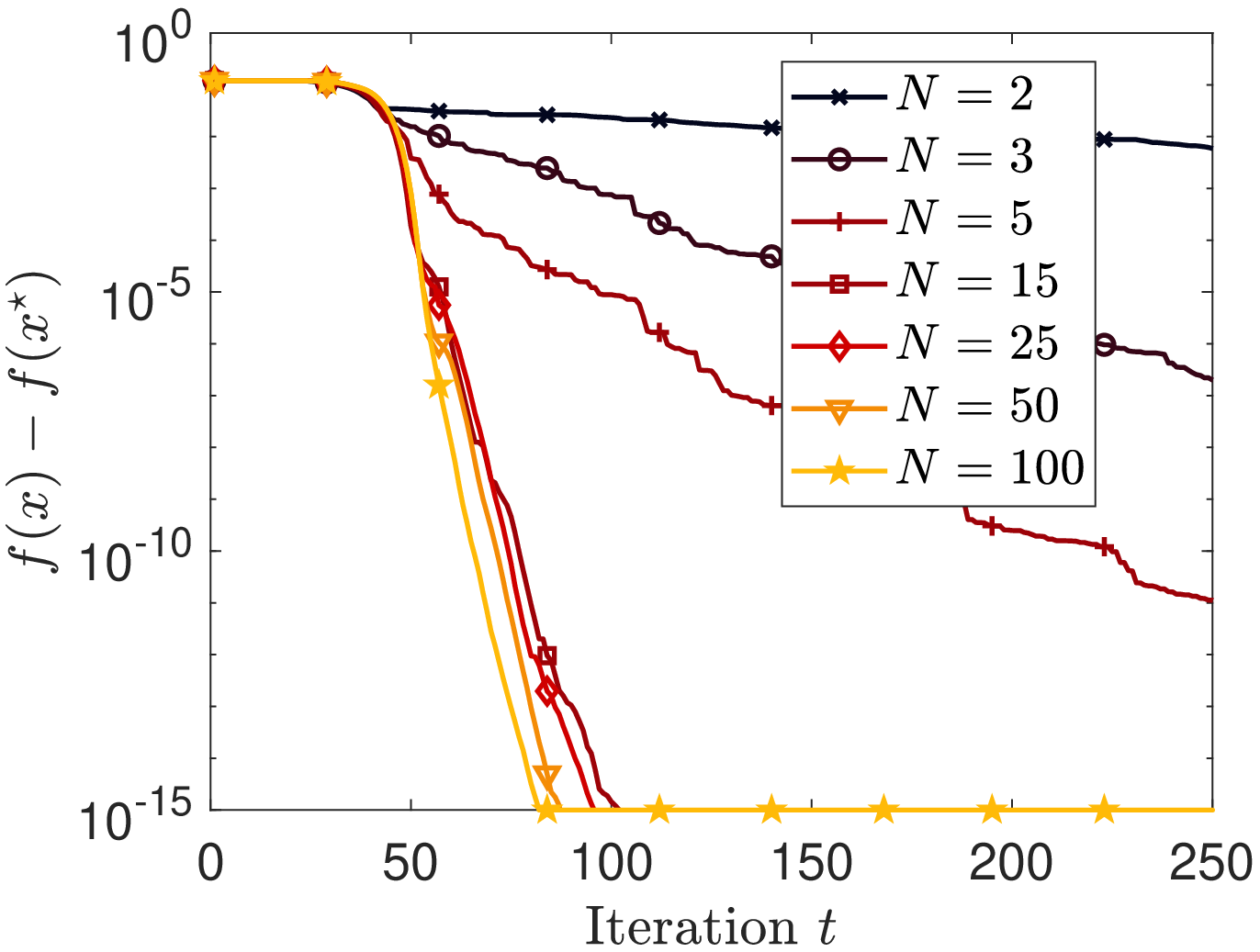}
    \includegraphics[width=0.37\textwidth]{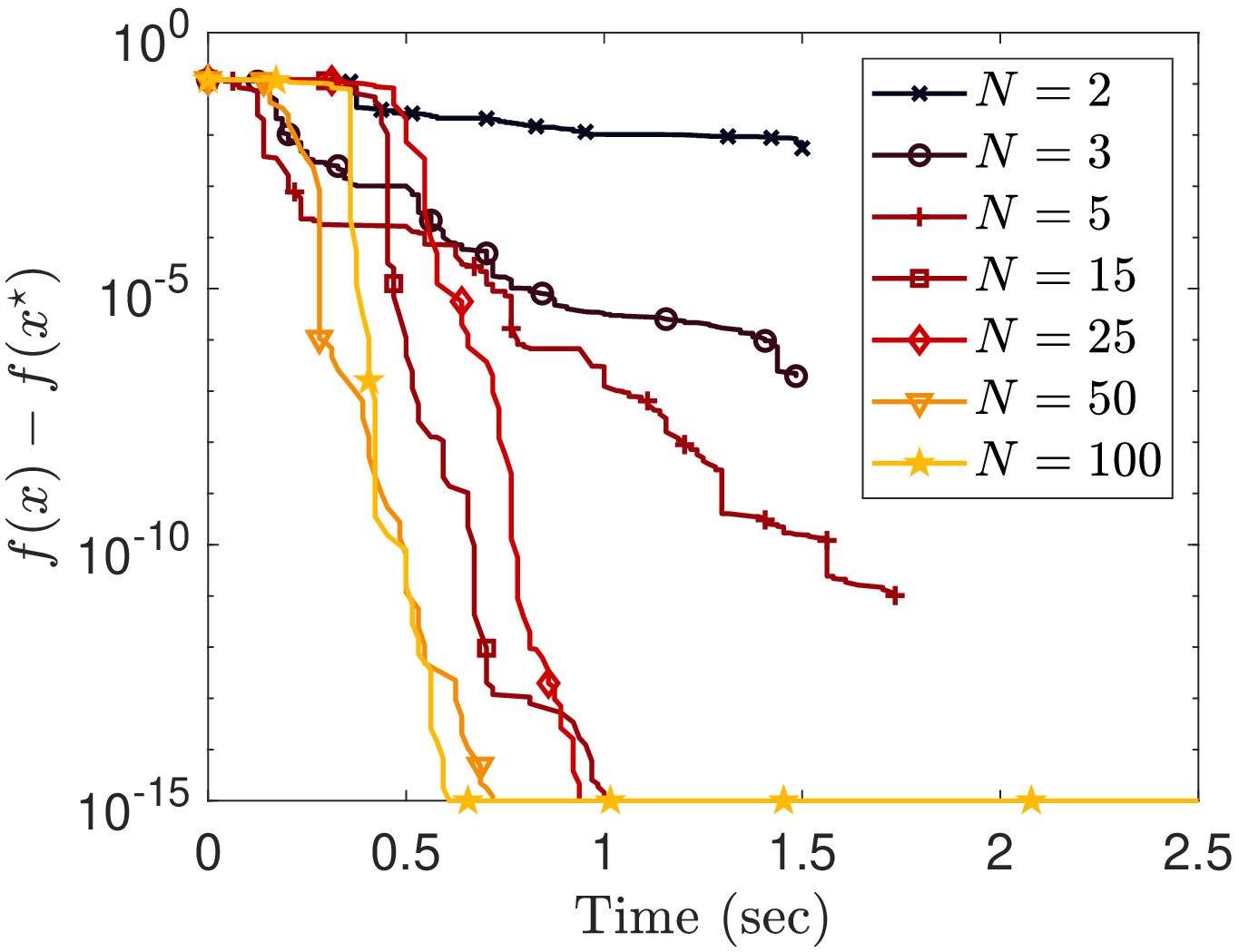}
    \caption{Impact of the memory size $N$ on the convergence rate of the type 1 method with the "Orthogonal forward estimate" update rule to minimize a logistic loss on the Madelon dataset. Left: number of iterations versus suboptimality, right: time versus suboptimality. Overall, it is always better to increase the memory parameter in terms of the number of iterations, but there is an effect of diminishing returns.}
    \label{fig:dependency_N}
\end{figure}

\subsection{Nonconvex optimization}

\begin{figure}[h!t]
    \centering
    \includegraphics[width=0.3\textwidth]{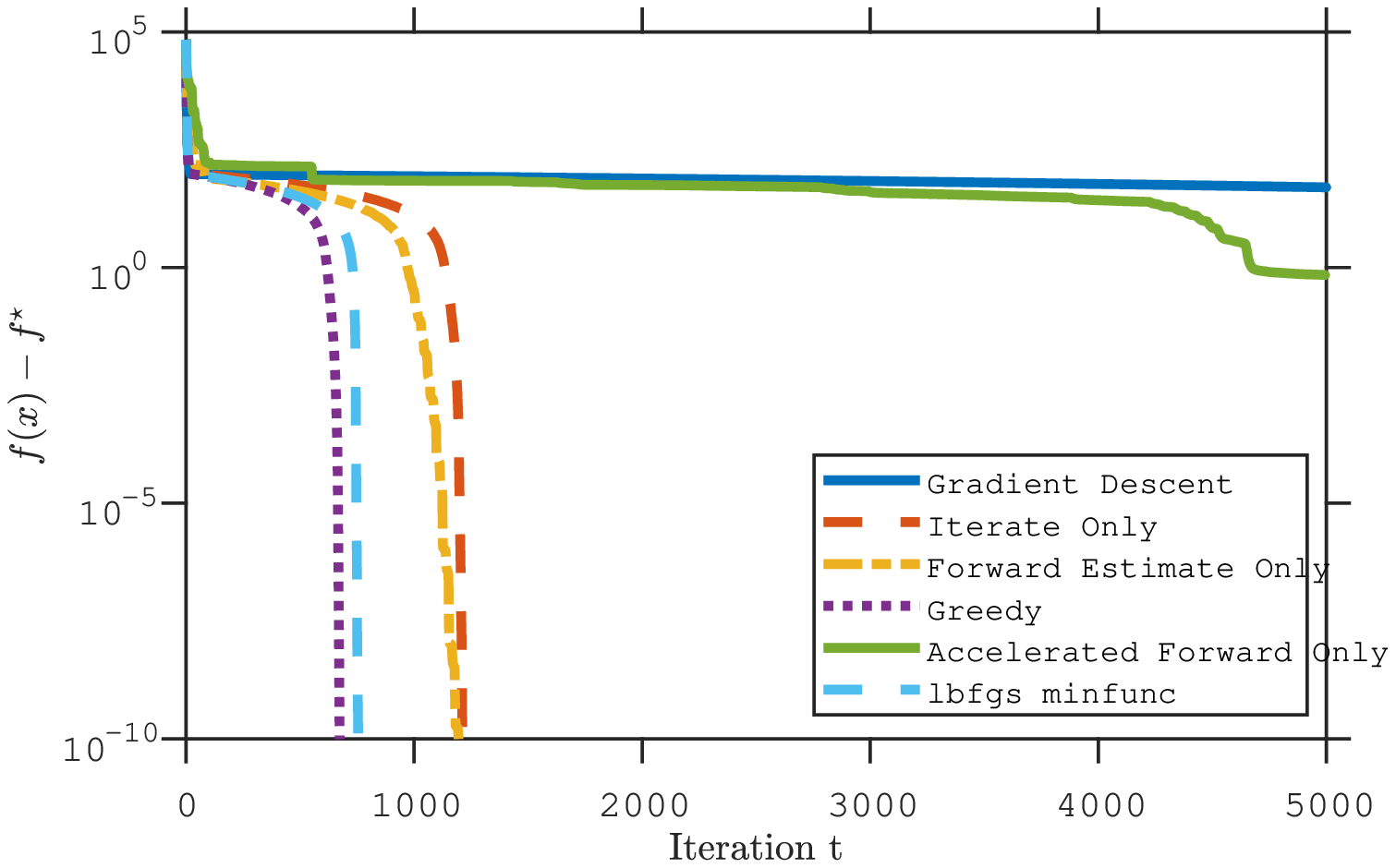}
\includegraphics[width=0.3\textwidth]{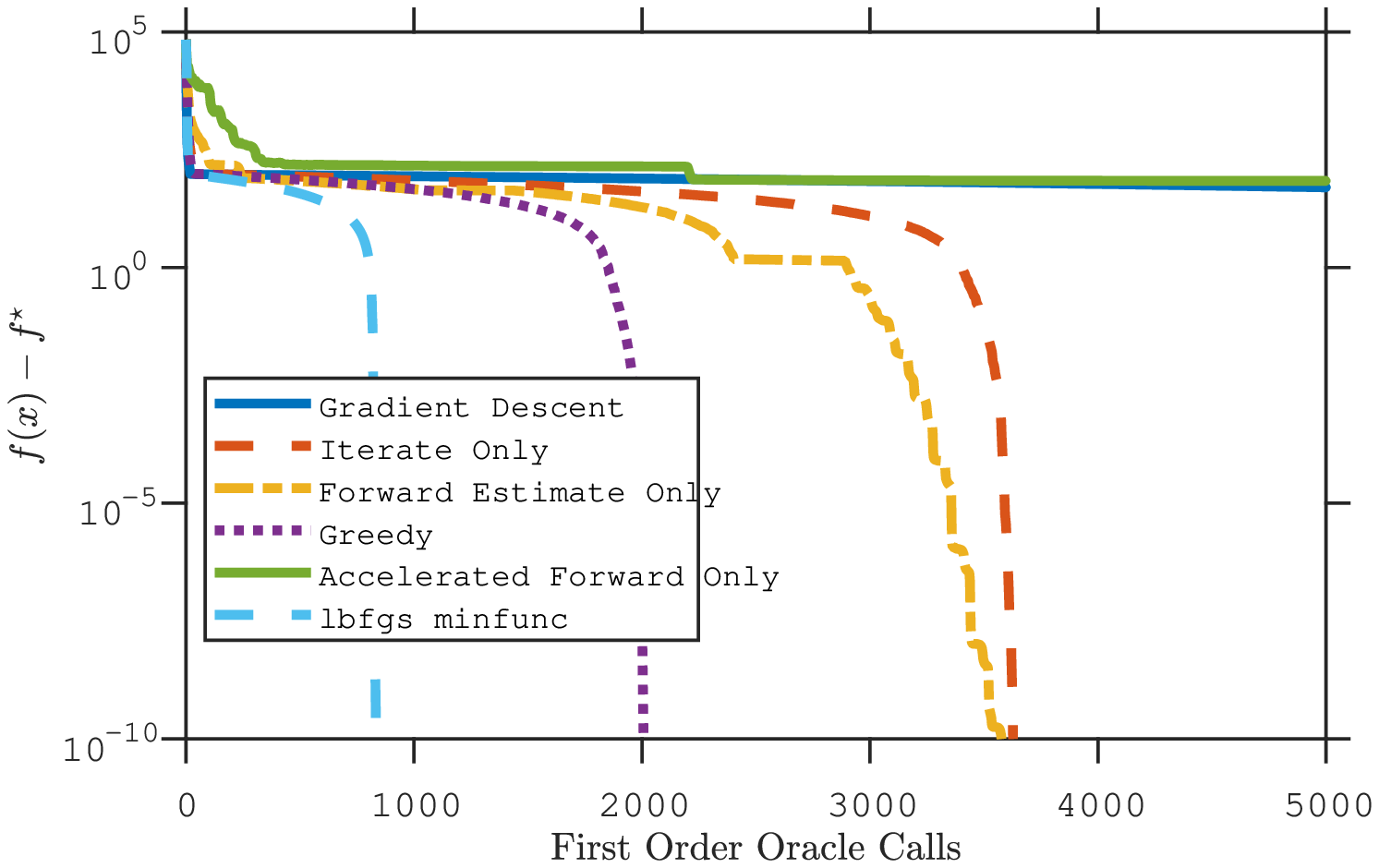}
\includegraphics[width=0.3\textwidth]{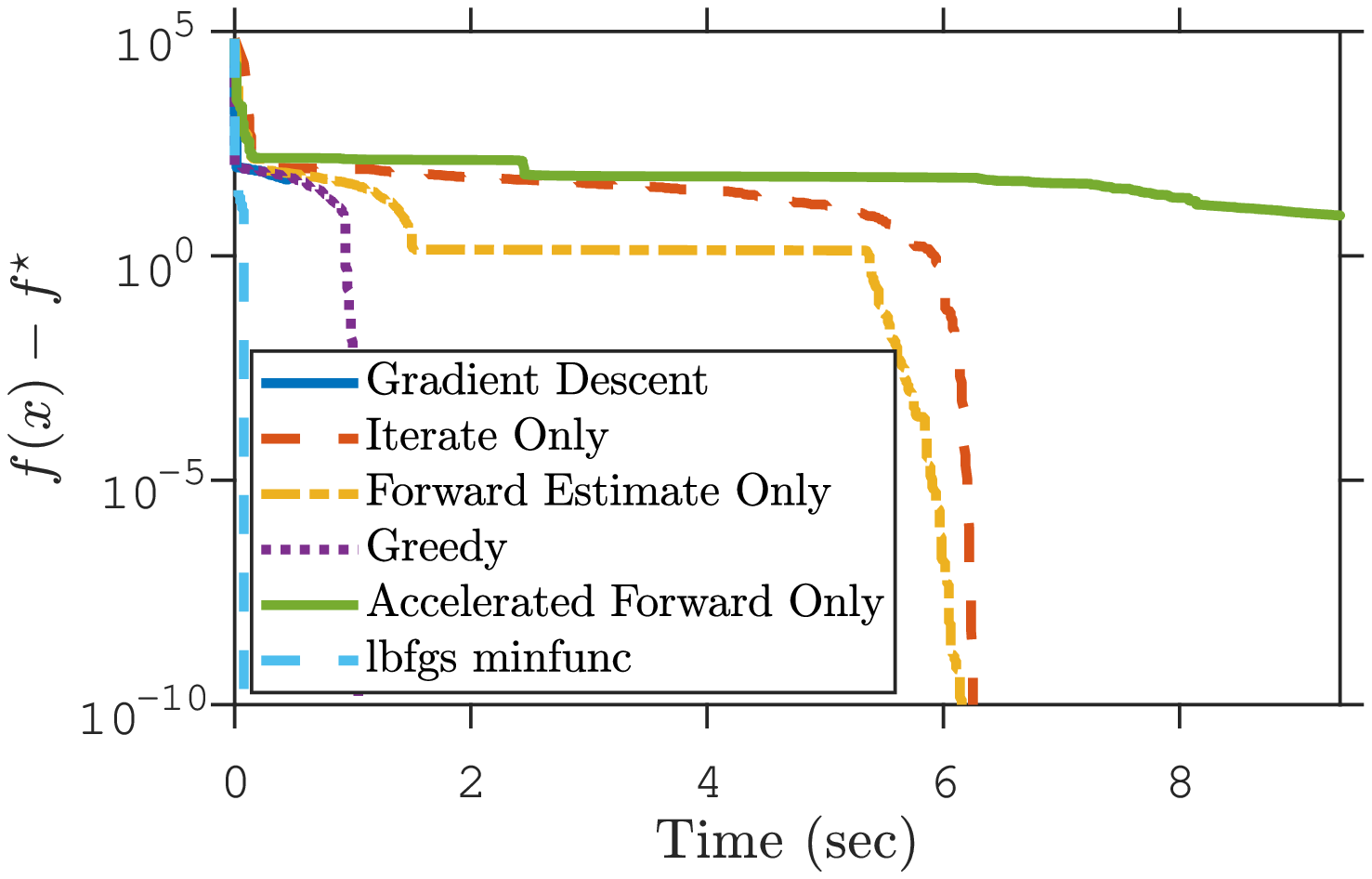}
    \caption{Comparison of type 1 methods on the Generalized Rosenbrock function in $\mathbb
    {R}^{100}$.}
    \label{fig:rosembrock}
\end{figure}

\begin{figure}[h!t]
    \centering
    \includegraphics[width=0.49\textwidth]{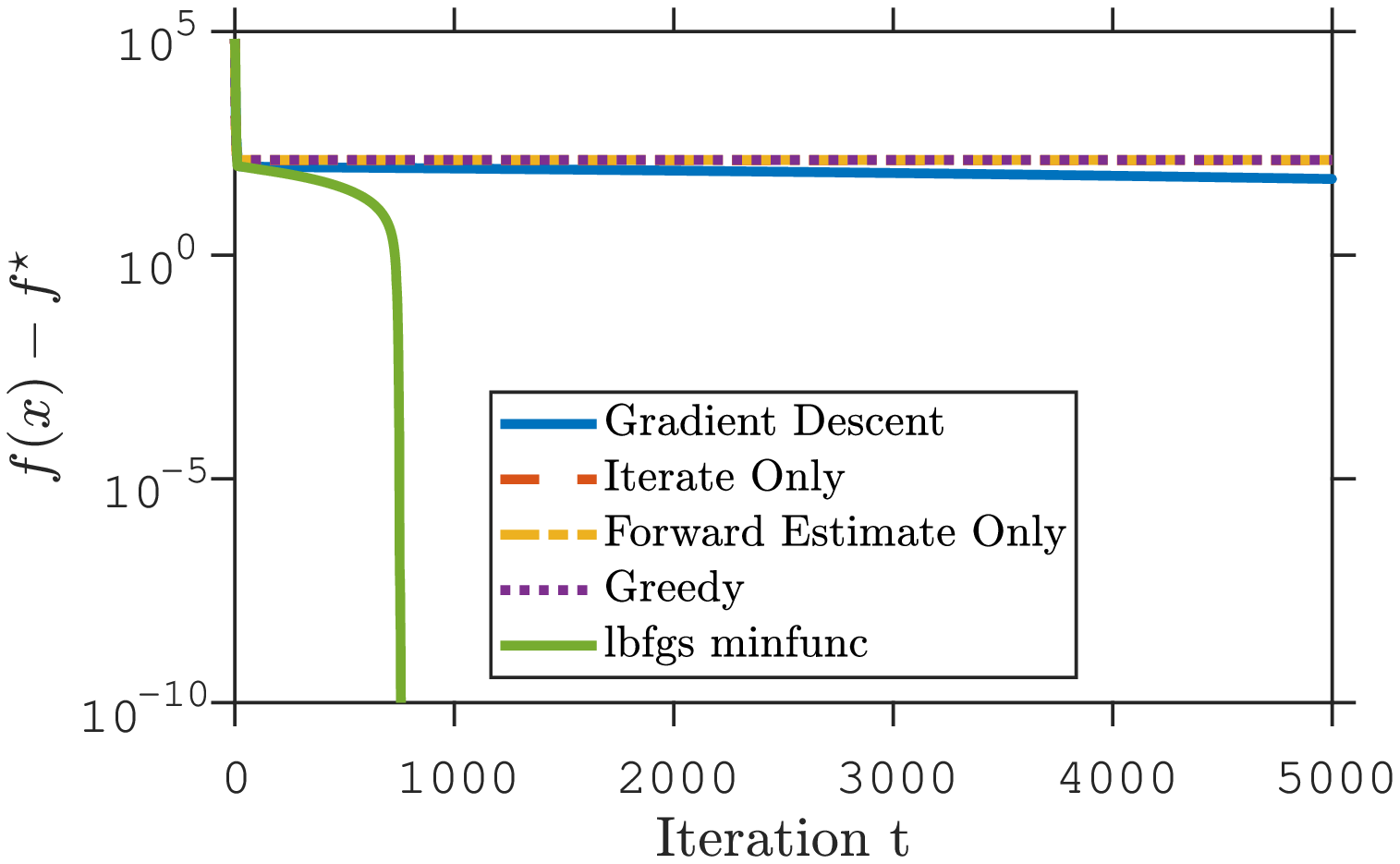}
\includegraphics[width=0.49\textwidth]{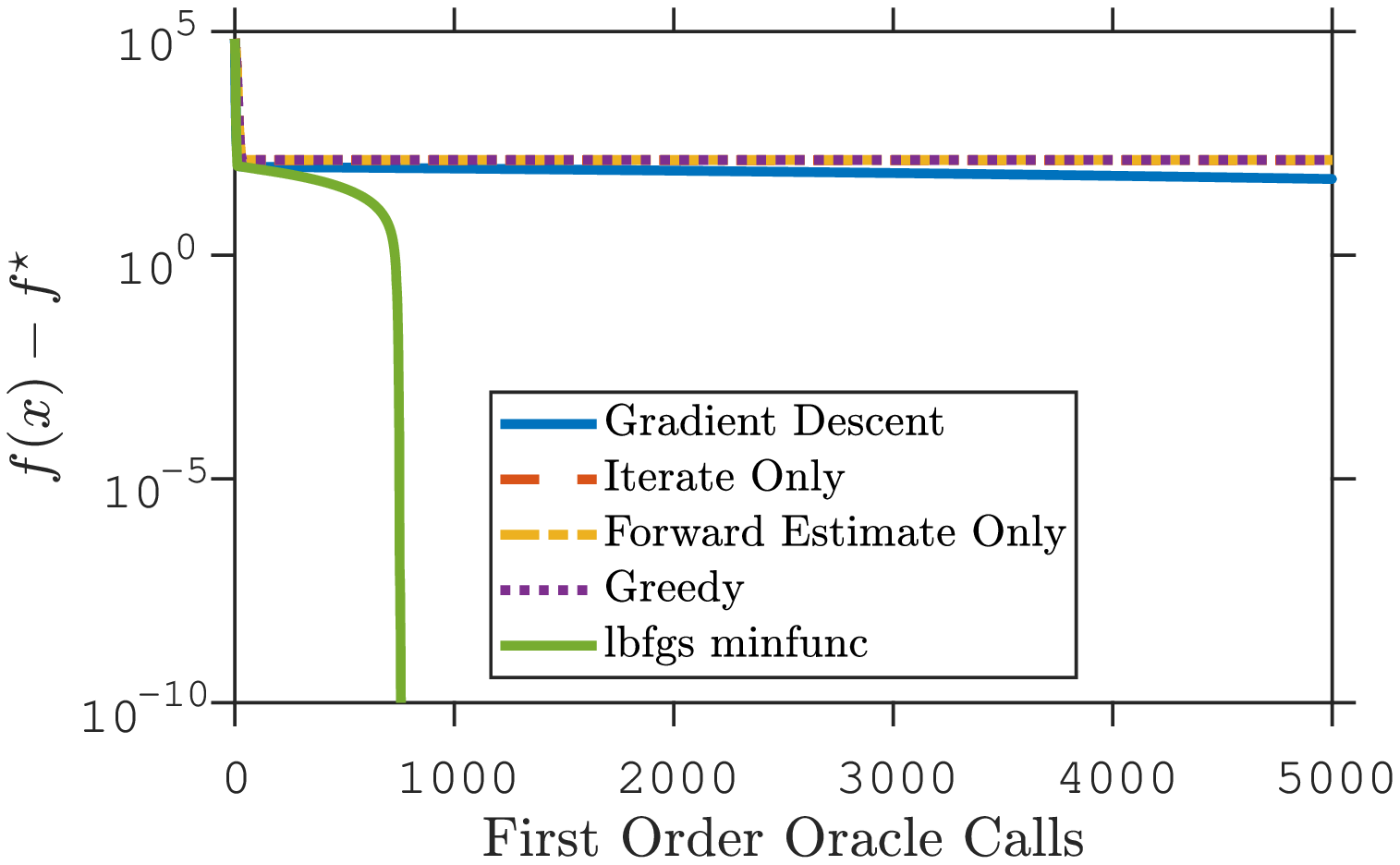}
    \caption{Comparison of type 2 methods on the Generalized Rosenbrock function in $\mathbb
    {R}^{100}$}
    \label{fig:rosembrock_type2}
\end{figure}

\clearpage
\subsection{Comparison of Type 1 Methods on Convex Problems}
\subsubsection{Square loss and cubic regularization}

\begin{figure}[h!t]
    \centering
    \includegraphics[width=0.3\textwidth]{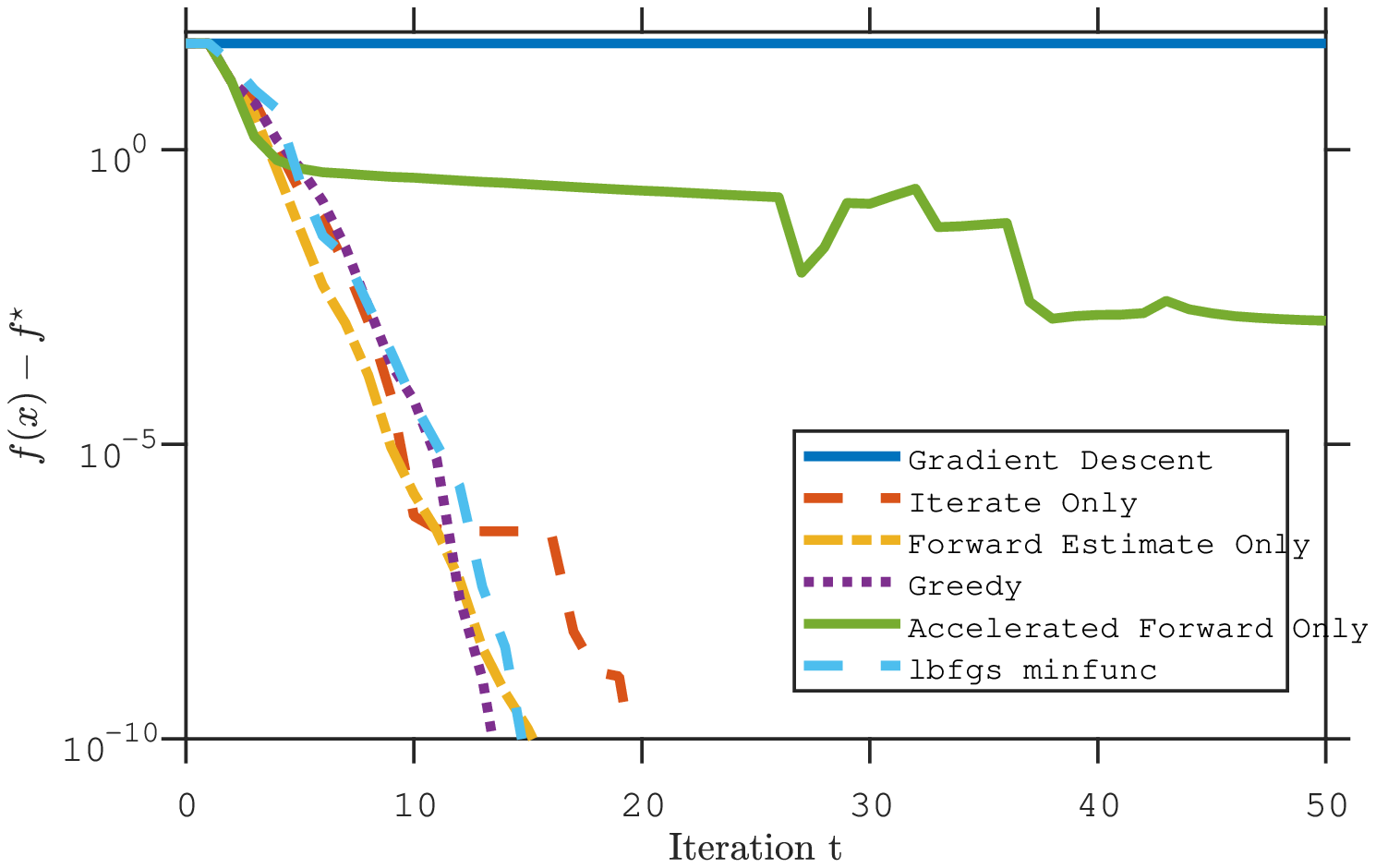}
\includegraphics[width=0.3\textwidth]{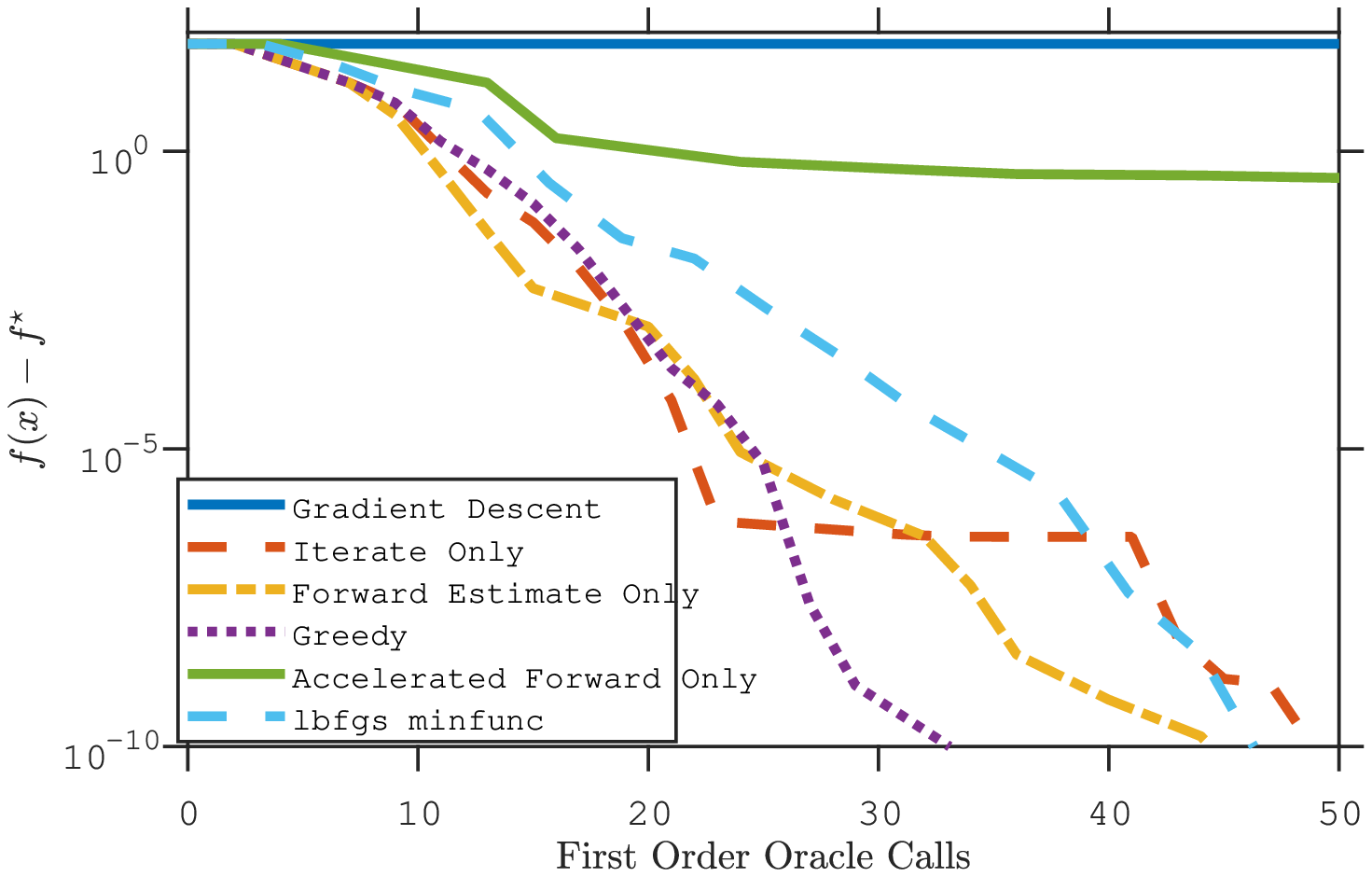}
\includegraphics[width=0.3\textwidth]{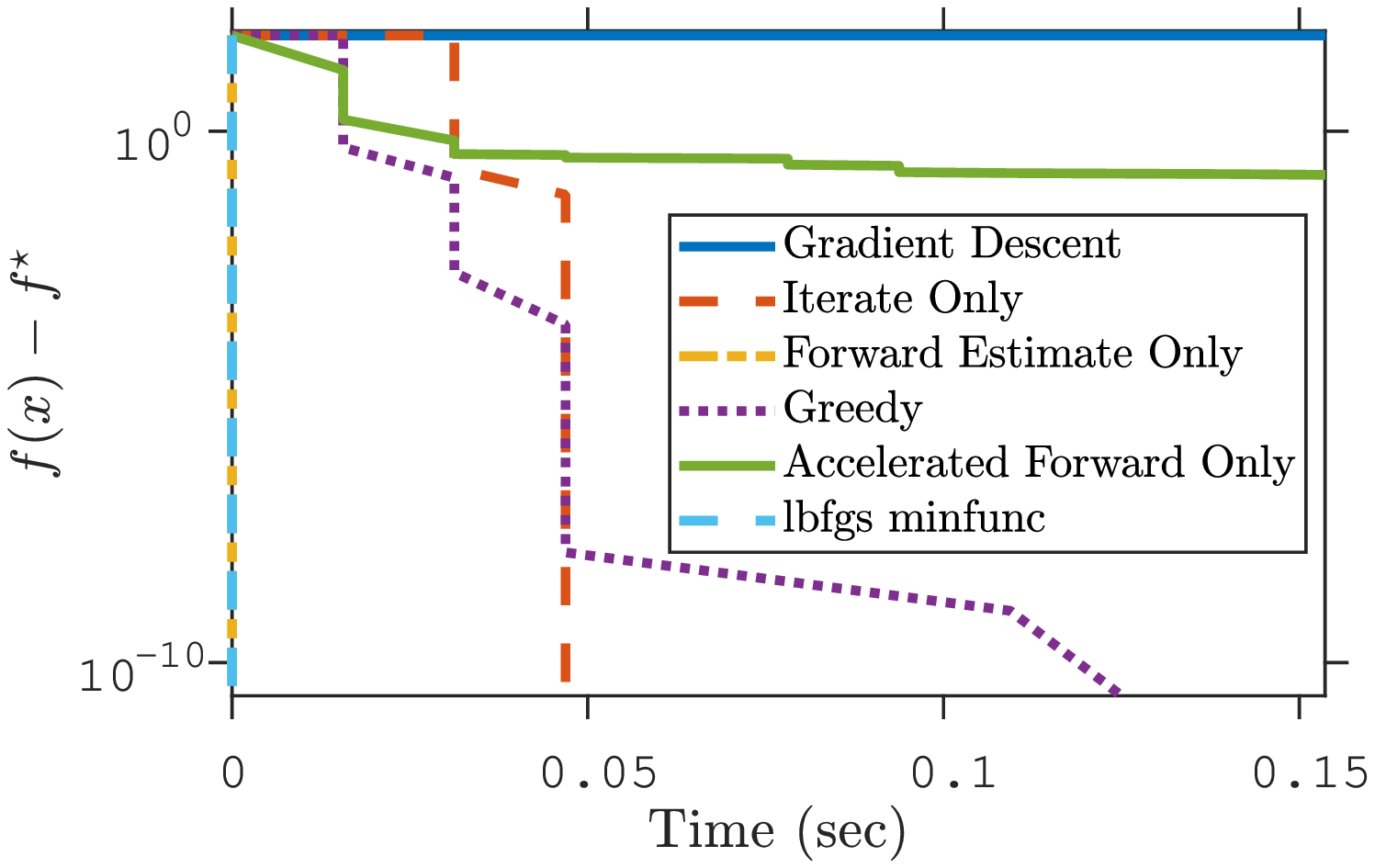}
    \caption{Comparison of type 1 methods: Square loss and cubic regularization on Madelon dataset}
    \label{fig:madelon_quad}
\end{figure}

\begin{figure}[h!t]
    \centering
    \includegraphics[width=0.3\textwidth]{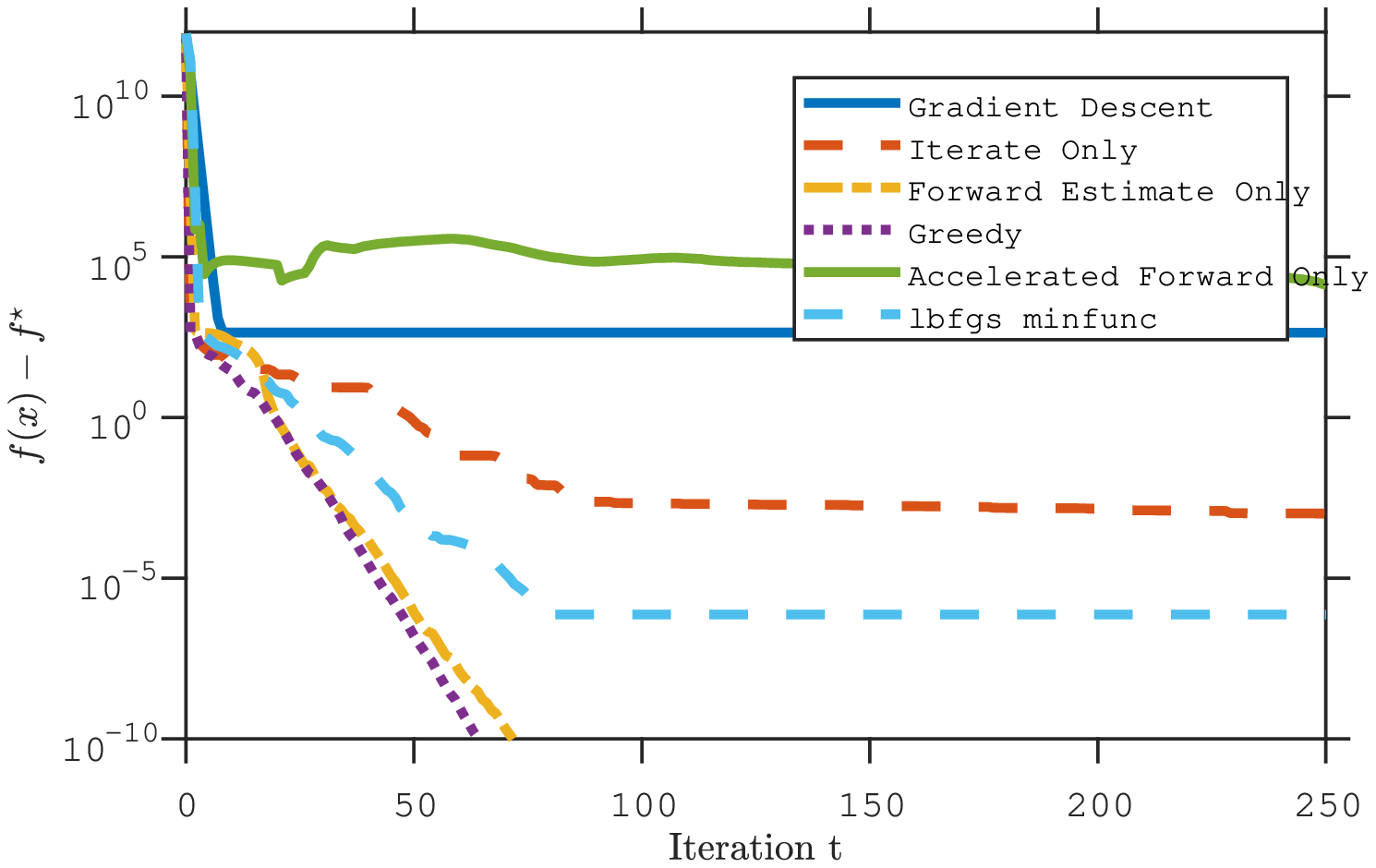}
\includegraphics[width=0.3\textwidth]{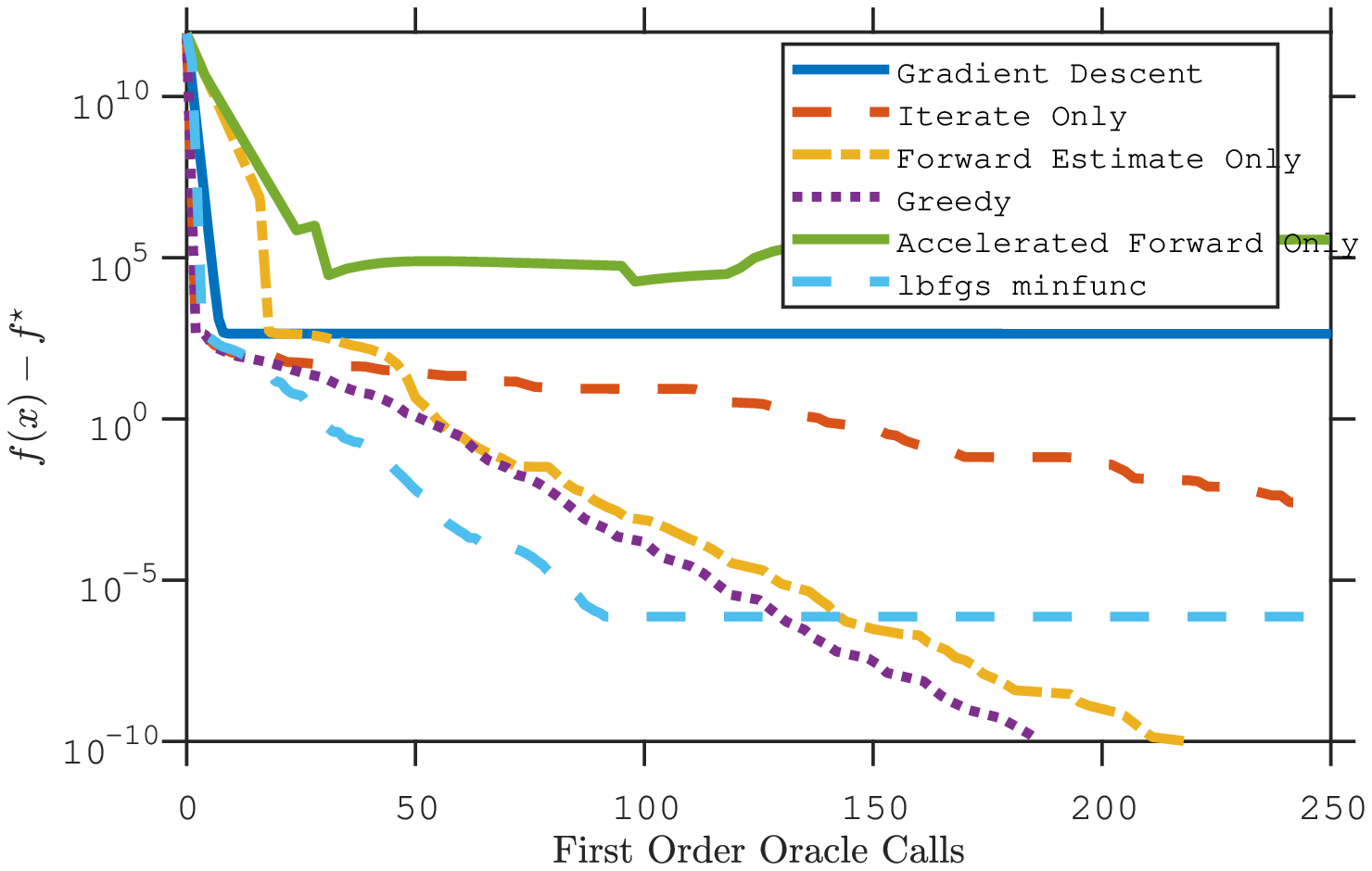}
\includegraphics[width=0.3\textwidth]{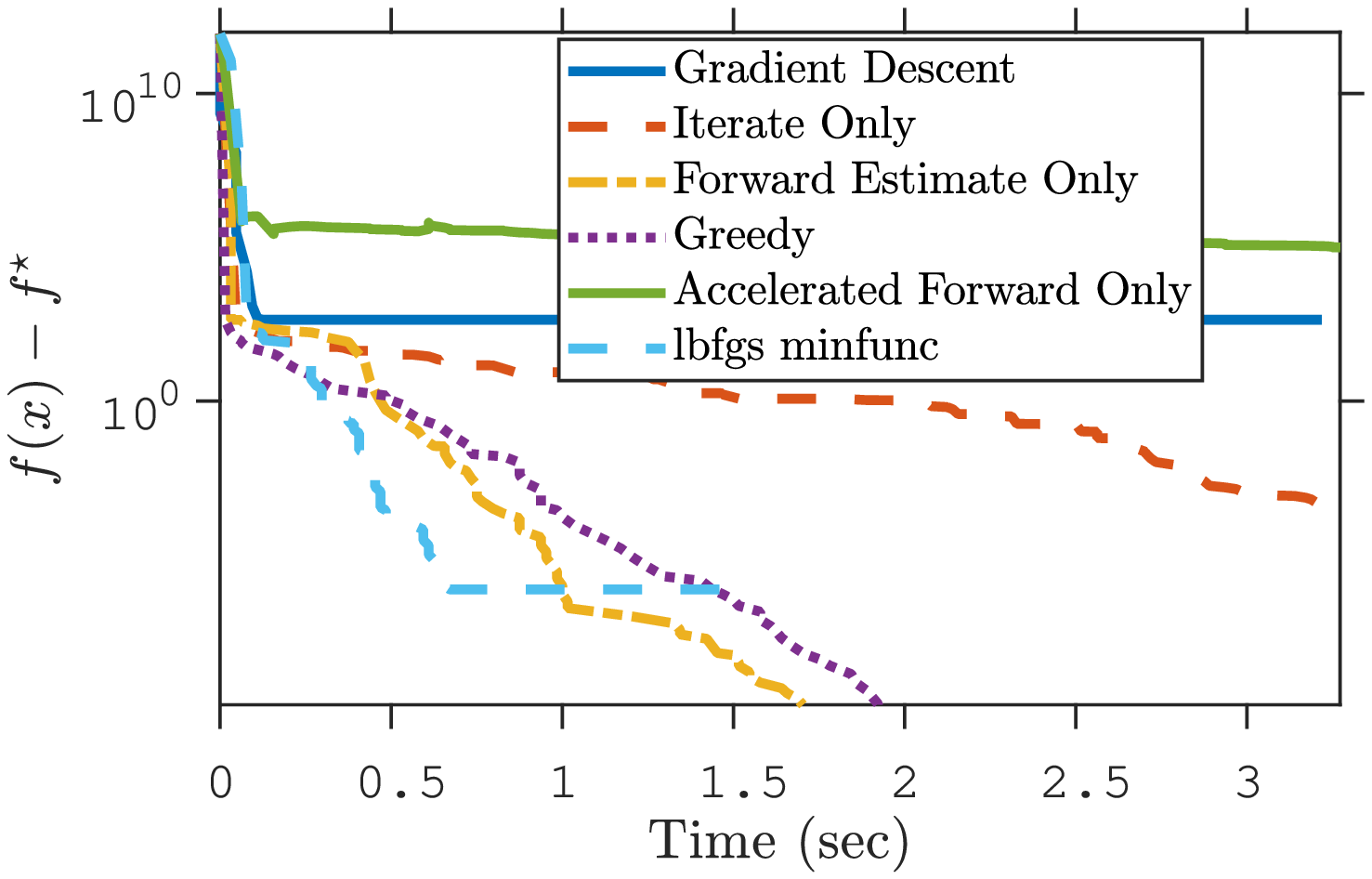}
    \caption{Comparison of type 1 methods: Square loss and cubic regularization on sido0 dataset}
    \label{fig:sido0_quad}
\end{figure}

\begin{figure}[h!t]
    \centering
    \includegraphics[width=0.3\textwidth]{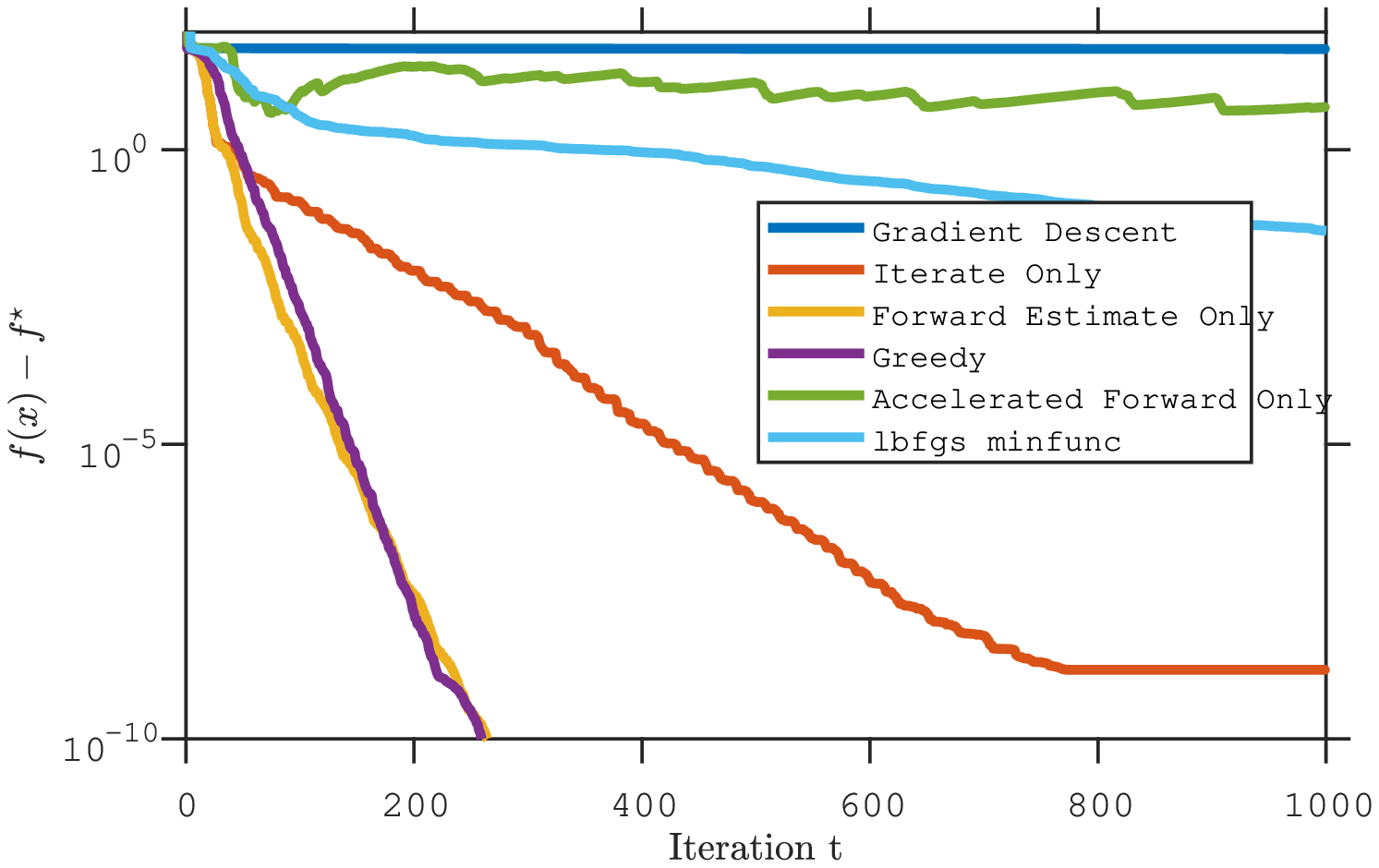}
\includegraphics[width=0.3\textwidth]{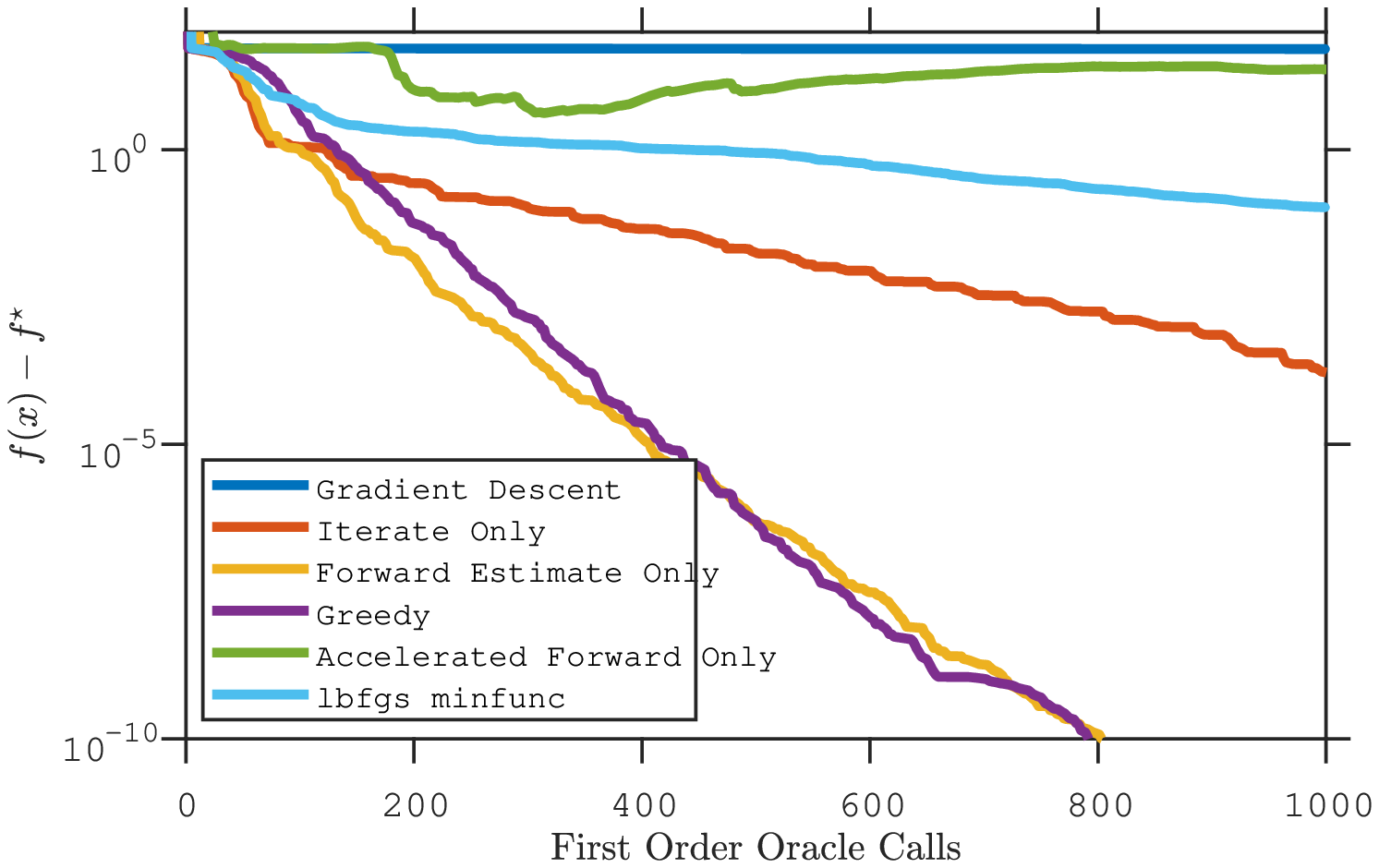}
\includegraphics[width=0.3\textwidth]{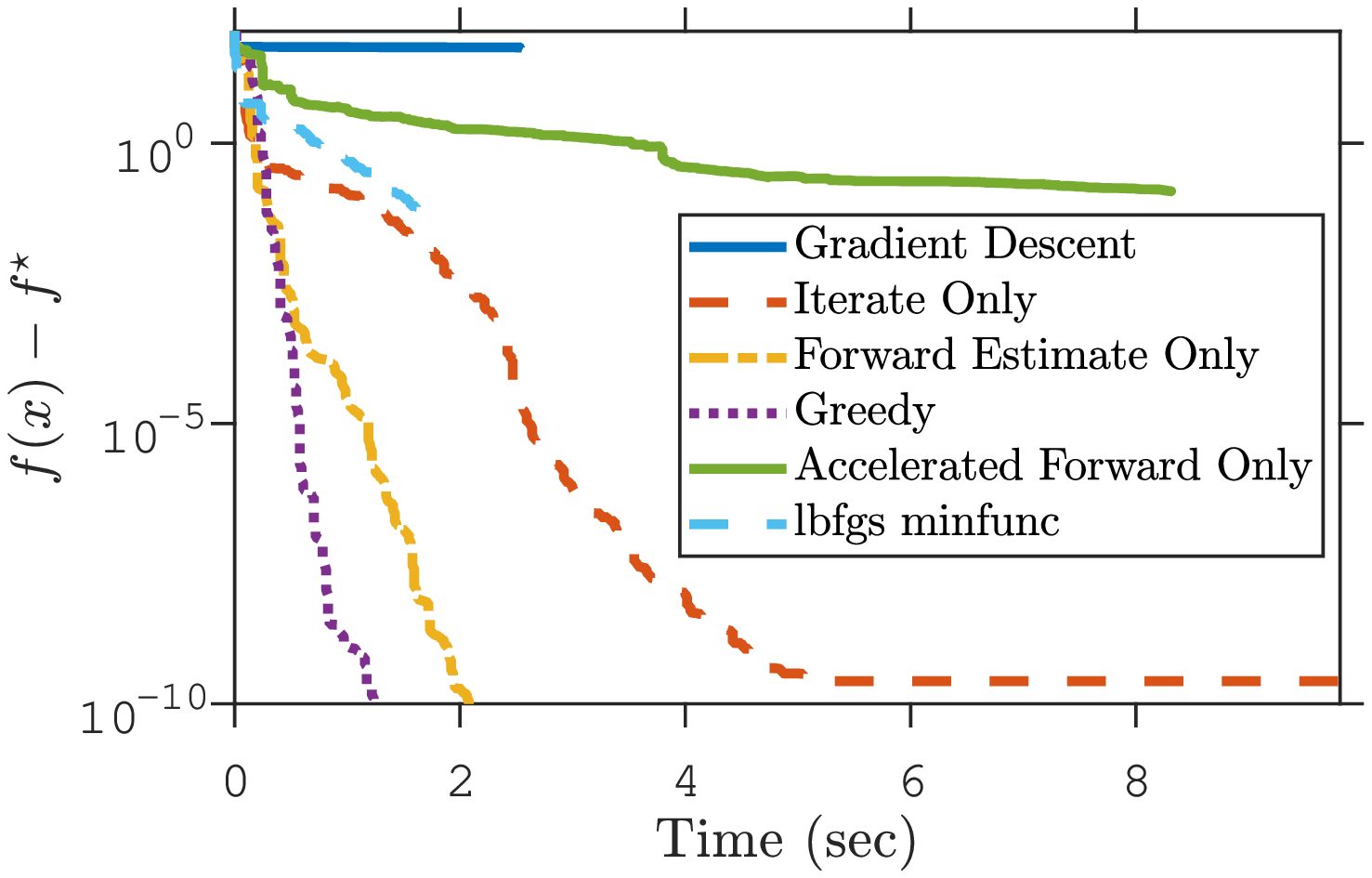}
    \caption{Comparison of type 1 methods: Square loss and cubic regularization on marti2 dataset}
    \label{fig:marti2_quad}
\end{figure}

\clearpage
\subsubsection{Logistic regression}

\begin{figure}[h!t]
    \centering
    \includegraphics[width=0.3\textwidth]{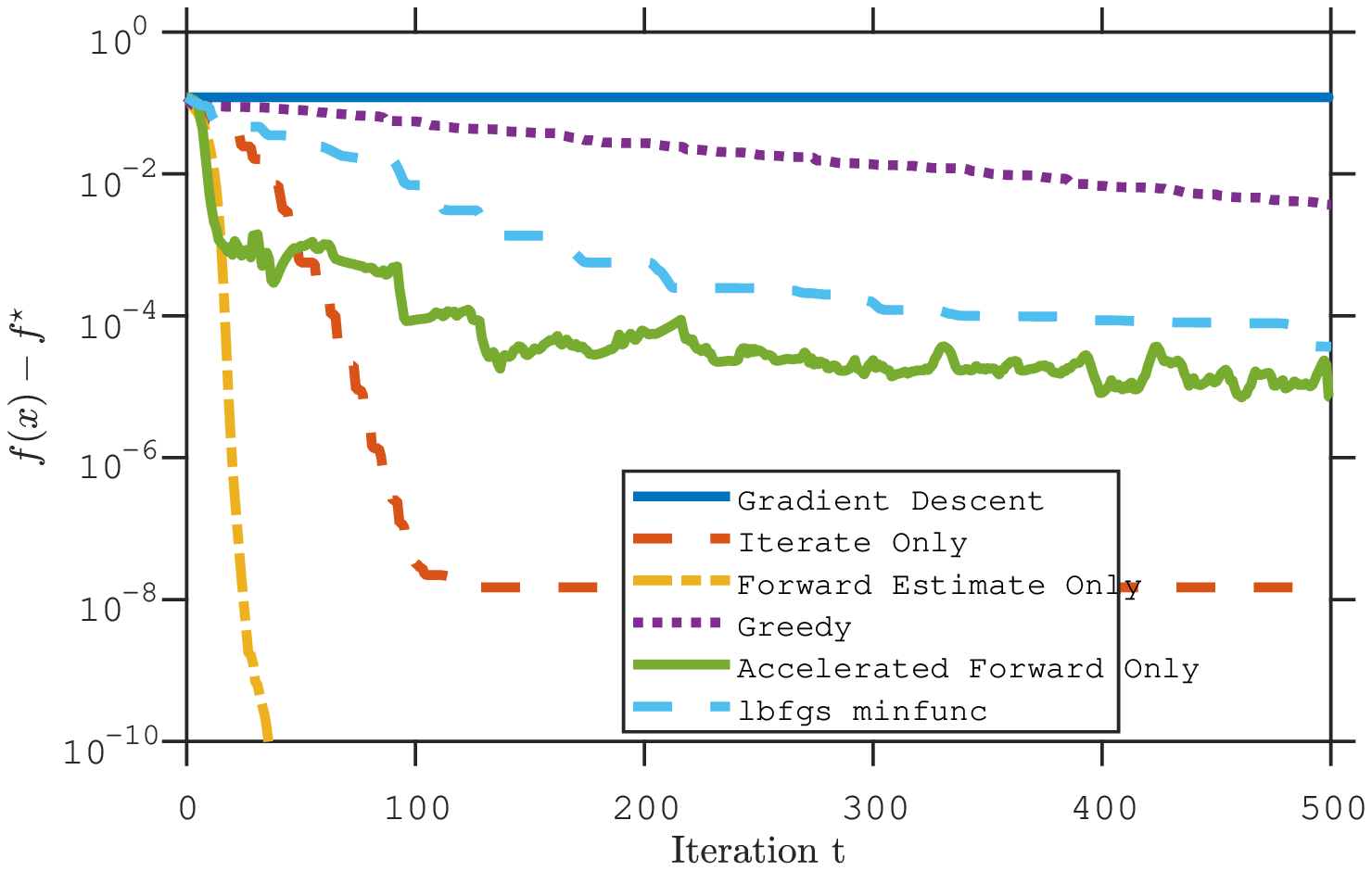}
\includegraphics[width=0.3\textwidth]{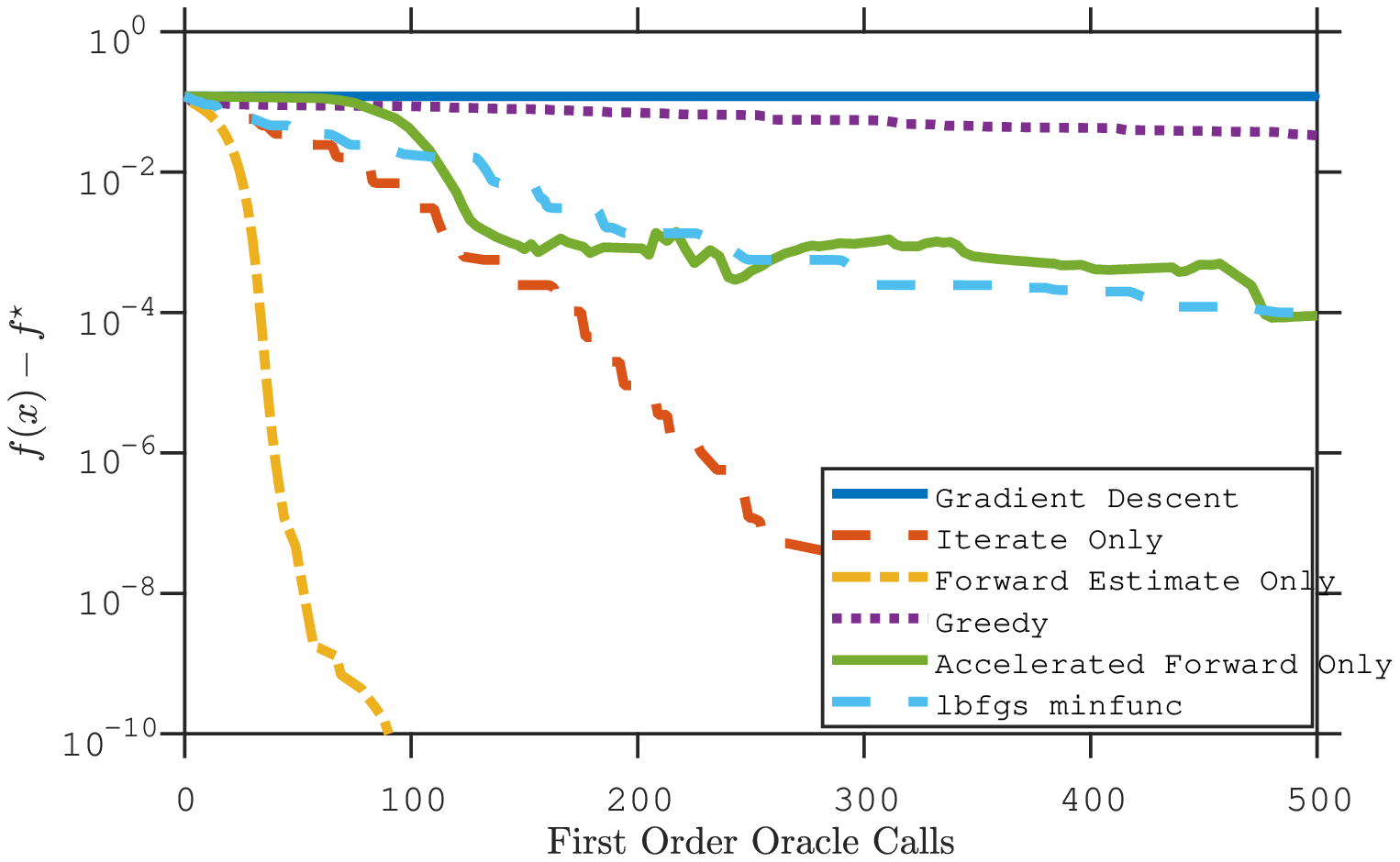}
\includegraphics[width=0.3\textwidth]{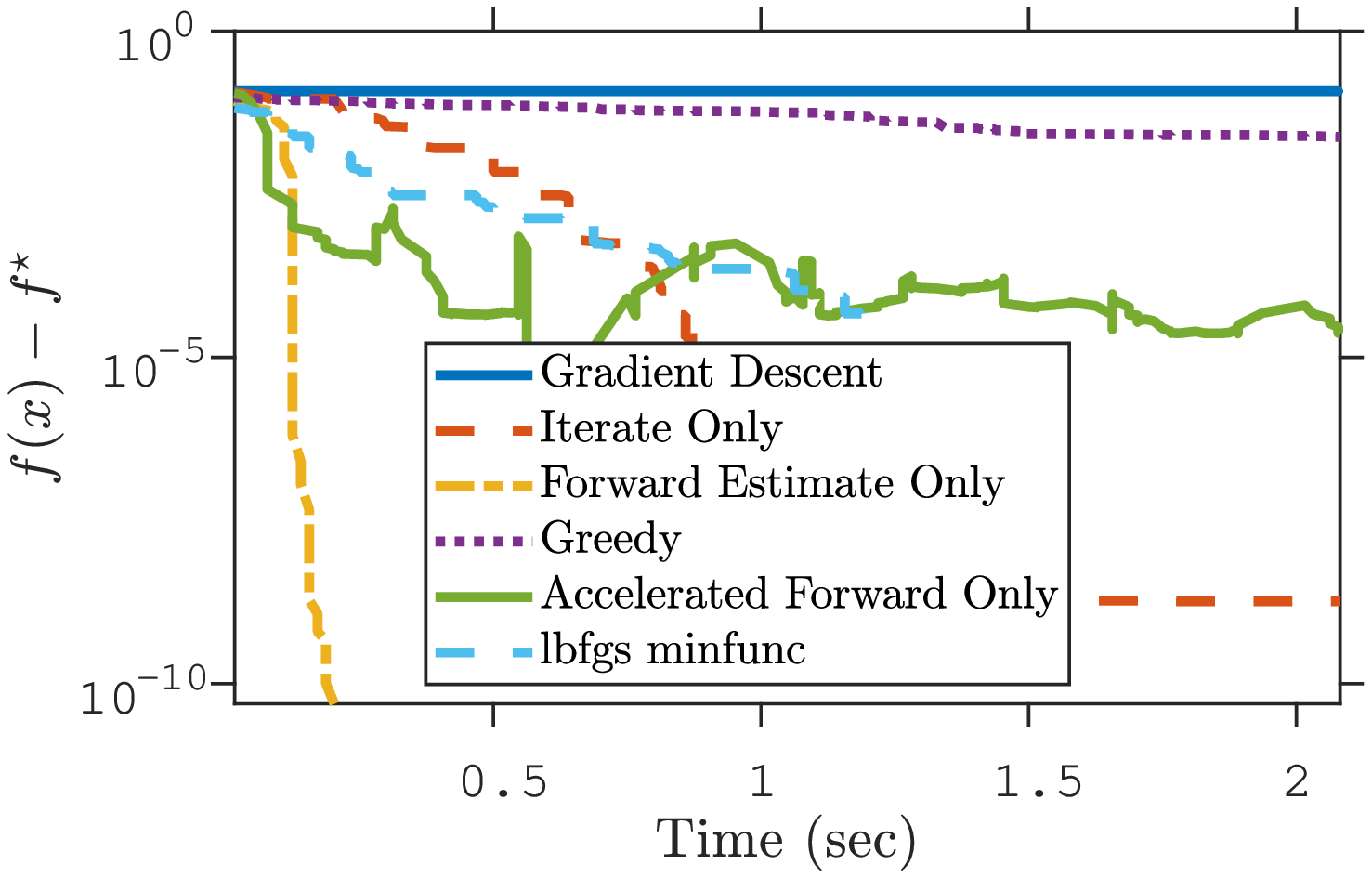}
    \caption{Comparison of type 1 methods: Logistic loss and cubic regularization on Madelon dataset}
    \label{fig:madelon_logistic}
\end{figure}

\begin{figure}[h!t]
    \centering
    \includegraphics[width=0.3\textwidth]{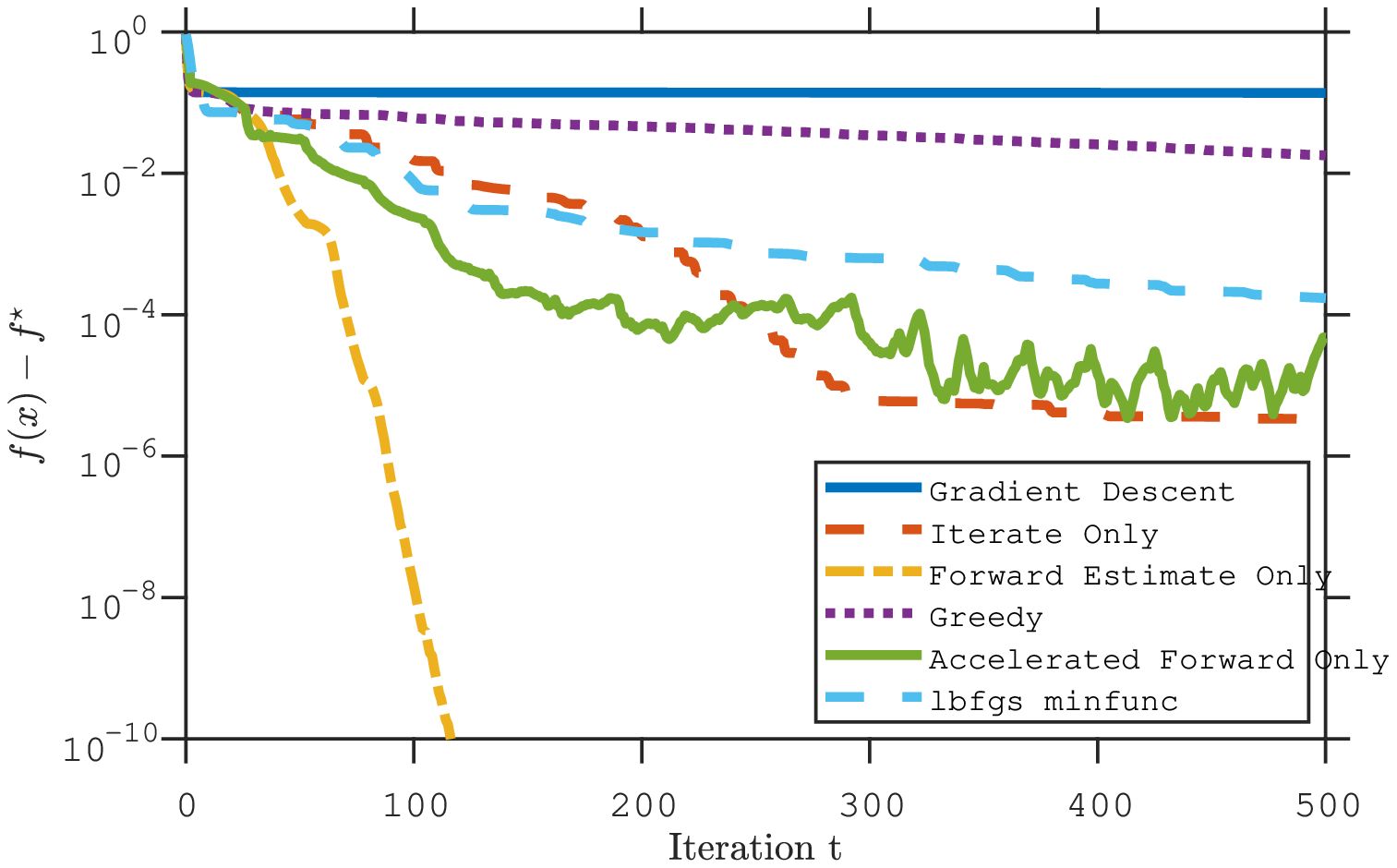}
\includegraphics[width=0.3\textwidth]{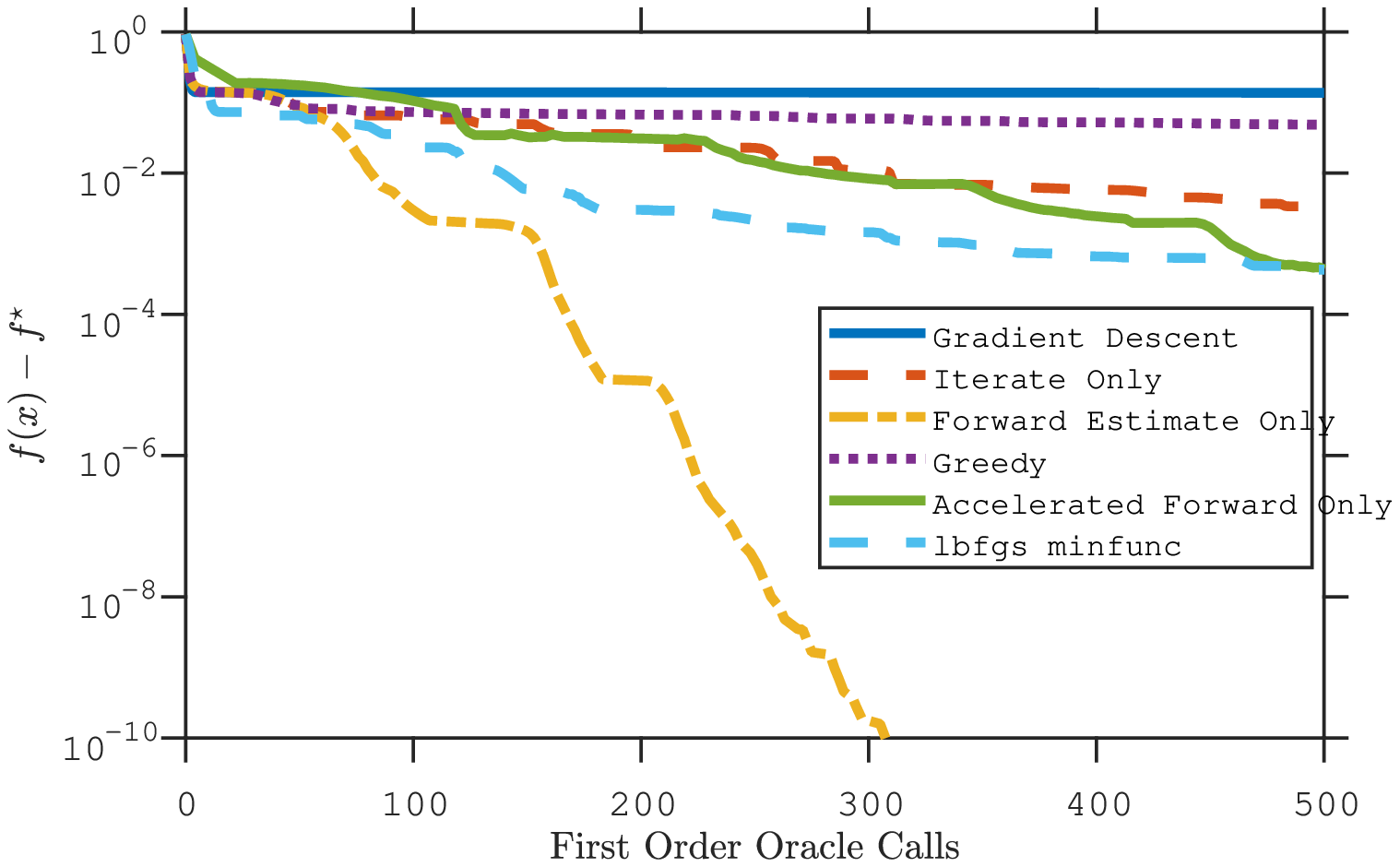}
\includegraphics[width=0.3\textwidth]{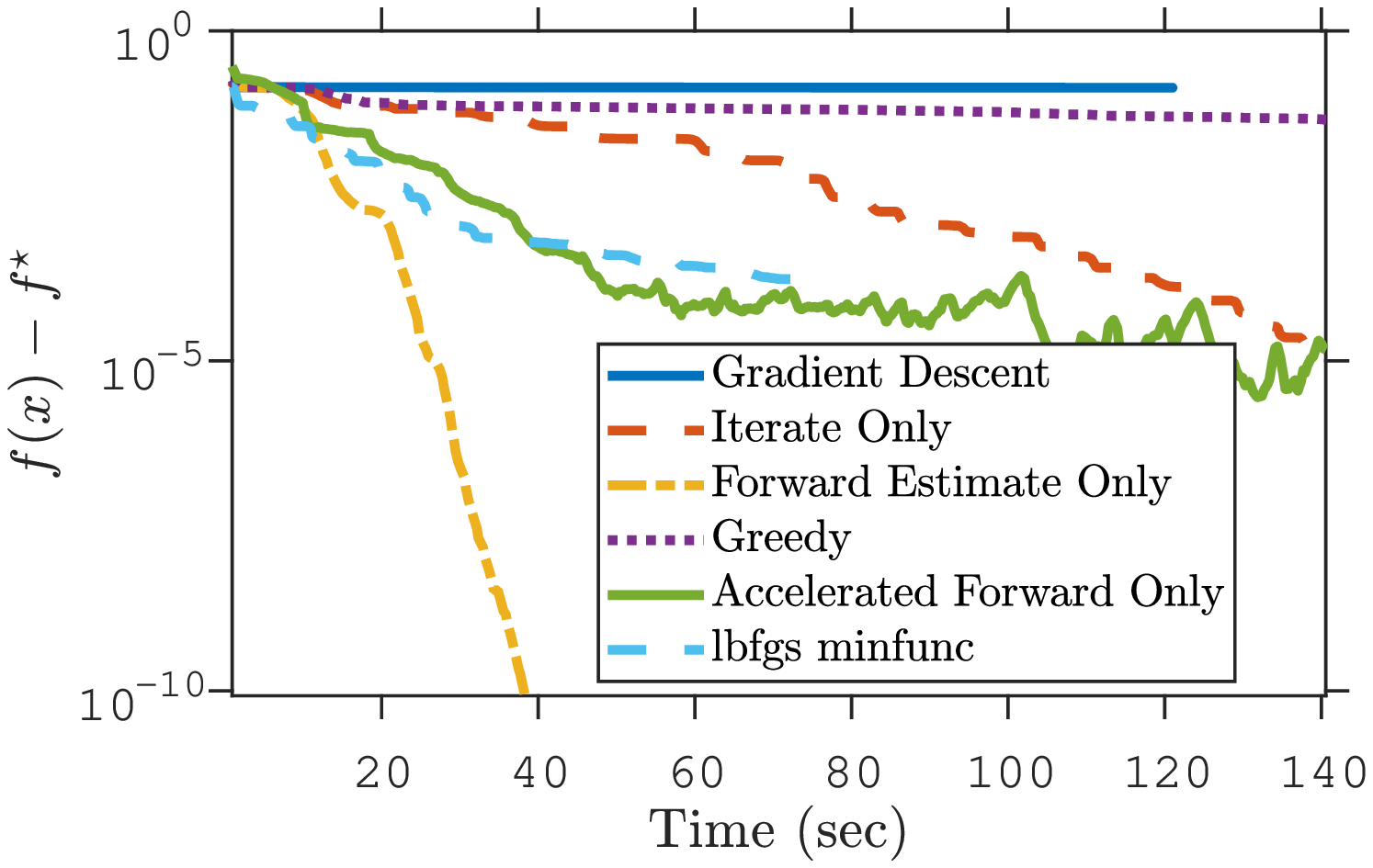}
    \caption{Comparison of type 1 methods: Logistic loss and cubic regularization on sido0 dataset}
    \label{fig:sido0_logistic}
\end{figure}

\begin{figure}[h!t]
    \centering
    \includegraphics[width=0.3\textwidth]{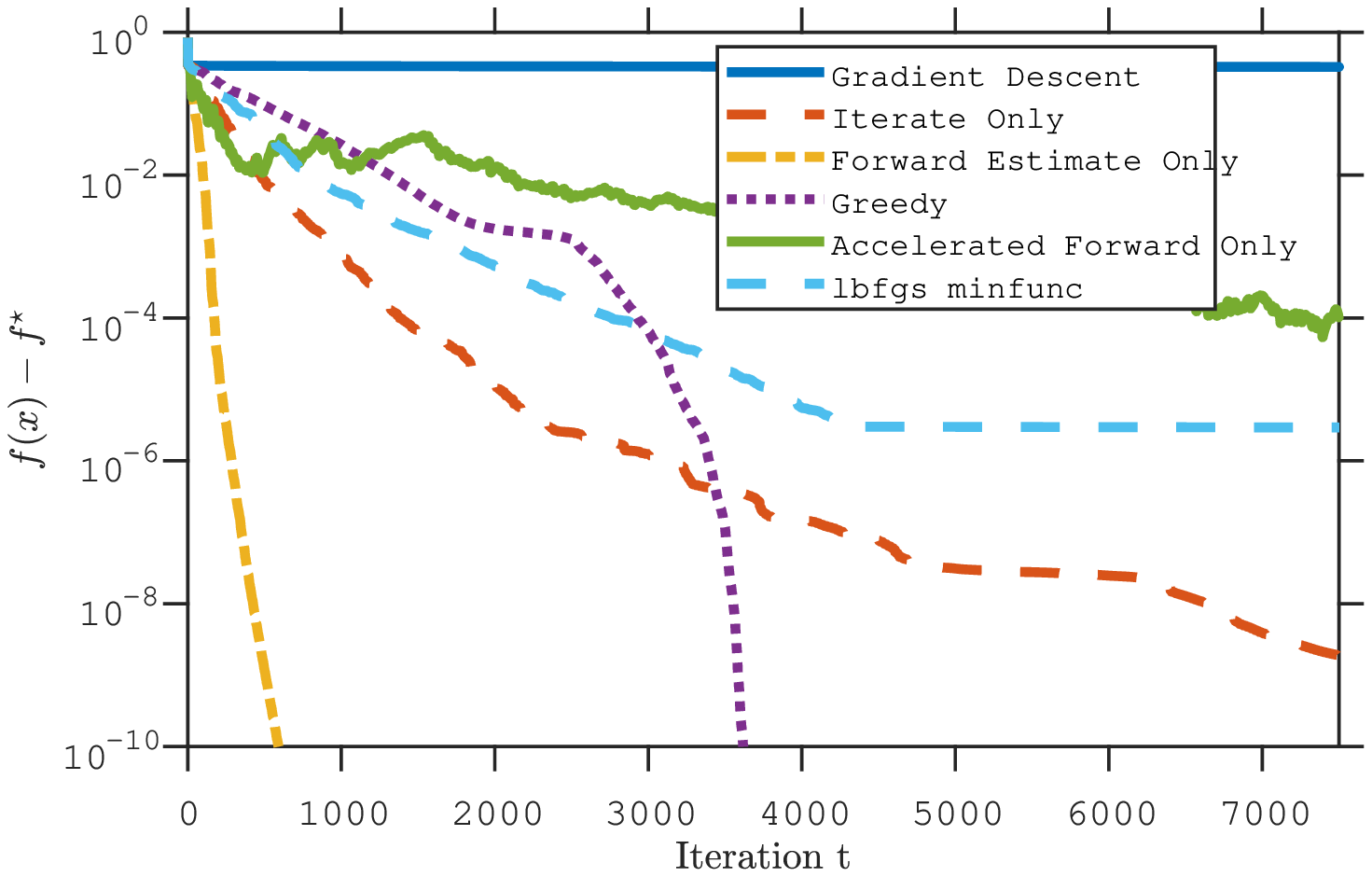}
\includegraphics[width=0.3\textwidth]{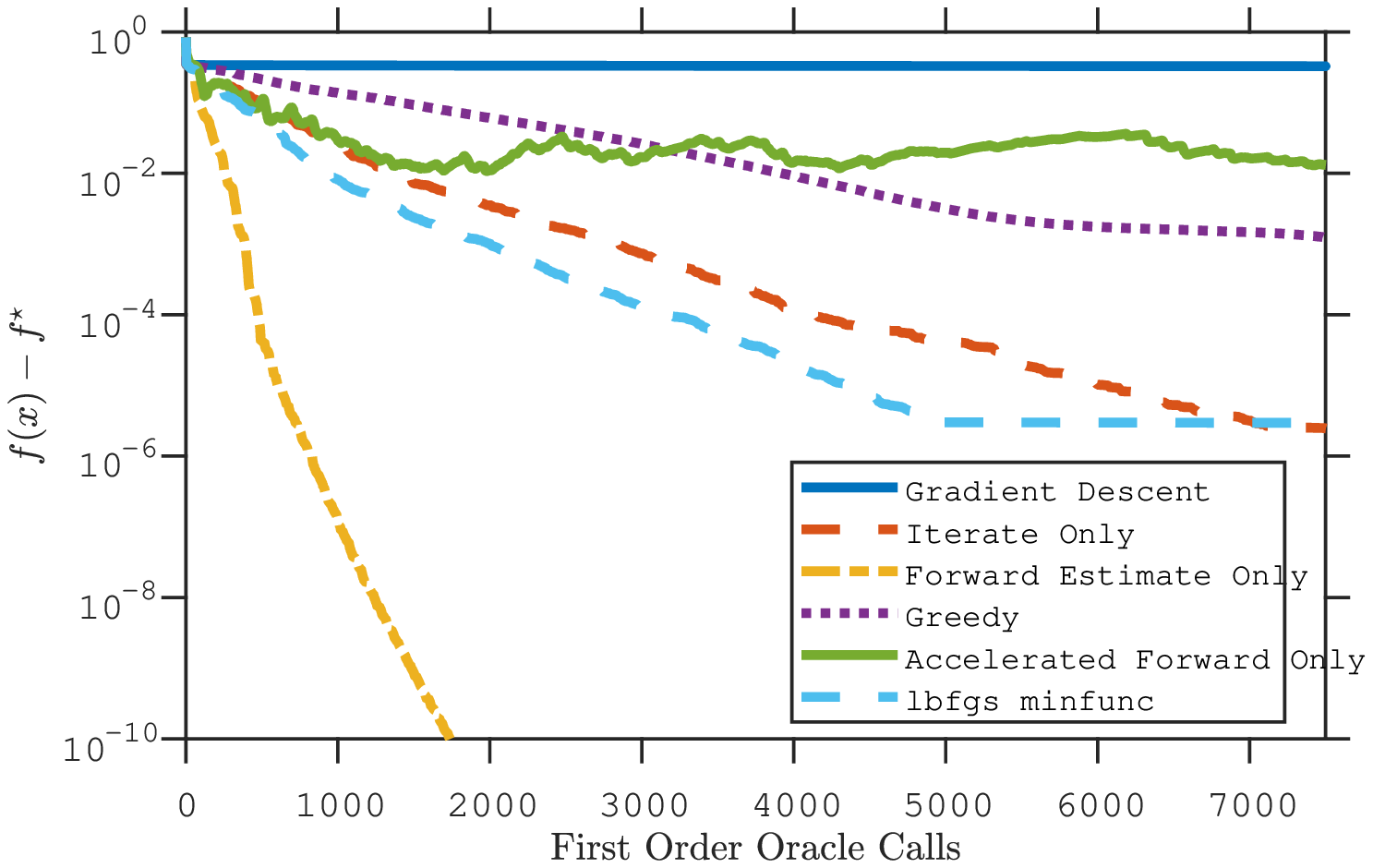}
\includegraphics[width=0.3\textwidth]{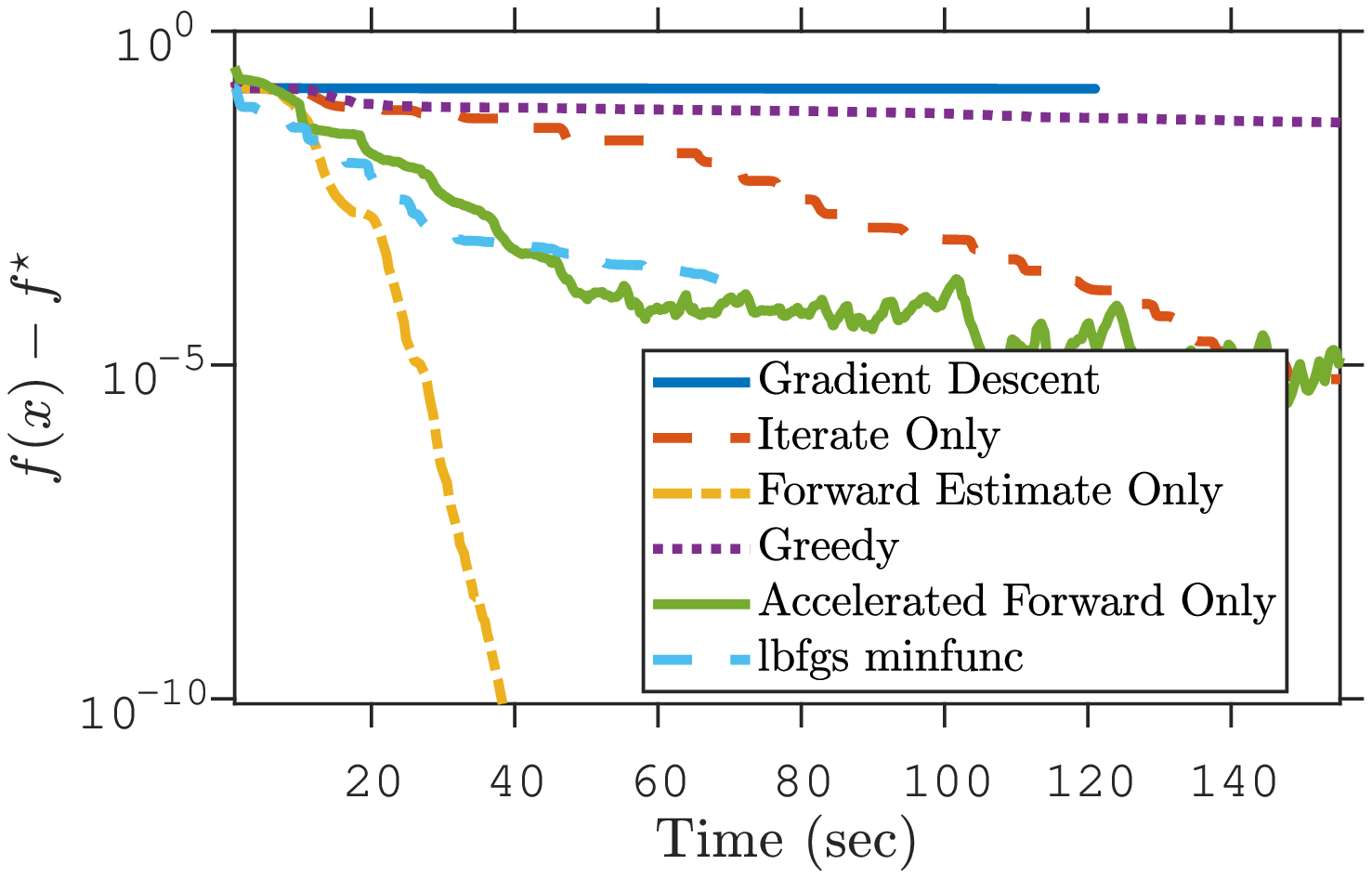}
    \caption{Comparison of type 1 methods: Logistic loss and cubic regularization on marti2 dataset}
    \label{fig:marti2_logistic}
\end{figure}

\clearpage
\subsection{Comparison of Type 2 Methods on Convex Problems}

The type-2 method was not the focus of this study. Its prototypical implementation is rather slow, hence, the time VS suboptimality graph are not showed.

\subsubsection{Square loss and cubic regularization}

\begin{figure}[h!t]
    \centering
    \includegraphics[width=0.49\textwidth]{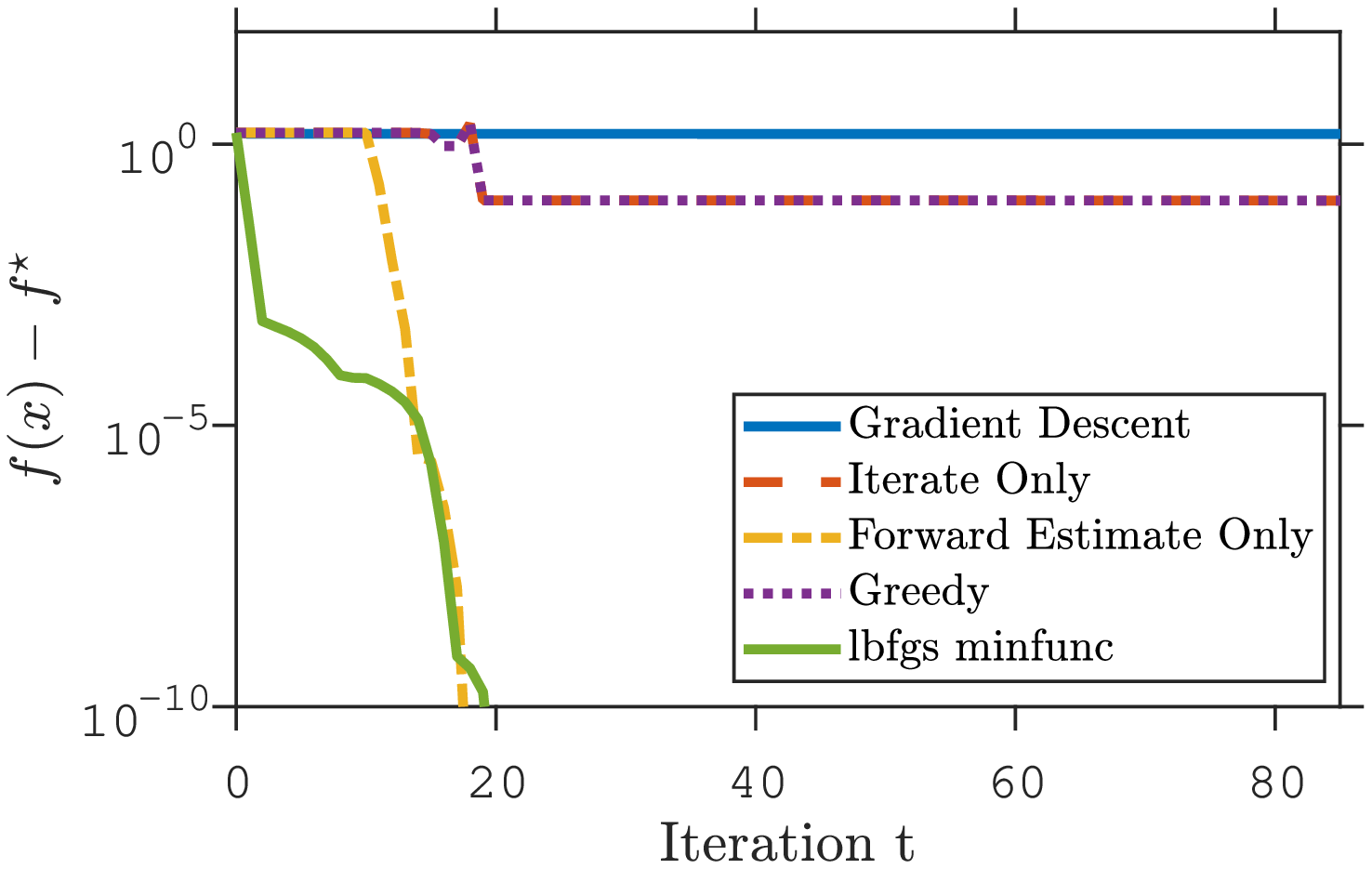}
\includegraphics[width=0.49\textwidth]{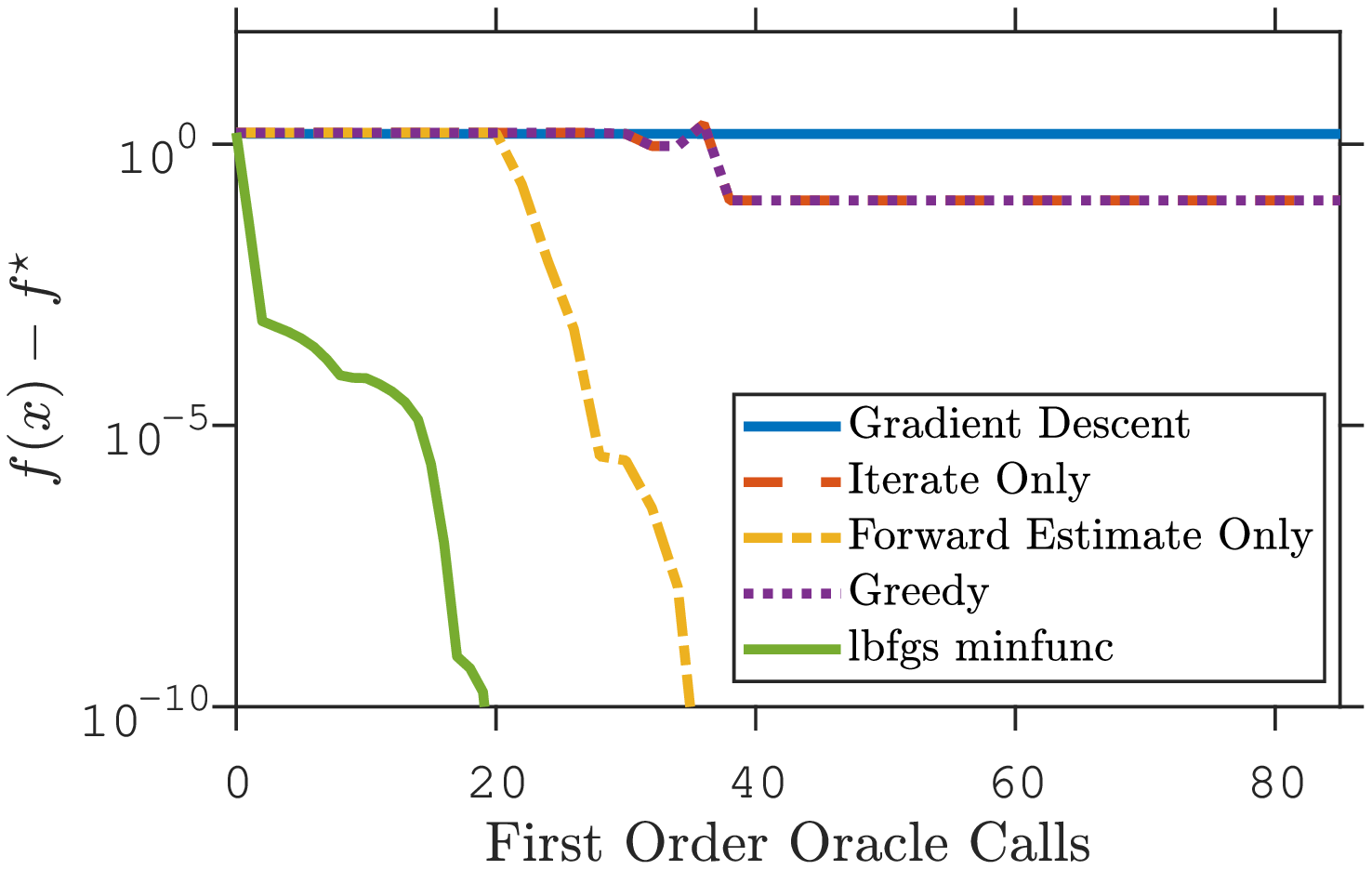}
    \caption{Comparison of type 2 methods: Square loss and cubic regularization on Madelon dataset}
    \label{fig:madelon_quad_type2}
\end{figure}

\begin{figure}[h!t]
    \centering
    \includegraphics[width=0.49\textwidth]{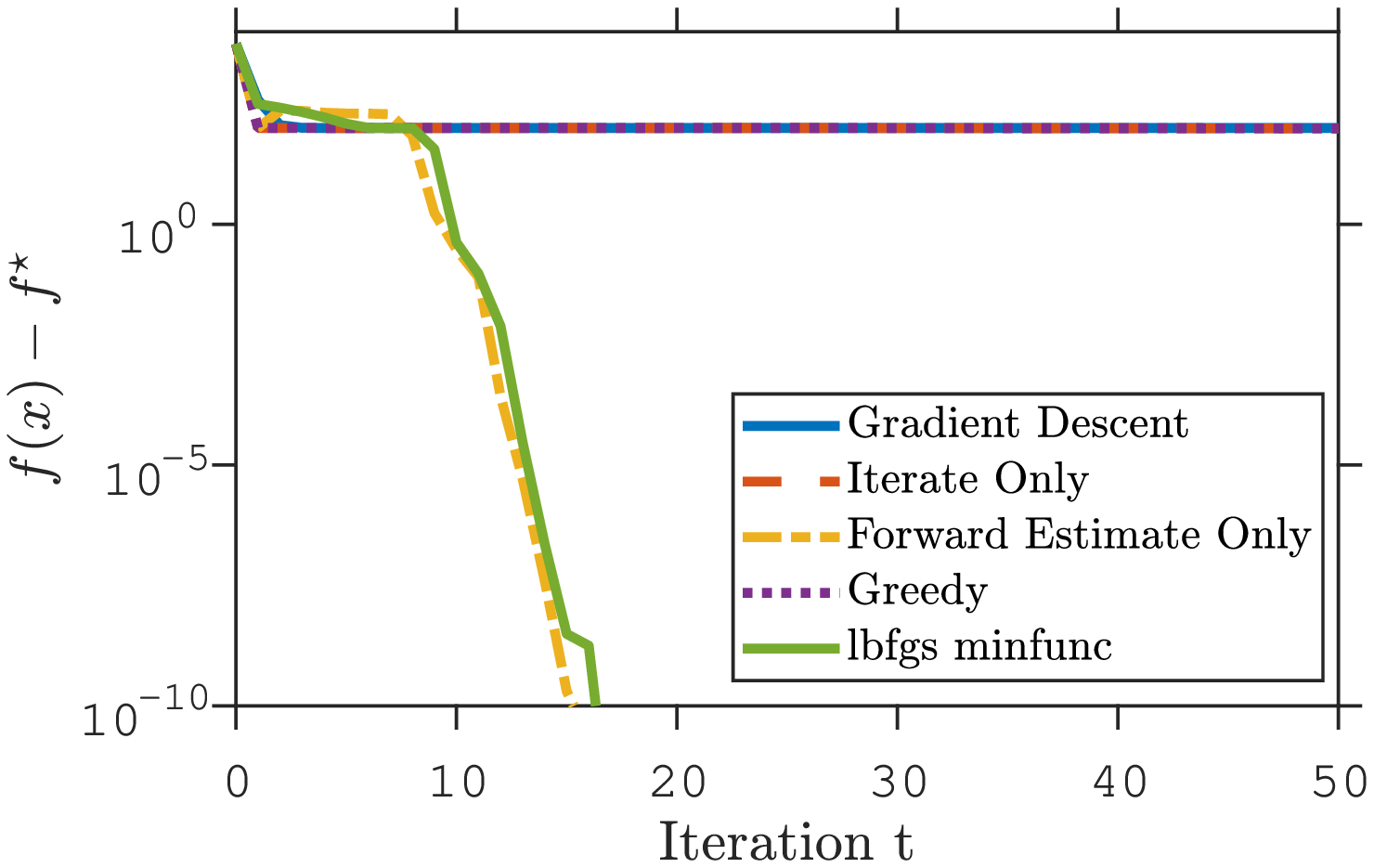}
\includegraphics[width=0.49\textwidth]{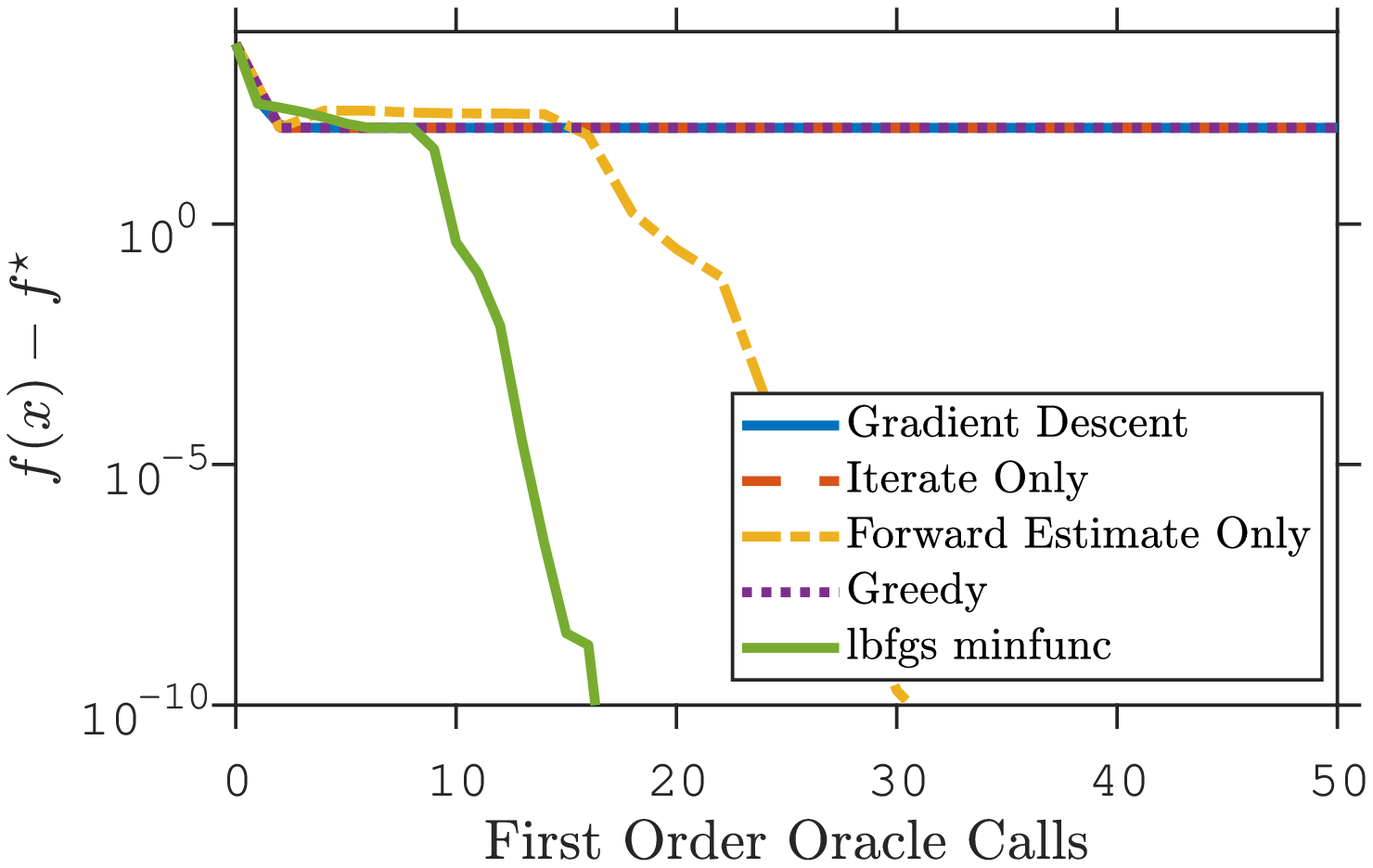}
    \caption{Comparison of type 2 methods: Square loss and cubic regularization on sido0 dataset}
    \label{fig:sido0_quad_type2}
\end{figure}

\begin{figure}[h!t]
    \centering
    \includegraphics[width=0.49\textwidth]{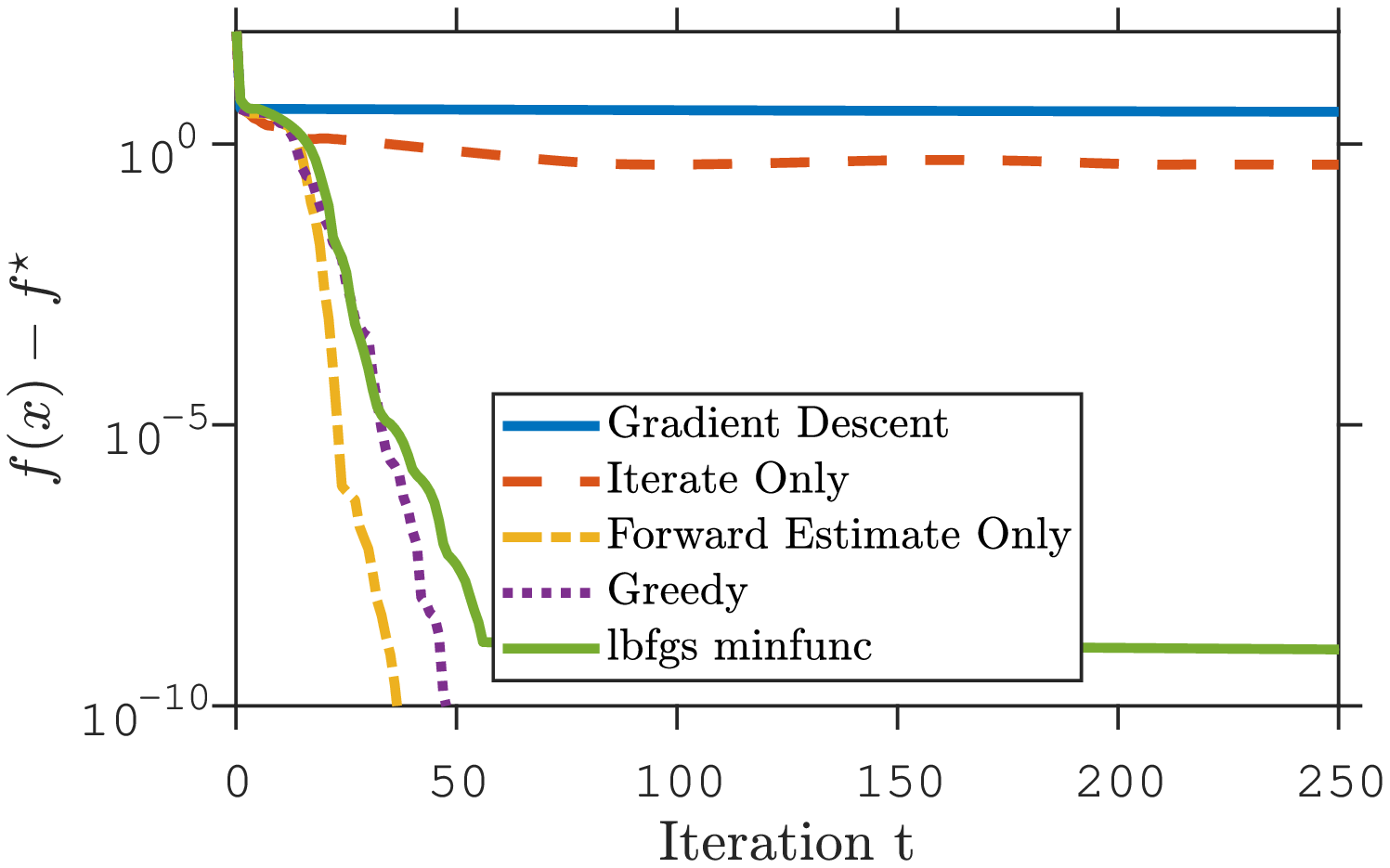}
\includegraphics[width=0.49\textwidth]{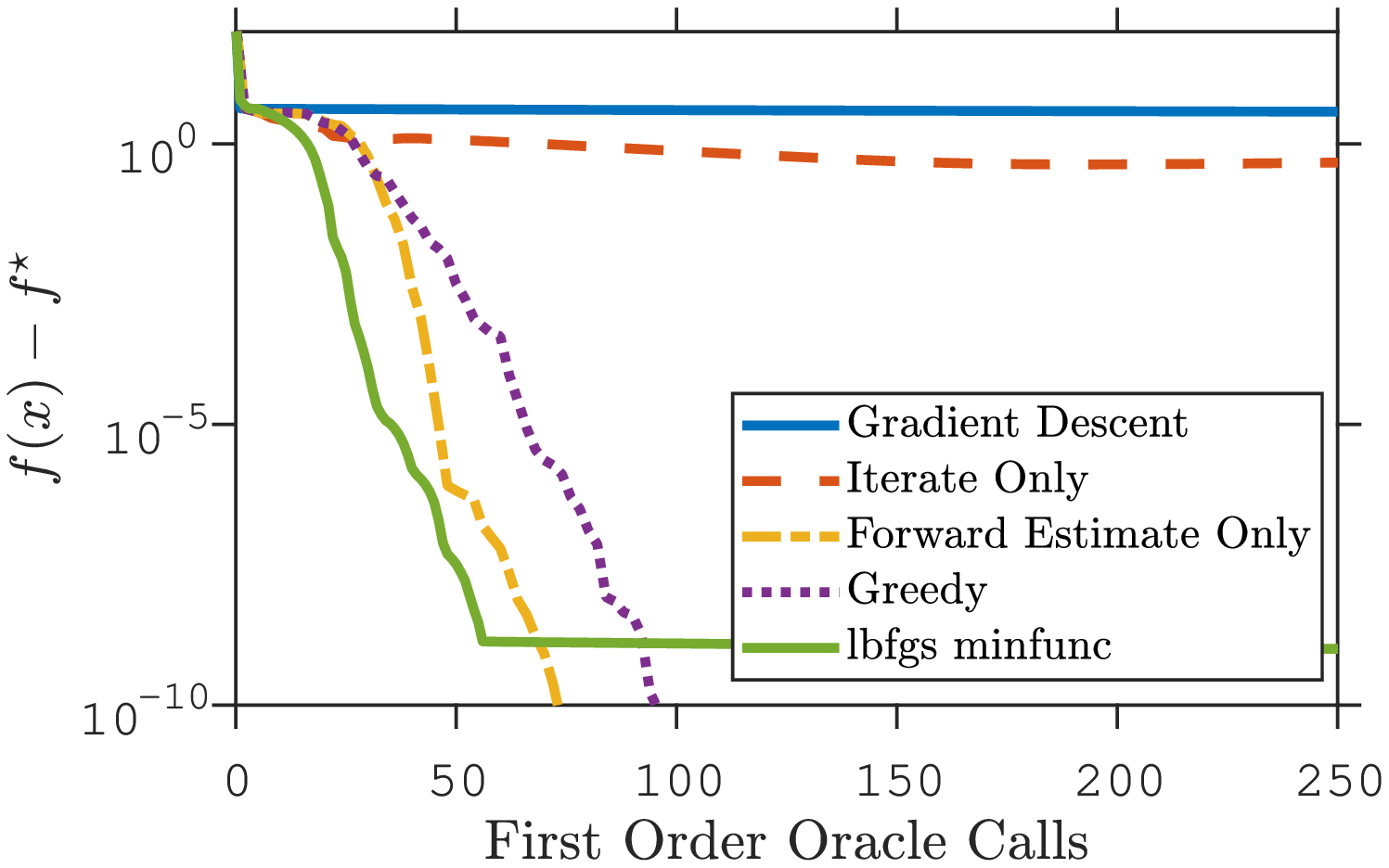}
    \caption{Comparison of type 2 methods: Square loss and cubic regularization on marti2 dataset}
    \label{fig:marti2_quad_type2}
\end{figure}

\clearpage
\subsubsection{Logistic regression}

\begin{figure}[h!t]
    \centering
    \includegraphics[width=0.49\textwidth]{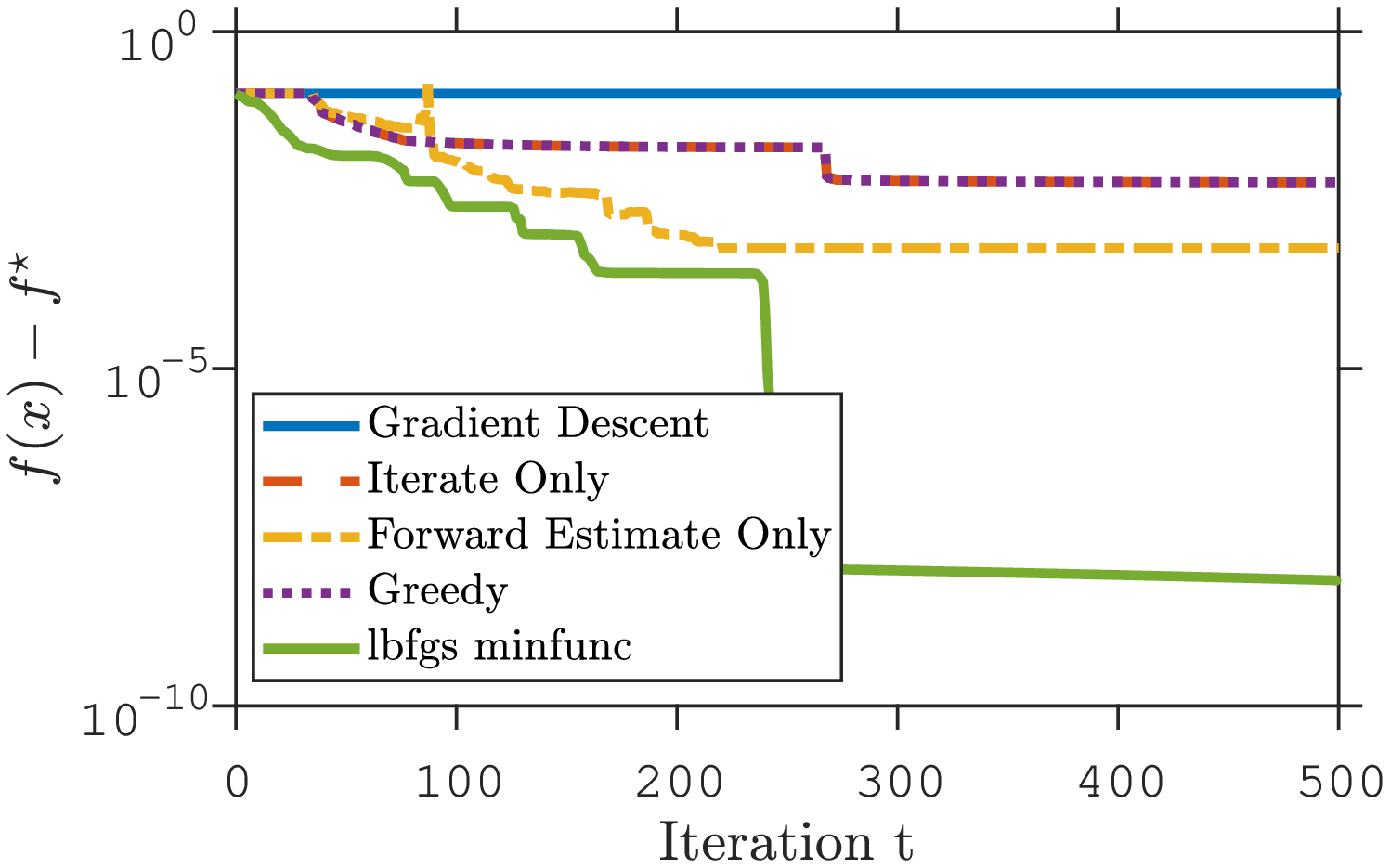}
\includegraphics[width=0.49\textwidth]{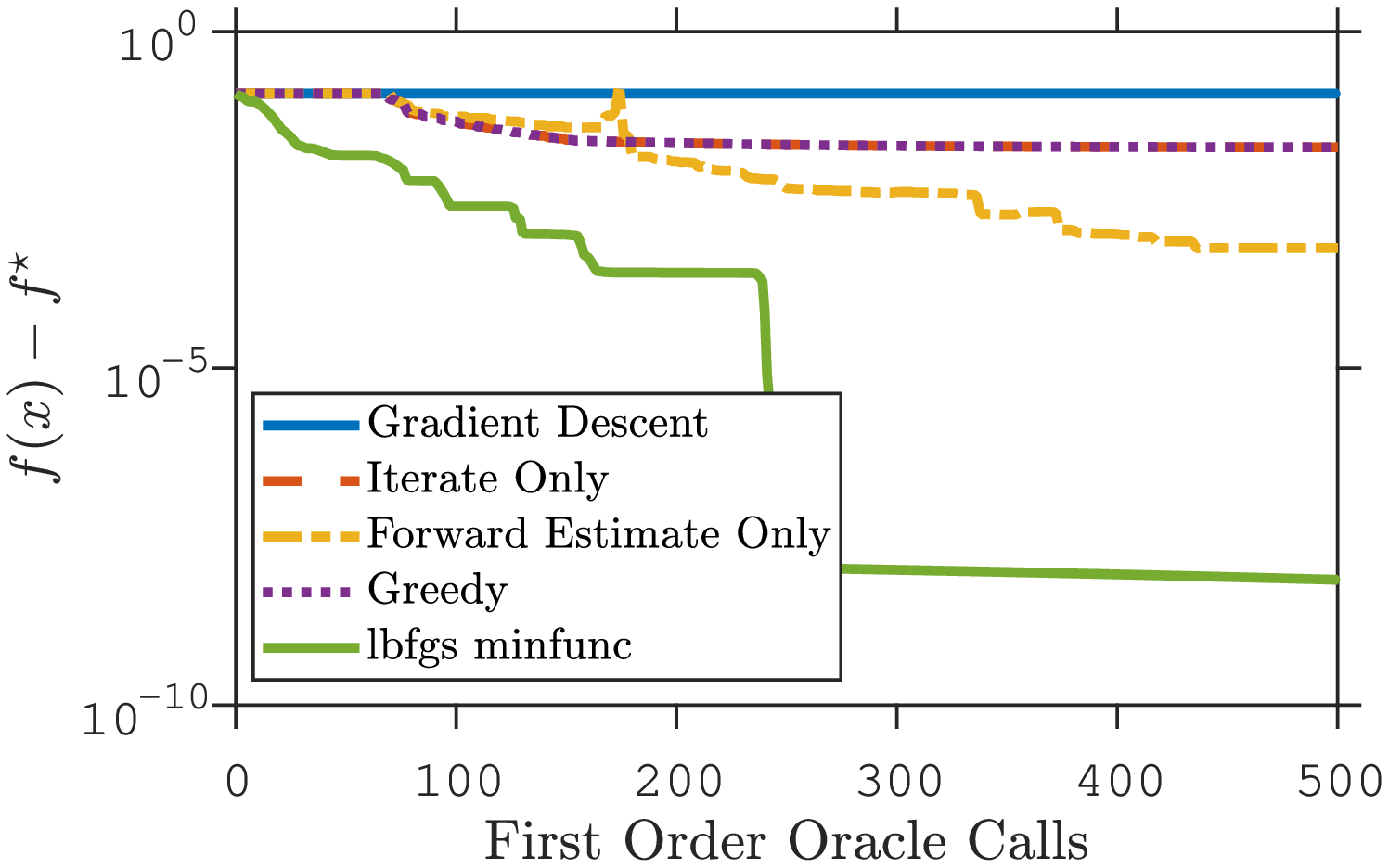}
    \caption{Comparison of type 2 methods: Logistic loss and cubic regularization on Madelon dataset}
    \label{fig:madelon_logistic_type2}
\end{figure}

\begin{figure}[h!t]
    \centering
    \includegraphics[width=0.49\textwidth]{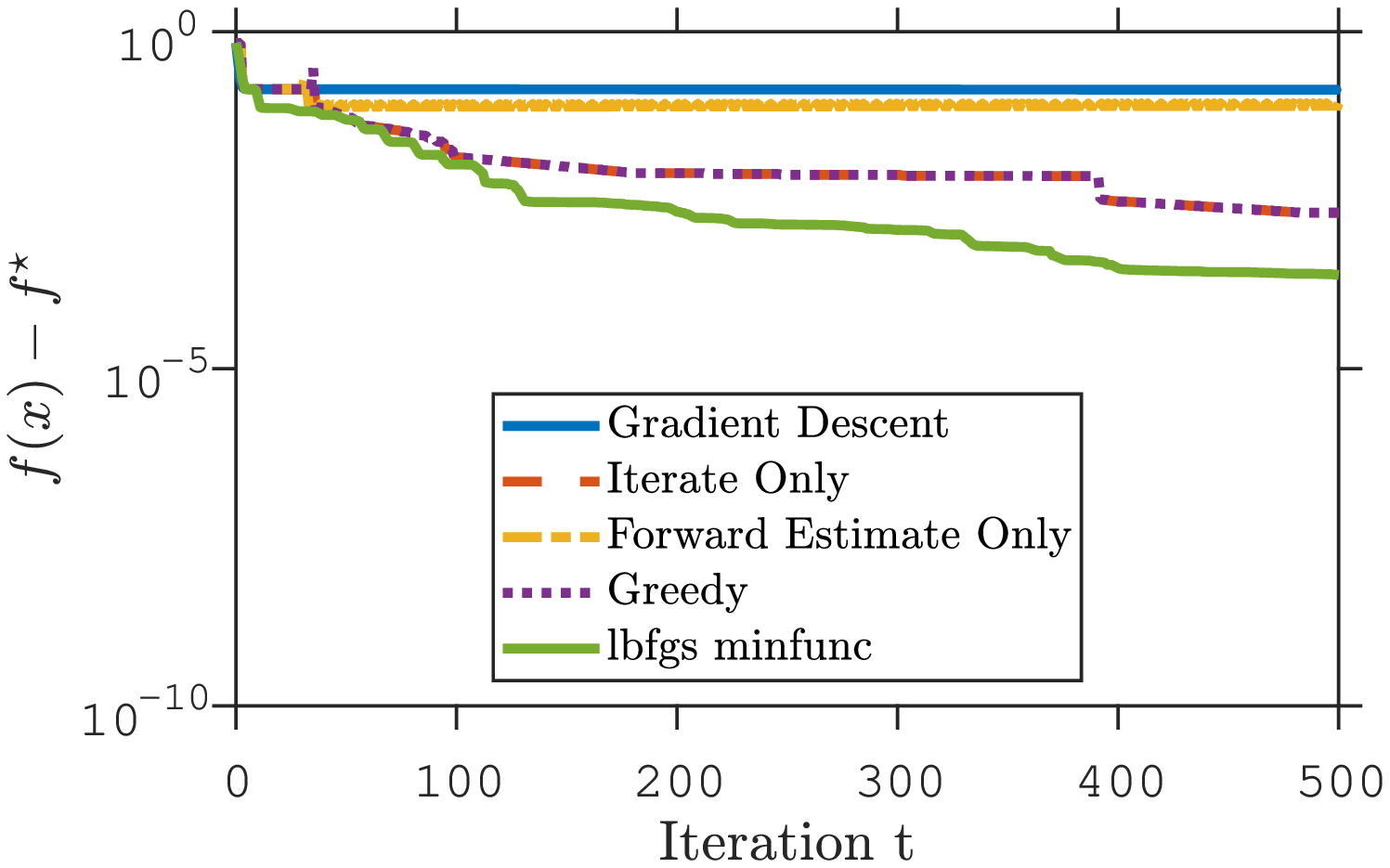}
\includegraphics[width=0.49\textwidth]{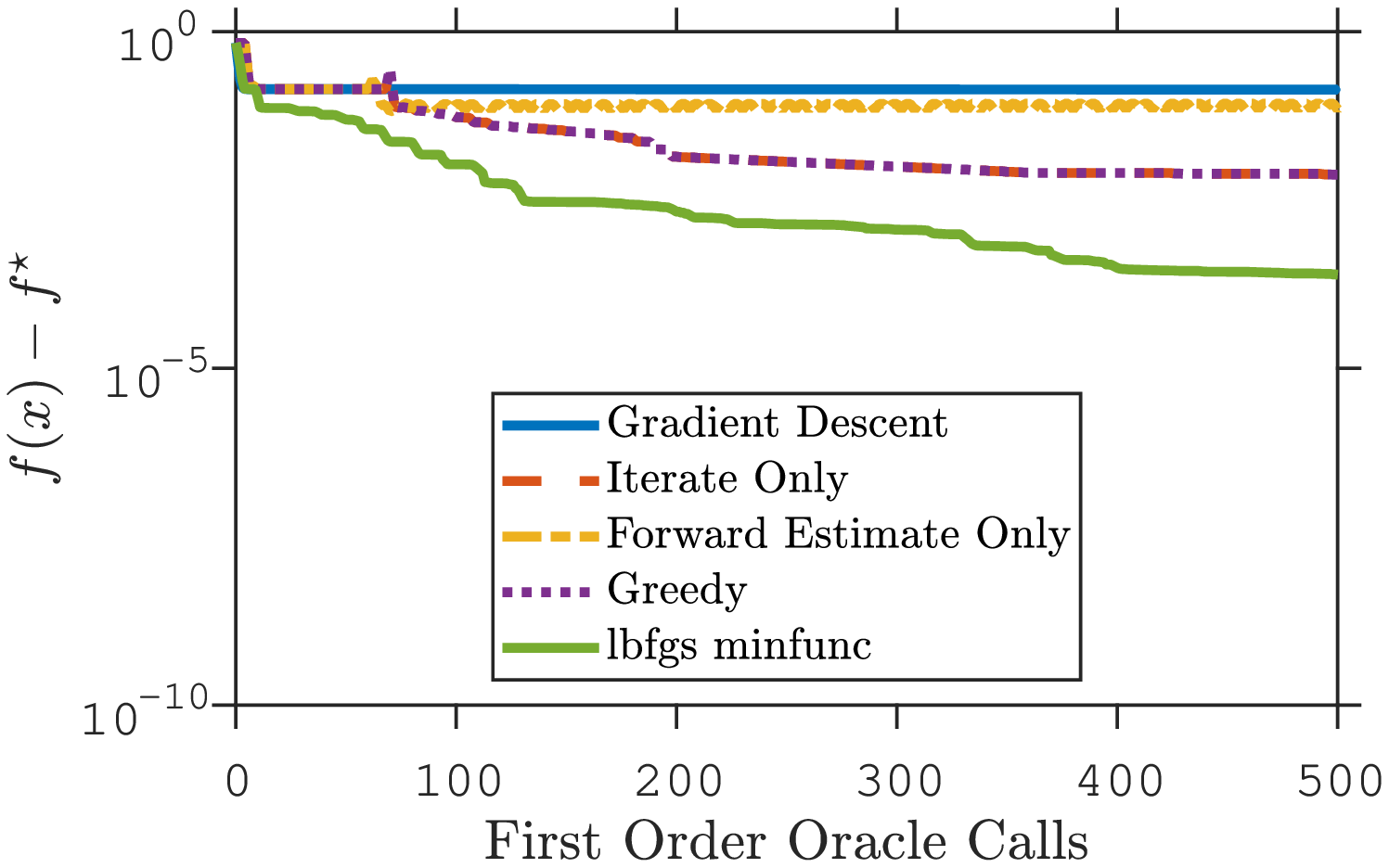}
    \caption{Comparison of type 2 methods: Logistic loss and cubic regularization on sido0 dataset}
    \label{fig:sido0_logistic_type2}
\end{figure}

\begin{figure}[h!t]
    \centering
    \includegraphics[width=0.49\textwidth]{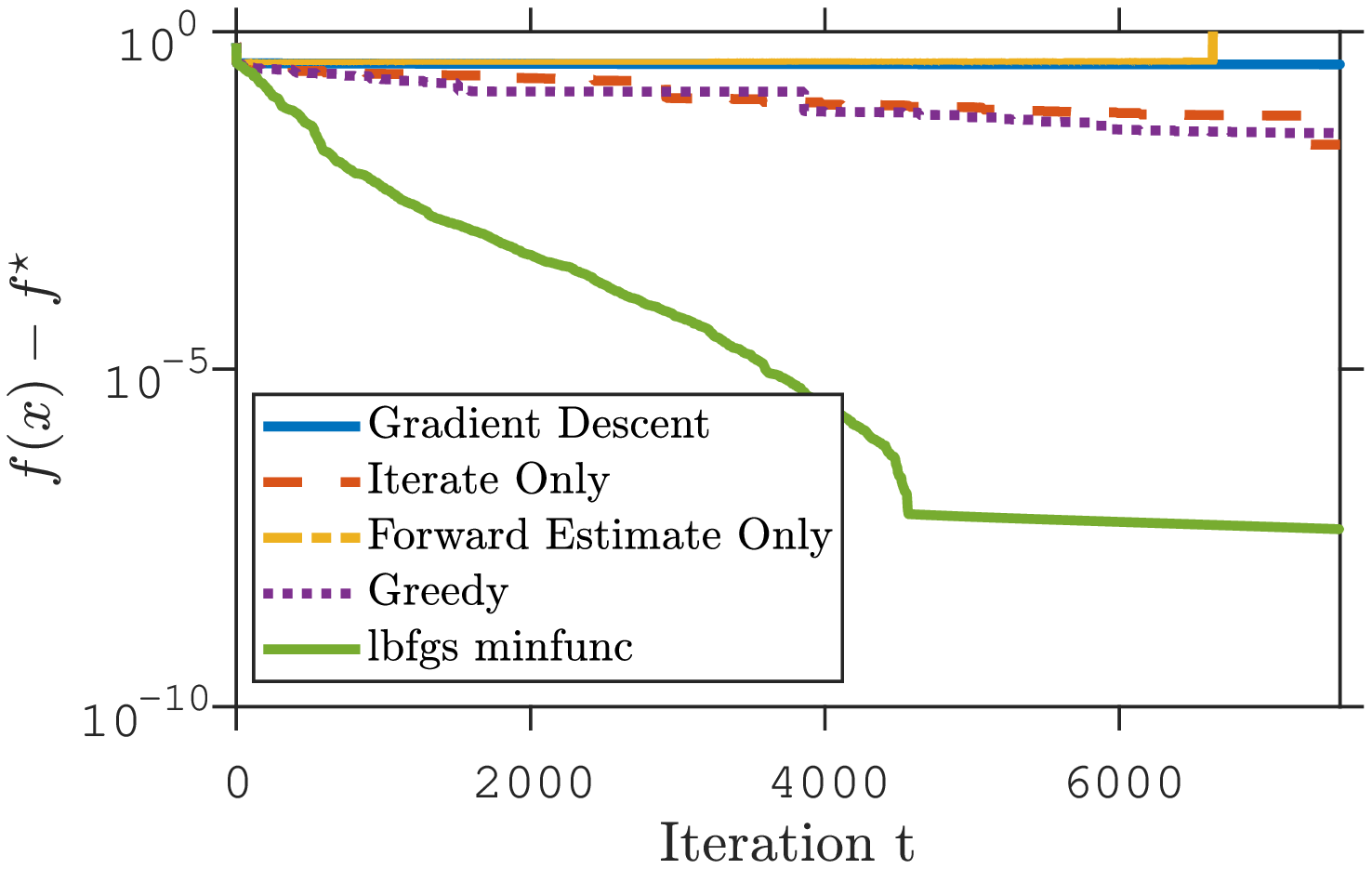}
\includegraphics[width=0.49\textwidth]{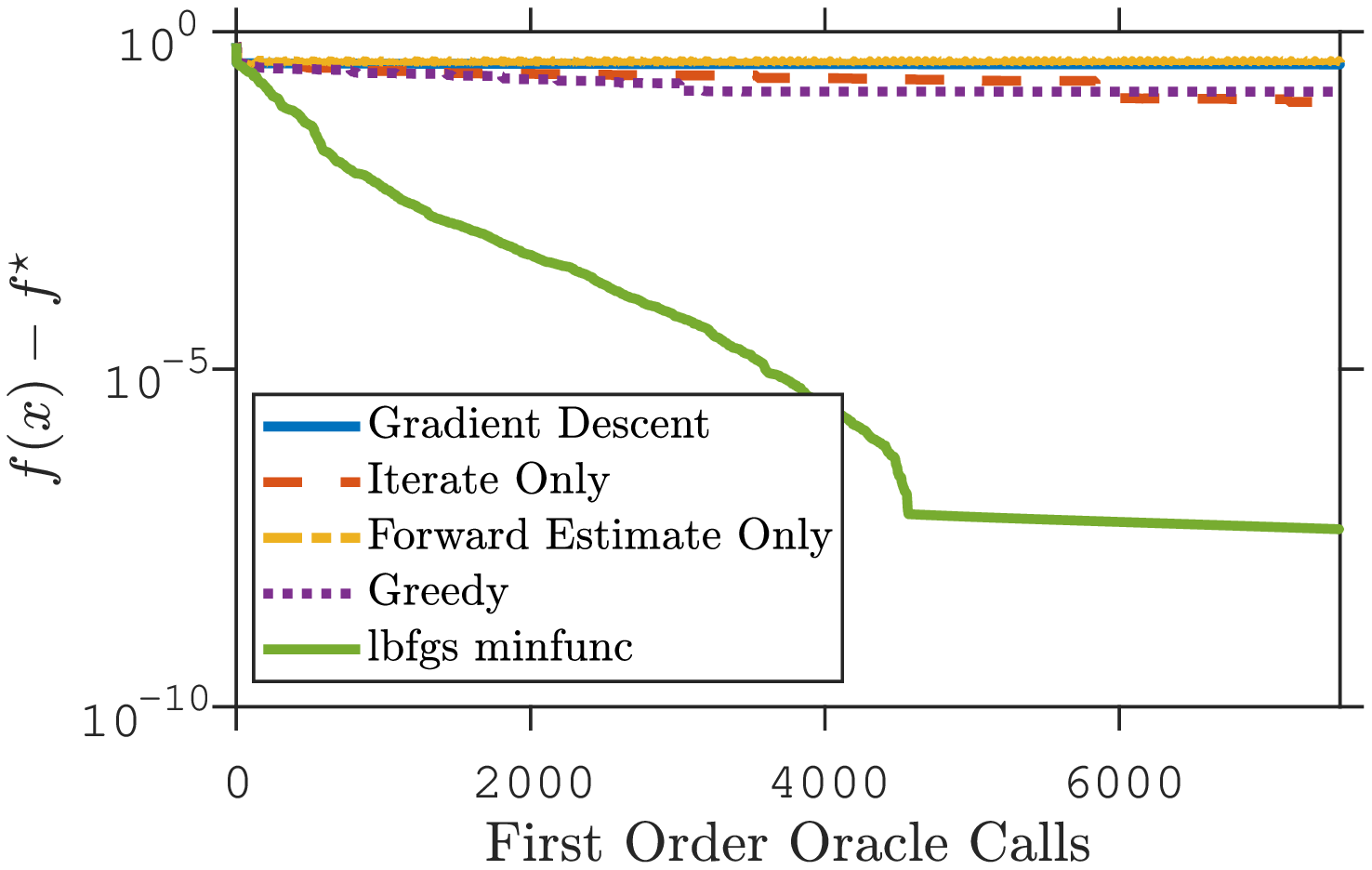}
    \caption{Comparison of type 2 methods: Logistic loss and cubic regularization on marti2 dataset}
    \label{fig:marti2_logistic_type2}
\end{figure}

\clearpage

\section{Missing proofs} \label{sec:missing_proofs}

In this section, when not needed, the subscript $t$ has been removed for clarity. The following definitions simplify the notations:
\begin{align}
    D_\dagger & = (D^TD)^{-1}D^T,& \label{eq:inv_D} \\
    D^T_\dagger & = D(D^TD)^{-1},& \label{eq:inv_D_traspose} \\
    \kappa_D &= \|D_\dagger\|\|D\| \label{eq:kappaD},
\end{align}
Note that the pseudo inverse $D_\dagger$ exists under \cref{assump:bounded_conditionning}. Note that
\[
    D_\dagger D = I, \qquad DD_\dagger = P_D = P.
\]

\subsection{Technical Result: Hessian Approximation}

This section presents technical results related to the approximation of the Hessian $\nabla^2 f(x)$. To simplify notations, let the matrices $H_0$ and $\tilde H_0$ be
\begin{align}\label{eq:def_h0}
    H_0 = \frac{D^TR+R^TD}{2},\qquad \tilde H_0 = D_\dagger^T H_0 D_\dagger = \frac{PGD_{\dagger}+D_{\dagger}^TG^TP}{2}.
\end{align}
Intuitively, $\tilde H_0$ is the Hessian approximation, while $H_0$ is the approximation of the quadratic form $D^T\nabla^2 f(x)D$.

\begin{proposition}[Subspace Hessian Approximation Error]\label{prop:accuracy_estimation}
    Assume $D$ satisfies \cref{assump:projector_grad}. Then, the following holds:
    \[
        \left\| \left(\tilde H_0 - P\nabla^2 f(x)P\right)D\alpha \right\| \leq \frac{L}{2}\|D_{\dagger}\|\|\varepsilon\|\|D\alpha\|
    \]
\end{proposition}
\begin{proof}
    Since $D^T D_\dagger=D_\dagger^TD^T=P$, $D_\dagger D = \mathrm{I}$, $PD=D$, $\|P\| = 1$, and using \cref{eq:def_h0},
    \begin{align*}
        & \left\|\left[\frac{PGD_{\dagger}+D^T_{\dagger}G^TP}{2}-P\nabla^2 f(x)P\right]D\alpha\right\| \\
        \leq \quad &  \frac{1}{2}\big(\left\|(PGD_{\dagger} -P\nabla^2 f(x)P)D\alpha\right\| + \left\|(D^T_{\dagger}G^TP-P\nabla^2 f(x)P)D\alpha\right\|\big)\\
        \leq \quad &  \frac{1}{2}\Big(\left\|G\alpha -\nabla^2 f(x)D\alpha\right\| + \|D_{\dagger}\|\left\|(G^T-D^T\nabla^2 f(x))D\alpha\right\|\Big)\\
    \end{align*}
    Using inequality \cref{eq:bound_hessian_scalar} for the first term and \cref{eq:bound_hessian_dalpha} for second gives
    \begin{align*}
        \left\|\left[\frac{PGD_{\dagger}+D^T_{\dagger}G^TP}{2}-P\nabla^2 f(x)P\right]D\alpha\right\| \leq \frac{1}{2}\left(\frac{L}{2} |\alpha|^T \varepsilon + \|D_{\dagger}\| \frac{L\|D\alpha\|}{2} \|\varepsilon\|\right)
    \end{align*}
    Because $|\alpha|^T \varepsilon \leq \|\alpha\|\|\varepsilon\|\leq \|D_{\dagger}\| \|D\alpha\| \|\varepsilon\|$,
    \begin{align*}
        \left\|\left[\frac{PGD_{\dagger}+D^T_{\dagger}G^TP}{2}-P\nabla^2 f(x)P\right]D\alpha\right\| \leq \frac{L}{2}\|D_{\dagger}\|\|\varepsilon\|\|D\alpha\|.
    \end{align*}
\end{proof}

\begin{proposition} \label{prop:bound_projector_hessian}[Out-of-subspace Error Estimation]
    Let the function $f$ satisfy \cref{assump:lipchitiz_hessian}. Let the matrices $D,\,G$ be defined as in \cref{def:matrices} and vector $\varepsilon$ as in \cref{eq:error_vector}. 
    Then, for all $\alpha\in\mathbb{R}^N$,
    \[
        \frac{\|(I-P)\nabla^2 f(x)D\alpha\|}{\|D\alpha\|} \leq  \left(\|(I-P)G\| + L\|\varepsilon\|\right)\frac{\kappa_D}{\|D\|}.
    \]
\end{proposition}
\begin{proof}
    Indeed, using \cref{eq:bound_hessian_dalpha},
    \begin{align*}
        \|(I-P)\nabla^2 f(x)D\alpha\| & = \|(I-P)(G-G+\nabla^2 f(x)D)\alpha\|\\
        & \leq \|(I-P)(\nabla^2 f(x)D-G)\alpha\| + \|(I-P)G\alpha\|\\
        & \leq \|(\nabla^2 f(x)D-G)\alpha\| + \|(I-P)G\alpha\|\\
        & \leq \|\nabla^2 f(x)D-G\|\|\alpha\| + \|(I-P)G\alpha\|\\
        & \leq \left(\frac{L\|\varepsilon\|}{2}\|\alpha\| + \|(I-P)G\alpha\|\right)\\
    \end{align*}
    Hence,
    \[
        \frac{\|(I-P)\nabla^2 f(x)D\alpha\|}{\|D\alpha\|}  \leq \frac{ \left(\frac{L\|\varepsilon\|}{2}\|\alpha\| + \|(I-P)G\alpha\|\right)}{\|D\alpha\|}.
    \]
    Moreover,
    \begin{align*}
        \frac{\|(I-P)\nabla^2 f(x)D\alpha\|}{\|D\alpha\|} 
        & \leq \left(\frac{L\|\varepsilon\|}{2} + \|(I-P)G\|\right) \|\alpha\|.\\
        & \leq \left(\frac{L\|\varepsilon\|}{2} + \|(I-P)G\|\right) \frac{\|\alpha\|}{\|D\alpha\|}\\
        & \leq \max_\alpha \left(\frac{L\|\varepsilon\|}{2} + \|(I-P)G\|\right) \frac{\|\alpha\|}{\|D\alpha\|}\\
        & = \left(\frac{L\|\varepsilon\|}{2} + \|(I-P)G\|\right) \sigma^{-1}_{\min}(D).
    \end{align*}
    The desired result follows from the fact that $\kappa_D = \frac{\|D\|}{\sigma_{\min}(D)}$.
\end{proof}

\subsection{Technical Results: Cubic Subproblem}

This section presents results on the properties of the solution of the cubic subproblem
\begin{equation}\label{eq:app_cubic_subproblem}
    \alpha^\star \defas \argmin_\alpha \nabla f(x)^T(D\alpha) + \frac{1}{2}(D\alpha)^T\tilde H_\Gamma (D\alpha) + \frac{M}{6}\|D\alpha\|^3,\qquad  x_+ = x+D\alpha^\star
\end{equation}
where $\tilde H_\Gamma \in\mathbb{R}^{d\times d}$ is a rank $N$ matrix such that
\begin{align}\label{eq:def_tildeH}
    \tilde H = D_\dagger^T H_\gamma D_\dagger, \qquad \Leftrightarrow\;\; H = D^T\tilde H_\Gamma D, \qquad H_\gamma = \frac{R^TD+D^TR+\Gamma}{2},
\end{align}
and $\Gamma$ is a $N\times N$ matrix. For instance, setting $\Gamma = M\|\varepsilon\|\|D\|I$ gives the $H$ in \cref{alg:type1}.

\begin{proposition} The first-order and second-order conditions of the subproblem  \cref{eq:app_cubic_subproblem} read
\begin{align} 
    D^T\nabla f(x) + H_{\Gamma}\alpha + \frac{M}{2}D^TD\alpha \|D\alpha\| & = 0,\label{eq:first_order_condition}\\
    H_{\Gamma} + \frac{M}{2}D^TD\|D\alpha\| & \succeq 0.\label{eq:second_order_condition}
\end{align}
\end{proposition}
\begin{proof}
    See \citep{nesterov2008accelerating}, equation (3.3), and \citep{nesterov2006cubic}, equation (2.7).
\end{proof}

\begin{proposition} \label{prop:ineq_grad_qn_optimal}
     Let $f$ satisfies \cref{assump:lipchitiz_hessian} and $B\in\mathbb{R}^{d\times d}$ be any matrix. Assume the matrix $D$ satisfies \cref{assump:projector_grad}, and $\alpha$ satisfies the first-order condition~\cref{eq:first_order_condition}. Let $\tilde H_\Gamma$ be defined in \cref{eq:def_tildeH}. Then,
     \begin{align}
         \|\nabla f(x) + BD\alpha - \nabla f(x_+)\| & = \|(\tilde H_{\Gamma}-B + \frac{M\|D\alpha\|}{2})D\alpha + \nabla f(x_+)\| \\
         & \leq \frac{L}{2}\|D\alpha\|^2 + \|[B-\nabla^2 f(x)]D\alpha\|.
     \end{align}
     Then, the following equation follows from the optimality condition multiplied by $D(D^TD)^{-1}$, writing $P = DD_\dagger = D_\dagger^TD^T$, assuming $P\nabla f(x) = \nabla f(x)$,
    \[
        \nabla f(x) + (\tilde H_{\Gamma} + \frac{M\|D\alpha\|}{2})D\alpha = 0.
    \]
    Replacing $\nabla f(x)$ gives
    \[
    \|\nabla f(x) + BD\alpha - \nabla f(x_+)\| = \|-(\tilde H_{\Gamma} + \frac{M\|D\alpha\|}{2})D\alpha + BD\alpha - \nabla f(x_+)\|,
    \]
    which is the desired result.
\end{proposition}
\begin{proof}
    The inequality follows directly from \cref{eq:ineq_secant},
    \begin{align*}
        \|\nabla f(x) + BD\alpha - \nabla f(x_+)\| & \leq \|\nabla f(x) + \nabla^2f(x)D\alpha - \nabla f(x_+)\| + \|BD\alpha -\nabla^2f(x)D\alpha\| \\
        &\leq \frac{L}{2} \|D\alpha\|^2 + \|[B -\nabla^2f(x)]D\alpha\| .
    \end{align*}
        
\end{proof}

\begin{proposition}
\label{prop:ineq_grad_qn_optimal_no_b}
    Assume $D$ satisfies \cref{assump:projector_grad}. Let $\tilde H$ be defined in \cref{eq:def_tildeH}. Then, for all $\tilde \Gamma$, if 
    \[
        B = \tilde H_\Gamma - \frac{1}{2} D_\dagger\tilde \Gamma D_\dagger^T
    \]
        in \cref{prop:ineq_grad_qn_optimal}, the following holds:
    \begin{align}
        \left\|\left(\frac{1}{2}D_\dagger\tilde \Gamma D_\dagger^T + \frac{M\|D\alpha\|}{2}\right)D\alpha + \nabla f(x_+)\right\| \leq \frac{L}{2}\|D\alpha\|^2 + \|[B-\nabla^2 f(x)]D\alpha\|,
    \end{align}
    where
    \begin{align*}
        \|[B-\nabla^2 f(x)]D\alpha\| \leq \|D\alpha\|\left( \frac{L}{2}\|D_{\dagger}\|\|\varepsilon\|+ \frac{\|(I-P)\nabla^2 f(x)D\alpha\|}{\|D\alpha\|}  + \frac{1}{2}\left\|D_\dagger (\Gamma-\tilde \Gamma) D_\dagger\right\|\right)
    \end{align*}
\end{proposition}
\begin{proof}
    From \cref{prop:ineq_grad_qn_optimal},
    \[
        \|(\tilde H_{\Gamma}-B + \frac{M\|D\alpha\|}{2})D\alpha + \nabla f(x_+)\| \leq \frac{L}{2}\|D\alpha\|^2 + \|[B-\nabla^2 f(x)]D\alpha\|.
    \]
    Replacing $B$ in the left-hand-side gives
    \[
         \|(\tilde H_{\Gamma}-B + \frac{M\|D\alpha\|}{2})D\alpha + \nabla f(x_+)\| =  \|( \frac{D_\dagger\Gamma D_\dagger^T}{2}+ \frac{M\|D\alpha\|}{2})D\alpha + \nabla f(x_+)\|
    \]
    Since
    \[
        \nabla^2 f(x)D\alpha = P\nabla^2 f(x)PD\alpha + (I-P)\nabla^2 f(x)PD\alpha,
    \]
    where $P = D(D^TD)^{-1}D^T$, and because $PD = D$,
    the inequality becomes
    \begin{align}
        \|[B-\nabla^2 f(x)]D\alpha\| = &\|\left[\tilde H_\Gamma - \frac{1}{2} D_\dagger \tilde \Gamma D_\dagger^T -\nabla^2 f(x)\right]D\alpha\|
        \\
        = &\|[P + (I-P)]\left[\tilde H_\Gamma - \frac{1}{2} D_\dagger \tilde \Gamma D_\dagger^T -\nabla^2 f(x)\right]P D\alpha\|\\
        \leq &  \left\| \left(\tilde H_0 - P\nabla^2 f(x)P\right)D\alpha \right\| 
        \label{eq:temp_term1} \\
        & + \left( \frac{1}{2}\left\| D_{\dagger}^T (\Gamma-\tilde \Gamma) D_{\dagger} \right\| + \frac{\left\| (I-P)\nabla^2 f(x)D\alpha) \right\|}{\|D\alpha\|} \right)\|D\alpha\|\label{eq:temp_term2}
    \end{align}
\end{proof}

\begin{corollary}[Bound depending on $\tilde \Gamma$] 
\label{cor:bound_depend_gamma}
In \cref{prop:ineq_grad_qn_optimal_no_b}, 
\begin{itemize}
    \item  if $\tilde \Gamma = 0$ and $\Gamma = M\|D\|\|\varepsilon\|I$,
    \begin{align} \label{eq:bound_gamma_0}
    \left\|\frac{M\|D\alpha\|}{2}D\alpha + \nabla f(x_+)\right\| \leq \frac{L}{2}\|D\alpha\|^2+ \|D\alpha\|\left(  \frac{\|\varepsilon\|}{\|D\|}\left(\frac{L+M\kappa_D}{2}\right)\kappa_D
     + \|(I-P)\nabla^2 f(x)P\| \right)
\end{align}
    \item if $\tilde \Gamma = \Gamma$,
\begin{align}\label{eq:bound_cancel_gamma}
    \left\|\left(\frac{1}{2}D_\dagger \Gamma D_\dagger^T + \frac{M\|D\alpha\|}{2}\right)D\alpha + \nabla f(x_+)\right\| \leq \frac{L}{2}\|D\alpha\|^2 + \|D\alpha\|\left( \frac{L}{2} \frac{\|\varepsilon\|}{\|D\|} \kappa_D + \frac{\|(I-P)\nabla^2 f(x)D\alpha\|}{\|D\alpha\|}\right)
\end{align}
\item If $\tilde \Gamma =  D( M \|D\alpha\|)D^T$ and $\Gamma = M\|D\|\|\varepsilon\|I$,
\begin{align}\label{eq:lower_bound_gradient_plus}
        \|\nabla f(x_+)\|
         & \leq \frac{L+M}{2}\|D\alpha\|^2 + \|D\alpha\|\left( \frac{\|\varepsilon\|}{\|D\|}\left(\frac{L+M\kappa_D}{2}\right)\kappa_D+ \|(I-P)\nabla^2 f(x)P\| \right)
    \end{align}
\end{itemize}
\end{corollary}

\subsection{Technical Results: Decrease Guarantees}

This section presents two technical results on the minimal decrease of the function $f$.

\begin{proposition} \label{prop:global_lower_bound_xplus}
    Let \cref{assump:lipchitiz_hessian,assump:projector_grad,assump:bounded_epsilon,assump:bounded_conditionning} hold. Then, $\forall y\in\mathbb{R}^d$, \cref{alg:type1} ensures
    \[
        f(x_+) \leq f(y) + \frac{M+L}{6}\|y-x\|^3  + \frac{\|y-x\|^2}{2} \left( \|\nabla^2 f(x)- P\nabla^2 f(x)P\| + \delta\frac{L \kappa +  M\kappa^2}{2}  \right)
    \]
\end{proposition}
\begin{proof}
    The output of \cref{alg:type1} ensures that 
    \begin{align*}
        f(x_+) \leq \min_\alpha f(x) + \nabla f(x)^TD\alpha + \frac{1}{2} (D\alpha)^T\nabla^2 f(x)D\alpha  + \frac{1}{2} \alpha^T\left(H-D^T\nabla^2 f(x)D\right)\alpha + \frac{M}{6}\|D\alpha\|^3
    \end{align*}
    However, by the definition of $H$ \cref{eq:type1_bound},
    \begin{align*}
        & \frac{1}{2} \alpha^T\left(H-D^T\nabla^2 f(x)D\right)\alpha\\
        \leq & \frac{1}{2}\left(  \alpha^T\left(\frac{G^TD+D^TG}{2}-D^T\nabla^2 f(x)D\right)\alpha +  \|\alpha\|^2\frac{M\|D\|\|\varepsilon\|}{2} \right)\\
        \leq & \frac{1}{2}\left( \alpha^T\left(\frac{G^TD+D^TG}{2}-D^T\nabla^2 f(x)D\right)\alpha +  \|D^\dagger\|^2\|D\alpha\|\frac{M\|D\|\|\varepsilon\|}{2} \right)\\
        = & \frac{1}{2}\left(  (D\alpha)^T\left(G-\nabla^2 f(x)D\right)\alpha +  \|D^\dagger\|^2\|D\alpha\|\frac{M\|D\|\|\varepsilon\|}{2}\right).
    \end{align*}
    The last equality comes from the fact that
    \[
        \alpha^T\left( D^TG \right)\alpha = \alpha^T\left( \frac{D^TG + G^TD}{2} + \frac{D^TG - G^TD}{2} \right)\alpha = \alpha^T\left( \frac{D^TG + G^TD}{2}\right)\alpha.
    \]
    Now, using \cref{eq:bound_hessian_scalar} with $w = D\alpha$ gives
    \begin{align*}
        \frac{1}{2} \alpha^T\left(H-D^T\nabla^2 f(x)D\right)\alpha \leq  \frac{L\|D\alpha\|}{4} \sum_{i=1}^N |\alpha_i| \varepsilon_i +  \|D^\dagger\|^2\|D\alpha\|\frac{M\|D\|\|\varepsilon\|}{4}.
    \end{align*}
    Finally, since 
    \[
        \sum_{i=1}^N |\alpha_i| \varepsilon_i \leq \|\alpha\|\|\varepsilon\| \leq \|D^{\dagger}\|\|D\alpha\|\|\varepsilon\|,
    \]
    the inequality becomes
    \begin{align*}
        \frac{1}{2} \alpha^T\left(H-D^T\nabla^2 f(x)D\right)\alpha & \leq  \frac{\|D\alpha\|^2}{4}\left(L \|D^{\dagger}\|\|\varepsilon\| +  M\|D^\dagger\|^2\|D\|\|\varepsilon\|\right)\\
        & = \frac{\|D\alpha\|^2}{4}\frac{\|\varepsilon\|}{\|D\|}\left(L \kappa_D +  M\kappa_D^2\right).
    \end{align*}
    All together,
    \begin{align*}
        & f(x_+) \\
        \leq & \min_\alpha f(x) + \nabla f(x)^TD\alpha + \frac{1}{2} (D\alpha)^T\nabla^2 f(x)D\alpha + \frac{1}{2} \alpha^T\left(H-D^T\nabla^2 f(x)D\right)\alpha + \frac{M}{6}\|D\alpha\|^3\\
        \leq & \min_\alpha f(x) + \nabla f(x)^TD\alpha + \frac{1}{2} (D\alpha)^T\nabla^2 f(x)D\alpha + \frac{\|D\alpha\|^2}{4}\frac{\|\varepsilon\|}{\|D\|}\left(L \kappa_D +  M\kappa_D^2\right) + \frac{M}{6}\|D\alpha\|^3
    \end{align*}
    Now, by \cref{assump:bounded_conditionning}, for all $y$, one can find $\alpha$ such that
    \[
        D\alpha = P(y-x) = D D^\dagger(y-x).
    \]
    Indeed, multiplying both sides by $D^{\dagger}$ gives
    \[
        \alpha = D^{\dagger}(y-x).
    \]
    Therefore, the minimum can be written as a function of $y$ instead of $\alpha$,
    \begin{align}
        f(x_+) \leq \min_{y\in\mathbb{R}^d} \;\; & f(x) + \nabla f(x)^TP(y-x) + \frac{1}{2} (P(y-x))^T\nabla^2 f(x)P(y-x) \nonumber\\
        & + \frac{\|P(y-x)\|^2}{4}\frac{\|\varepsilon\|}{\|D\|}\left(L \kappa_D +  M\kappa_D^2\right) + \frac{M}{6}\|P(y-x)\|^3.\label{eq:temp_global_lower_bound}
    \end{align}
    Since $P\nabla f(x) = \nabla f(x)$ by \cref{assump:projector_grad}, and using the crude bound $\|P(y-x)\|\leq \|y-x\|$,
    \begin{align*}
        f(x_+) \leq \min_{y\in\mathbb{R}^d} \;\; & f(x) + \nabla f(x)^T(y-x) + \frac{1}{2} (y-x)^T\nabla^2 f(x)(y-x)  \\
        & + \frac{1}{2} (y-x)\left[\nabla^2 f(x)- P\nabla^2 f(x)P\right](y-x) \\
        & + \frac{\|y-x\|^2}{4}\frac{\|\varepsilon\|}{\|D\|}\left(L \kappa_D +  M\kappa_D^2\right) + \frac{M}{6}\|y-x\|^3.
    \end{align*}
    Using the lower bound \cref{eq:ineq_function},
    \[
        f(x) + \nabla f(x)^T(y-x) + \frac{1}{2} (y-x)^T\nabla^2 f(x)(y-x) - \frac{L}{6}\|y-x\|^3 \leq f(y),
    \]
    the crude bound $(y-x)\left[\nabla^2 f(x)- P\nabla^2 f(x)P\right](y-x) \leq \|\nabla^2 f(x)- P\nabla^2 f(x)P\|\|y-x\|^2$, and \cref{assump:bounded_conditionning,assump:bounded_epsilon} lead to the desired result,
    \begin{align*}
        f(x_+) \leq f(y) + \frac{M+L}{6}\|y-x\|^3  + \frac{\|y-x\|^2}{2} \left( \|\nabla^2 f(x)- P\nabla^2 f(x)P\| + \delta\frac{L \kappa +  M\kappa^2}{2}  \right)
    \end{align*}
\end{proof}

\begin{proposition} \label{prop:global_lower_bound_xplus_stoch}
    Let \cref{assump:lipchitiz_hessian,assump:random_projector,assump:bounded_epsilon,assump:bounded_conditionning} hold. Then, $\forall y\in\mathbb{R}^d$, \cref{alg:type1} ensures
    \begin{align*}
        \mathbb{E}f(x_+) \leq & \left(1-\frac{N}{d}\right)f(x) + \frac{N}{d} f(y)  + \frac{N}{d}\frac{(M+L)}{6}\|y-x\|^3 \\
        & + \frac{N}{d}\frac{\|y-x\|^2}{2} \left(\delta \frac{L \kappa +  M\kappa^2}{2} +\frac{(d-N)}{d}\|\nabla^2 f(x)\|\right)
    \end{align*}
\end{proposition}
\begin{proof}
    The proof is the same as for \cref{prop:global_lower_bound_xplus}, until equation \cref{eq:temp_global_lower_bound},
    \begin{align*}
        f(x_+) \leq \min_{y\in\mathbb{R}^d} \;\; & f(x) + \nabla f(x)^TP(y-x) + \frac{1}{2} (P(y-x))^T\nabla^2 f(x)P(y-x) \nonumber\\
        & + \frac{\|P(y-x)\|^2}{4}\frac{\|\varepsilon\|}{\|D\|}\left(L \kappa_D +  M\kappa_D^2\right) + \frac{M}{6}\|P(y-x)\|^3.
    \end{align*}
    With \cref{assump:random_projector}, the following relations hold (see \citep[lemma 5.7]{hanzely2020stochastic})
    \begin{align}
        \mathbb{E}[\|P(y-x)\|^2] & = (y-x)^T\mathbb{E}[P](y-x) = \frac{N}{d}\|y-x\|^2,\\
        \mathbb{E}[\|P(y-x)\|^3] & \leq \mathbb{E}[\|P(y-x)\|^2]\|y-x\| = \frac{N}{d}\|y-x\|^2,\\
        \mathbb{E}[(y-x)^TP\nabla^ 2f(x)P(y-x)] & \leq \frac{N^2}{d^2} (y-x)\nabla^2 f(x)(y-x) + \frac{N(d-N)}{d^2}\|\nabla^2 f(x)\|\|y-x\|^2
    \end{align}
    Hence, removing the minimum and taking the expectation of \cref{eq:temp_global_lower_bound} gives
    \begin{align*}
        \mathbb{E}f(x_+) \leq & f(x) + \frac{N}{d}\nabla f(x)^T(y-x) \\
        & + \frac{1}{2} \left( \frac{N^2}{d^2} (y-x)\nabla^2 f(x)(y-x) + \frac{N(d-N)}{d^2}\|\nabla^2 f(x)\|\|y-x\|^2 \right)\\
        & + \frac{N}{d}\frac{\|y-x\|^2}{4}\frac{\|\varepsilon\|}{\|D\|}\left(L \kappa_D +  M\kappa_D^2\right) + \frac{N}{d}\frac{M}{6}\|y-x\|^3.
    \end{align*}
    Using the lower bound from \cref{eq:ineq_function}
    \[
        \frac{1}{2}(y-x)\nabla^2 f(x)(y-x) \leq f(y)+\frac{L}{6}\|y-x\|^3-f(x) -\nabla f(x)(y-x)
    \]
    in the inequality over the expectation gives
    \begin{align*}
        \mathbb{E}f(x_+) \leq & f(x) + \frac{N}{d}\nabla f(x)^T(y-x) \\
        & + \frac{N^2}{d^2}\left(f(y)+\frac{L}{6}\|y-x\|^3-f(x) -\nabla f(x)(y-x)\right)\\
        & +\frac{1}{2}\frac{N(d-N)}{d^2}\|\nabla^2 f(x)\|\|y-x\|^2 \\
        & + \frac{N}{d}\frac{\|y-x\|^2}{4}\frac{\|\varepsilon\|}{\|D\|}\left(L \kappa_D +  M\kappa_D^2\right) + \frac{N}{d}\frac{M}{6}\|y-x\|^3.
    \end{align*}
    After simplification,
    \begin{align*}
        \mathbb{E}f(x_+) \leq & \left(1-\frac{N^2}{d^2}\right)f(x) + \frac{N^2}{d^2} f(y)+ \frac{N}{d}\left(1-\frac{N}{d}\right)\nabla f(x)^T(y-x) \\
        & +\frac{1}{2}\frac{N(d-N)}{d^2}\|\nabla^2 f(x)\|\|y-x\|^2 \\
        & + \frac{N}{d}\frac{\|y-x\|^2}{4}\frac{\|\varepsilon\|}{\|D\|}\left(L \kappa_D +  M\kappa_D^2\right) + \left(\frac{N^2L}{6d^2}+\frac{NM}{6d}\right)\|y-x\|^3.
    \end{align*}
    To simplify the expression, since $N\leq d$,
    \[
        \left(\frac{N^2L}{6d^2}+\frac{NM}{6d}\right)\|y-x\|^3 \leq \frac{N(M+L)}{6d}\|y-x\|^3.
    \]
    Finally, since the function is convex,
    \[
        \frac{N}{d}\left(1-\frac{N}{d}\right)\nabla f(x)^T(y-x) \leq \frac{N}{d}\left(1-\frac{N}{d}\right)(f(y)-f(x)).
    \]
    From this last relation, \cref{assump:bounded_epsilon} and \cref{assump:bounded_conditionning} comes the desired result,
    \begin{align*}
        \mathbb{E}f(x_+) \leq & \left(1-\frac{N}{d}\right)f(x) + \frac{N}{d} f(y)  + \frac{N(M+L)}{6d}\|y-x\|^3 \\
        & + \frac{\|y-x\|^2}{2} \left(\frac{N}{d}\delta \frac{L \kappa +  M\kappa^2}{2} +\frac{N(d-N)}{d^2}\|\nabla^2 f(x)\|\right)
    \end{align*}
\end{proof}

\subsection{Technical Results: Accelerated Algorithm} 

\paragraph{Notations} The following functions define the estimate sequence,
\begin{align}
    \ell_t(x) &= \sum_{i=2}^t b_{i-1} \left(f(x_i) + \nabla f(x_i)(x-x_i)\right), \\
    \phi_t(x) &= f(x_1) + \ell_t(x) + \frac{\lambda_t^{(1)}}{2} \|x-x_0\|^2 + \frac{\lambda_t^{(2)}}{6} \|x-x_0\|^3\\
    \Phi_t(x) &= \frac{\phi_t(x)}{B_t},
\end{align}
where $\lambda_t^{(1,2)}$ are non-negative and increasing, and the sequences $b_t$, $B_t$ are
\begin{align}
    B_t & = \frac{t(t+1)(t+2)}{6} = \sum_{i=1}^t b_i,\\
    b_t & = \frac{(t+1)(t+2)}{2} = B_{t+1}-B_{t}.\\
\end{align}
Moreover, the following quantities will be important later,
\begin{align}
    v_t & = \argmin_x \phi_t(x) = \argmin_{x} \Phi_t(x), \label{eq:def_v}\\
    \beta_t &= \frac{b_t}{B_{t+1}},\label{eq:def_beta}\\
    y_t & = (1-\beta_t) x_t + \beta_t v_t. \label{eq:def_y}
\end{align}

\begin{lemma} \label{lem:breg_norm}
From \citep[Lemma 4]{nesterov2008accelerating}. The Bregman divergence of the function $\|x\|^i$ satisfies, for $i\geq 2$,
\[
    \|x\|^i-\|y\|^i - \nabla (\|y\|^i)(x-y) \geq \frac{1}{2^{i-2}}\|x-y\|^i.
\]    
\end{lemma}

\begin{proposition}
    The function $\phi_t$ is lower-bounded by
    \begin{equation}\label{eq:lower_bound_phi_star}
        \phi_t \geq \underbrace{\phi_t(v_t)}_{=\phi_t^\star} + \frac{\lambda_t^{(1)}}{2} \|x-v_t\|^2 + \frac{\lambda_t^{(2)}}{12} \|x-v_t\|^3  
    \end{equation}
    where $v_t = \argmin_x \phi_t(x)$.
\end{proposition}
\begin{proof}
    The first order condition on $\phi_t$ reads,
    \[
        \ell'_t + \nabla \left(\frac{\lambda_t^{(1)}}{2} \|v_t-x_0\|^2 + \frac{\lambda_t^{(2)}}{6}\|v_t-x_0\|^3\right) = 0.
    \]
    Multiplying both sides by $(x-v_t)$ gives
    \[
        \ell'_t(x-v_t) + \nabla \left(\frac{\lambda_t^{(1)}}{2} \|v_t-x_0\|^2 + \frac{\lambda_t^{(2)}}{6}\|v_t-x_0\|^3\right)(x-v_t) = 0.
    \]
    Note that, since $\ell_t$ is an affine function, $\ell'_t(x-v_t) = \ell_t(x)-\ell_t(v_t)$. Hence,
    \[
        \ell_t(x)-\ell_t(v_t) + \nabla \left(\frac{\lambda_t^{(1)}}{2} \|v_t-x_0\|^2 + \frac{\lambda_t^{(2)}}{6}\|v_t-x_0\|^3\right)(x-v_t) = 0.
    \]
    Finally, adding $\frac{\lambda_t^{(1)}}{2} \|x-x_0\|^2 + \frac{\lambda_t^{(2)}}{6}\|x-x_0\|^3$ on both sides  and after reorganizing the terms,
    \begin{equation}\label{eq:temp_lower_bound_phi}
        \textstyle \phi_t(x) = \ell_t(v_t) + \frac{\lambda_t^{(1)}}{2} \|x-x_0\|^2 + \frac{\lambda_t^{(2)}}{6}\|x-x_0\|^3 - \nabla \left(\frac{\lambda_t^{(1)}}{2} \|v_t-x_0\|^2 + \frac{\lambda_t^{(2)}}{6}\|v_t-x_0\|^3\right)(x-v_t).
    \end{equation}
    From \cref{lem:breg_norm} with $x=x-x_0$, $y=v_t-x_0$, and after reorganizing the terms,
    \[
        \|x-x_0\|^i - \nabla (\|v_t-x_0\|^i)(x-v_t) \geq \frac{1}{2^{i-2}}\|x-v_t\|^i + \|v_t-x_0\|^i.
    \]
    Therefore, using the previous inequality with $i=2$ and $i=3$, \cref{eq:temp_lower_bound_phi} becomes
    \[
        \phi_t(x) \geq \ell_t(v_t) + \frac{\lambda_t^{(1)}}{2} \|v_t-x_0\|^2 + \frac{\lambda_t^{(2)}}{6}\|v_t-x_0\|^3 + \frac{\lambda_t^{(2)}}{2} \|v_t-x\|^2 + \frac{\lambda_t^{(3)}}{12}\|v_t-x\|^3
    \]
    By definition of $\phi_t^\star = \phi_t(v_t)$,
    \[
        \phi_t(x) \geq \phi_t^\star + \frac{\lambda_t^{(1)}}{2} \|v_t-x\|^2 + \frac{\lambda_t^{(2)}}{12}\|v_t-x\|^3.
    \]
\end{proof}

\begin{proposition} \label{prop:condition_M}
    Let 
    \[
        \gamma = \frac{\kappa_D}{\|D\|}\left(\frac{3}{2}\|\varepsilon\| +2 \frac{\|(I-P)G\|}{M}\right).
    \]
    Then, under the assumptions of \cref{prop:bound_projector_hessian} the condition
    \[  
        \frac{\|f(x_+)\|^2}{M\left(\gamma+\|D\alpha\|\right)} \leq -\nabla f(x)^TD\alpha
    \]
    is guaranteed as long as $M \geq 2L$.
\end{proposition}
\begin{proof}
    The starting point is \cref{eq:bound_cancel_gamma} combined with \cref{prop:bound_projector_hessian}:
    \begin{align*}
        \left\|\left(\frac{1}{2}D_\dagger \Gamma D_\dagger^T + \frac{M\|D\alpha\|}{2}\right)D\alpha + \nabla f(x_+)\right\| & \leq \frac{L}{2}\|D\alpha\|^2 + \|D\alpha\|\left( \frac{L}{2} \frac{\|\varepsilon\|}{\|D\|} \kappa_D + \frac{\|(I-P)\nabla^2 f(x)D\alpha\|}{\|D\alpha\|}\right) \\
        & \leq \frac{L}{2}\|D\alpha\|^2 + \|D\alpha\|\left( \frac{L}{2} \frac{\|\varepsilon\|}{\|D\|} \kappa_D + \left(\|(I-P)G\| + L\|\varepsilon\|\right)\frac{\kappa_D}{\|D\|}\right)\\
        & \leq \frac{L}{2}\|D\alpha\|^2 + \|D\alpha\|\left( \frac{3L}{2} \frac{\|\varepsilon\|}{\|D\|} \kappa_D + |(I-P)G\|\frac{\kappa_D}{\|D\|}\right)
    \end{align*}
    To simplify, let $\Gamma = M D\gamma D^T $. Hence,
    \[
        \left\|M\left(\frac{\|D\alpha\|+\gamma}{2}\right)D\alpha + \nabla f(x_+)\right\| \leq \frac{L}{2}\|D\alpha\|^2 + \|D\alpha\|\left( \frac{3L}{2} \frac{\|\varepsilon\|}{\|D\|} \kappa_D + \|(I-P)G\|\frac{\kappa_D}{\|D\|}\right)
    \]
   Elevating to the square this inequality gives
   \begin{align*}
        & \left(M\left( \frac{\gamma +\|D\alpha\|}{2}\right)\right)^2 \|D\alpha\|^2 + \|\nabla f(x_+)\|^2 + 2\left(M\left(\frac{\gamma + \|D\alpha\|}{2}\right)\right) \nabla f(x_+)^TD\alpha \\
        \leq \;\;& \|D\alpha\|^2\left(\frac{L}{2}\|D\alpha\| + \frac{L}{2} \frac{\|\varepsilon\|}{\|D\|} \kappa_D + \frac{\|(I-P)\nabla^2 f(x)D\alpha\|}{\|D\alpha\|}\right)^2.
    \end{align*}
    The desired result holds if the following condition is satisfied,
    \begin{align*}
        & \left(M\left( \frac{\gamma +\|D\alpha\|}{2}\right)\right)^2 \|D\alpha\|^2
        \geq \|D\alpha\|^2\left(\frac{L}{2}\|D\alpha\| + \frac{3L}{2} \frac{\|\varepsilon\|}{\|D\|} \kappa_D + \frac{\|(I-P)G\|\kappa_D}{\|D\|}\right)^2.
    \end{align*}
    After simplification of the squares,
    \[
        M\frac{\gamma+\|D\alpha\|}{2} \geq \frac{L}{2}\|D\alpha\| + \frac{3L}{2} \frac{\|\varepsilon\|}{\|D\|} \kappa_D + \frac{\|(I-P)G\|\kappa_D}{\|D\|}.
    \]
    Replacing $\gamma$ by its value gives
    \[
        M\frac{\|D\alpha\| + \frac{ \kappa_D}{\|D\|}\left(\frac{3}{2}\|\varepsilon\| + 2\frac{\|(I-P)G\|}{M}\right)}{2} \geq \frac{L}{2}\|D\alpha\| + \frac{3L}{2} \frac{\|\varepsilon\|}{\|D\|} \kappa_D + \frac{\|(I-P)G\|\kappa_D}{\|D\|}.
    \]
    The condition is simplified into
    \[
        (M-L)\frac{\|D\alpha\|}{2} + (M-2L) \frac{3}{2} \frac{\|\varepsilon \|  \kappa_D}{\|D\|} \geq 0.
    \]
    This condition is implied by $M\geq 2L$. 
\end{proof}

\begin{proposition}\label{prop:cond_large_step}
    Under the same assumptions as \cref{prop:ineq_grad_qn_optimal_no_b}, if $M\geq 2L$, and if
    \[
        \gamma = \frac{\kappa_D}{\|D\|}\left(\frac{3}{2}\|\varepsilon\| + 2\frac{\|(I-P)G\|}{M}\right) \leq \frac{(\sqrt{3}-1)\|D\alpha\|}{4},
    \]
    then
    \[
        \frac{2}{3^{3/4}}\frac{\|\nabla f(x_+)\|^{3/2}}{\sqrt{M}} \leq -\nabla f(x_+)^TD\alpha.
    \]
\end{proposition}
\begin{proof}
    The starting point is \cref{eq:bound_cancel_gamma},
    \begin{align*}
        \left\|  M\frac{\|D\alpha\|}{2}D\alpha + \nabla f(x_+)\right\|\leq \frac{L}{2} \|D\alpha\|^2 +\|D\alpha\| \left(\frac{L}{2}\frac{\kappa_D\|\varepsilon\|}{\|D\|} + \frac{M\gamma}{2} + \frac{\left\|(I-P)\nabla^2 f(x)D\alpha \right\|}{\|D\alpha\|} \right)
    \end{align*}
    Therefore, to obtain
    \begin{align*}
        \left\|  M\frac{\|D\alpha\|}{2}D\alpha + \nabla f(x_+)\right\|\leq M\left(\frac{\|D\alpha\|}{4}+\gamma\right)\|D\alpha\| ,
    \end{align*}
    The following is sufficient,
    \[
        M\left(\frac{\|D\alpha\|}{4}+\gamma\right)\|D\alpha\|  \geq \frac{L}{2} \|D\alpha\|^2 +\|D\alpha\| \left(\frac{L}{2}\frac{\kappa_D\|\varepsilon\|}{\|D\|} + \frac{M\gamma}{2} + \frac{\left\|(I-P)\nabla^2 f(x)D\alpha \right\|}{\|D\alpha\|} \right).
    \]
    Using \cref{prop:bound_projector_hessian}, the condition can be strengthened into
    \begin{align*}
        & \frac{M}{2}\left(\frac{\|D\alpha\|+\gamma}{2}\right) \|D\alpha\| \\
        & \geq \frac{L}{2} \|D\alpha\|^2 +\|D\alpha\| \left(\frac{L}{2}\frac{\kappa_D\|\varepsilon\|}{\|D\|} + \frac{M\gamma}{2} + \left(\|(I-P)G\| + L\|\varepsilon\|\right)\frac{\kappa_D}{\|D\|}\right)\\
        & = \frac{L}{2} \|D\alpha\|^2 +\|D\alpha\| \left(\frac{3L}{2}\frac{\kappa_D\|\varepsilon\|}{\|D\|} + \frac{M\gamma}{2} + \|(I-P)G\| \frac{\kappa_D}{\|D\|}\right)
    \end{align*}
    Defining
    \[
        \frac{\gamma}{2} = \left(\frac{3}{4}\frac{\kappa_D\|\varepsilon\|}{\|D\|} + \frac{\|(I-P)G\| \frac{\kappa_D}{\|D\|}}{M}\right)
    \]
    simplifies the condition into
    \[
        M\left(\frac{\|D\alpha\|}{4}+\gamma\right) \|D\alpha\| \geq \frac{L}{2} \|D\alpha\|^2 +\|D\alpha\|\left(M\gamma + \frac{3(L-\frac{M}{2})}{2}\frac{\kappa_D\|\varepsilon\|}{\|D\|} \right)
    \]
    which is satisfied when $M>2L$. Now, assume that 
    \[
        \gamma \leq \frac{(\sqrt{3}-1)\|D\alpha\|}{4}.
    \]
    Then,
    \[
        \left\|  M\frac{\|D\alpha\|}{2}D\alpha + \nabla f(x_+)\right\|\leq \sqrt{3}\frac{M\|D\alpha\|^2}{4}.
    \]
    Elevating both sides to the square gives
    \[
        \|\nabla f(x_+)\|^2 + \frac{3M^2\|D\alpha\|^4}{16}\leq -M\|D\alpha\|\nabla f(x_+)^TD\alpha
    \]
    Writing $r = \|D\alpha\|$,
    \[
        \frac{\|\nabla f(x_+)\|^2}{Mr} + \frac{3Mr^3}{16}\leq -\nabla f(x_+)^TD\alpha.
    \]
    Using
    \[
        \frac{c_1}{r} + c_2r^3 \geq 4c_2^{1/4}\left(\frac{c_1}{3}\right)^{3/4},
    \]
    the inequality becomes
    \begin{align*} 
        -\nabla f(x_+)^TD\alpha & \geq \frac{M^{1/4}}{2}\frac{\|\nabla f(x_+)\|^{3/2}}{M^{3/4}}\frac{4}{3^{3/4}} \\
        & = \frac{2}{3^{3/4}}\frac{\|\nabla f(x_+)\|^{3/2}}{\sqrt{M}} .
    \end{align*}
\end{proof}

\begin{proposition}[Termination of \cref{alg:subroutine_acc}]  \label{prop:termination_algo} Let $f$ satisfies \cref{assump:lipchitiz_hessian}. Assume that \cref{assump:projector_grad,assump:bounded_conditionning,assump:bounded_epsilon} holds. Then, once $M\geq 2L$, \cref{alg:subroutine_acc} terminates with \texttt{ExitFlag} equals to either \texttt{SmallStep} or \texttt{LargeStep}. Moreover, if $M_0\leq L$, then the algorithm terminates with $M\leq 4L$. Moreover, if the algorithm terminates with \texttt{ExitFlag} equals to \texttt{SmallStep}, then
\[
    \|D\alpha\| \leq \frac{4\gamma_M}{\sqrt{3}-1}, \quad \gamma_M = \frac{\kappa_D}{\|D\|}\left(\frac{3}{2}\|\varepsilon\| + 2\frac{\|(I-P)G\|}{M}\right).
\]
\end{proposition}
\begin{proof}
    Let 
    \[
        \gamma_M = \frac{\kappa_D}{\|D\|}\left(\frac{3}{2}\|\varepsilon\| + 2\frac{\|(I-P)G\|}{M}\right).
    \]
    Assume that $M\geq 2L$. If $\gamma_M \leq \frac{(\sqrt{3}-1)\|D\alpha\|}{4}$, then, by \cref{prop:cond_large_step}, the following condition is satisfied:
    \[
        \frac{2}{3^{3/4}}\frac{\|\nabla f(x_+)\|^{3/2}}{\sqrt{M}} \leq -\nabla f(x_+)^TD\alpha.
    \]
    In this case the algorithm terminates with \texttt{ExitFlag = LargeStep}. In any case, by \cref{prop:condition_M,}, the following conditions is always satisfied when $M\geq 2L$:
    \begin{align*}
        \frac{\|f(x_+)\|^2}{M\left(\gamma+\|D\alpha\|\right)} & \leq -\nabla f(x)^TD\alpha.
    \end{align*}
    Then, if $\gamma_M \geq \frac{(\sqrt{3}-1)\|D\alpha\|}{4}$, the algorithm terminates with \texttt{ExitFlag = SmallStep} (otherwise the algorithm would have been terminated with \texttt{ExitFlag = LargeStep}).

    Since the algorithm doubles $M$ until one of the two condition is satisfied, in the worst case, $M=4L$.
\end{proof}

\begin{proposition} \label{prop:recursion_phi}
    If $\lambda_t^{(1)}$ and $\lambda_t^{(2)}$ satisfy
    \[
        \lambda^{(1)}_t \geq \frac{b_{t+1}^2}{B_t}M_{t+1}\left(\gamma_t+\|D_t\alpha_t\|\right),\quad \lambda^{(2)}_t \geq \frac{4}{\sqrt{3}}\frac{b_{t+1}^{3}}{B^2_t}M_{t+1},
    \]
    where $\gamma_t = \frac{\kappa_{D_t}}{\|D_t\|}\left(\frac{3}{2}\|\varepsilon_t\| + 2\frac{\|(I-P_t)G_t\|}{M_{t+1}}\right).$
    Then, the function $\phi$ satisfies
    \[
        B_t f(x_t) \leq \phi_t(x),\qquad \phi_t(x) \leq B_t f(x) + \frac{\lambda_t^{(1)}+\tilde\lambda^{(1)}}{2}\|x-x_0\|^2 + \frac{\lambda_t^{(2)}+\tilde\lambda^{(2)}}{6}\|x-x_0\|^3,
    \]
    where
    \[
        \tilde\lambda^{(1)} = \|\nabla f(x_0)-P_0\nabla f(x_0)P_0\|+\delta\left(\frac{L\kappa+M_1\kappa^2}{2}\right),\quad \tilde\lambda^{(2)} = M_1+L.
    \]
\end{proposition}
\begin{proof}
    The result is proven by recursion. At $t=1$, the condition $B_t f(x_t) \leq \phi_t(x)$ is obviously satisfied since
    \[
        f(x_1) \leq \min_v \phi_1(v)= f(x_1).
    \]
    On the other hand, by \cref{prop:global_lower_bound_xplus},
    \begin{align*}
        f(x_1) &\leq \min_x f(x) + \frac{\tilde\lambda^{(2)}}{6}\|x-x_0\|^3  + \frac{\tilde\lambda^{(1)}}{2}\|x-x_0\|^2\\
        &\leq f(x) + \frac{\tilde\lambda^{(2)}}{6}\|x-x_0\|^3  + \frac{\tilde\lambda^{(1)}}{2}\|x-x_0\|^2.
    \end{align*}
    Therefore, the second condition holds by definition of $\phi$,
    \begin{align*}
        \phi_t & = f(x_1) + \frac{\lambda_t^{(1)}}{2}\|x-x_0\|^2 + \frac{\lambda_t^{(2)}}{6}\|x-x_0\|^3\\
        & \leq \frac{\lambda_1^{(1)}+\tilde\lambda^{(1)}}{2}\|x-x_0\|^2 + \frac{\lambda_1^{(2)}+\tilde\lambda^{(2)}}{6}\|x-x_0\|^3.
    \end{align*}

    Now, assume $t > 1$, and $B_{t}f(x_{t})\leq \phi_{t}(x)$. Hence,
    \begin{align*}
        & \min_x \phi_{t+1}(x) \\
        = & \min_x \ell_{t}(x) + b_t \left[ f(x_{t+1}) + \nabla f(x_{t+1})(x-x_{t+1}) \right] + \frac{\lambda_{t+1}^{(1)}}{2} \|x-x_0\|^2 + \frac{\lambda_{t+1}^{(2)}}{6} \|x-x_0\|^3\\
        = & \min_x \phi_{t}(x) + b_{t} \left[ f(x_{t+1}) + \nabla f(x_{t+1})(x-x_{t+1}) \right]  \\
        & \quad + \frac{\lambda_{t+1}^{(1)}-\lambda_{t}^{(1)}}{2} \|x-x_0\|^2 + \frac{\lambda_{t+1}^{(2)}-\lambda_{t}^{(2)}}{6} \|x-x_0\|^3\\
        \geq & \min_x \phi_{t}(x) + b_{t} \left[ f(x_{t+1}) + \nabla f(x_{t+1})(x-x_{t+1}) \right] \\
        \overset{\eqref{eq:lower_bound_phi_star}}{\geq } & \min_x \phi_{t}^\star + \frac{\lambda_{t}^{(1)}}{2} \|x-v_{t}\|^2 + \frac{\lambda_{t}^{(2)}}{12} \|x-v_{t}\|^3  +  b_{t} \left[ f(x_{t+1}) + \nabla f(x_{t+1})(x-x_{t+1}) \right]\\
        \geq & \min_x B_{t} f(x_{t}) + \frac{\lambda_{t}^{(1)}}{2} \|x-v_{t}\|^2 + \frac{\lambda_{t}^{(2)}}{12} \|x-v_{t}\|^3  +  b_{t} \left[ f(x_{t+1}) + \nabla f(x_{t+1})(x-x_{t+1}) \right]\\
        \overset{\text{A.}\ref{assump:convex}}{\geq} & \min_x B_{t} f(x_{t+1}) + \nabla f(x_{t+1})(x_{t}-x_{t+1})+  b_{t} \left[ f(x_{t+1}) + \nabla f(x_{t+1})(x-x_{t+1}) \right]\\
        & \quad +\frac{\lambda_{t}^{(1)}}{2} \|x-v_{t}\|^2 + \frac{\lambda_{t}^{(2)}}{12} \|x-v_{t}\|^3  \\
        = & \min_x B_{t+1} f(x_{t+1}) + \nabla f(x_{t+1})(B_{t}x_{t} + b_{t} x-B_{t+1}x_{t+1}) +\frac{\lambda_{t}^{(1)}}{2} \|x-v_{t}\|^2 + \frac{\lambda_{t}^{(2)}}{12} \|x-v_{t}\|^3  \\
        \overset{\eqref{eq:def_y}}{=} & \min_x B_{t+1} f(x_{t+1}) + B_{t+1}\nabla f(x_{t+1})(y_{t}-x_{t+1})\\
        & \quad + b_{t}\nabla f(x_{t+1})(x-v_{t}) +\frac{\lambda_{t}^{(1)}}{2} \|x-v_{t}\|^2 + \frac{\lambda_{t}^{(2)}}{12} \|x-v_{t}\|^3 
    \end{align*}
    The inequality is satisfied if either
    \begin{align*}
        \textbf{(a)}\quad  0 & \leq B_{t+1}\nabla f(x_{t+1})(y_{t}-x_{t+1}) + b_{t}\nabla f(x_{t+1})(x-v_{t}) + \frac{\lambda_{t}^{(2)}}{12}\|x-v_{t}\|^3, \;\; \text{or} \\
        \textbf{(\textbf{b})}\quad 0 & \leq B_{t+1}\nabla f(x_{t+1})(y_{t}-x_{t+1}) + b_{t}\nabla f(x_{t+1})(x-v_{t}) +\frac{\lambda_{t}^{(1)}}{2} \|x-v_{t}\|^2.
    \end{align*}
    It remains now to find \textit{sufficient condition} such that one of the previous inequalities hold.
    
    Define $x_{t+1}$ to be the output of \cref{alg:subroutine_acc} starting from $y_{t}$, hence $y_{t}-x_{t+1} = -D_{t}\alpha_{t}$. The algorithm guarantees that
    \begin{align}
         \textbf{(a)}\quad-\nabla f(x_{t+1})^TD_t\alpha_t & \geq \frac{2}{3^{3/4}}\frac{\|\nabla f(x_{t+1})\|^{3/2}}{\sqrt{M_{t+1}}} \quad \text{and}\;\; \quad \text{or}\\
         \textbf{(b)}\quad-\nabla f(x_{t+1})^TD_t\alpha_t &\geq \frac{\|f(x_{t+1})\|^2}{M_{t+1}\left(\gamma_t+\|D_t\alpha_t\|\right)}
    \end{align}
    Combining the expressions $\textbf{(a)}$ and $\textbf{(b)}$ leads to the following sufficient conditions:
    \begin{align}
        0 & \leq B_{t+1}\frac{2}{3^{3/4}}\frac{\|\nabla f(x_{t+1})\|^{3/2}}{\sqrt{M_{t+1}}} + b_{t}\nabla f(x_{t+1})(x-v_{t}) + \frac{\lambda_{t}^{(2)}}{12}\|x-v_{t}\|^3, \label{eq:temp_condition_1_acc} \\
        \quad 0 & \leq B_{t+1}\frac{\|f(x_{t+1})\|^2}{M_{t+1}\left(\gamma_t+\|D_t\alpha_t\|\right)} + b_{t}\nabla f(x_{t+1})(x-v_{t}) +\frac{\lambda_{t}^{(1)}}{2} \|x-v_{t}\|^2.\label{eq:temp_condition_2_acc}
    \end{align}
    \paragraph{Case 1: equation \eqref{eq:temp_condition_1_acc}.} Starting from the first order condition of the minimum of \cref{eq:temp_condition_1_acc} over $x$,
    \begin{equation}
        \label{eq:first_order_condition_1_acc}
        b_{t}\nabla f(x_{t+1}) + \frac{\lambda_{t}^{(2)}}{4}\|x-v_{t}\|(x-v_t)=0.
    \end{equation} 
    Multiplying \cref{eq:first_order_condition_1_acc} by $(x-v_t)$ gives
    \[
        b_{t}\nabla f(x_{t+1})(x-v_t) = - \frac{\lambda_{t}^{(2)}}{4}\|x-v_{t}\|^3
    \]
    Hence, when $x$ satisfies \cref{eq:first_order_condition_1_acc},
    \begin{equation}\label{eq:first_order_condition_1_acc_simple}
        b_{t}\nabla f(x_{t+1})(x-v_{t}) + \frac{\lambda_{t}^{(2)}}{12}\|x-v_{t}\|^3 = -\frac{\lambda_{t}^{(2)}}{6}\|x-v_{t}\|^3 .
    \end{equation}
    Going back to \cref{eq:first_order_condition_1_acc}, after isolating $x-v_t$,
    \begin{align*}
        (x-v_t) = -\frac{4b_{t}}{\lambda_{t}^{(2)}}\nabla f(x_{t+1}) \frac{1}{\|x-v_{t}\|}
    \end{align*}
    Therefore, after taking the norm and changing the power,
    \begin{align*}
         \|x-v_t\|^3 & = \left(\frac{4b_{t}}{\lambda_{t}^{(2)}}\|\nabla f(x_{t+1})\|\right)^{3/2},\\
        \Leftrightarrow  
        \frac{\lambda^{(2)}_t}{6}\|x-v_t\|^3 & = \frac{\lambda^{(2)}_t}{6}\left(\frac{4b_{t}}{\lambda_{t}^{(2)}}\|\nabla f(x_{t+1})\|\right)^{3/2}\\
         & = \frac{4}{3\sqrt{\lambda^{(2)}_t}}\left(b_{t}\|\nabla f(x_{t+1})\|\right)^{3/2}.
    \end{align*}
    After using \cref{eq:first_order_condition_1_acc_simple} and injecting the minimal value makes the condition \cref{eq:temp_condition_1_acc} stronger:
    \[
        0 \leq B_{t+1}\frac{2}{3^{3/4}}\frac{\|\nabla f(x_{t+1})\|^{3/2}}{\sqrt{M_{t+1}}} - \frac{4}{3\sqrt{\lambda^{(2)}_t}}\left(b_{t}\|\nabla f(x_{t+1})\|\right)^{3/2}.
    \]
    Hence, if $\lambda^{(2)}_{t}$ satisfies
    \begin{equation}
        B_{t+1}\frac{2}{3^{3/4}\sqrt{M_{t+1}}} \geq \frac{4}{3\sqrt{\lambda^{(2)}_t}}b_{t}^{(3/2)}\quad \Leftrightarrow \quad \lambda^{(2)}_t \geq \frac{4}{\sqrt{3}}\frac{b_{t}^{3}}{B^2_{t+1}}M_{t+1},
    \end{equation}
    then \cref{eq:temp_condition_1_acc} is satisfied.

    \paragraph{Case 2: equation \eqref{eq:temp_condition_2_acc}.} Starting from the first order condition of the minimum of \cref{eq:temp_condition_2_acc} over $x$,
    \begin{equation}
        \label{eq:first_order_condition_2_acc}
        b_{t+1}\nabla f(x_{t+1})+\lambda_{t}^{(1)} (x-v_{t}).
    \end{equation} 
    Hence,
    \[
        (x-v_{t}) = -\frac{b_{t}\nabla f(x_{t+1})}{\lambda_{t}^{(1)}}.
    \]
    Injecting the value back in \cref{eq:temp_condition_2_acc} gives
    \[
        B_{t+1}\frac{\|f(x_{t+1})\|^2}{M\left(\gamma_t+\|D_t\alpha_t\|\right)} - b^2_{t}\frac{\|\nabla f(x_{t+1})\|^2}{\lambda_{t}^{(1)}} +\frac{1}{2} b^2_{t} \frac{\|\nabla f(x_{t+1})\|^2}{\lambda_{t}^{(1)}}.
    \]
    Therefore, if the following condition holds,
    \[
        \frac{B_{t+1}}{2M_{t+1}\left(\gamma_t+\|D_t\alpha_t\|\right)} \geq \frac{b^2_{t}}{\lambda_{t}^{(1)}} \quad \Leftrightarrow \quad \lambda^{(1)}_t \geq \frac{b_{t}^2}{2B_{t+1}}M_{t+1}\left(\gamma_t+\|D_t\alpha_t\|\right),
    \]
    then \cref{eq:temp_condition_2_acc} is satisfied.
    
\end{proof}

\begin{proposition} \label{prop:bound_lambdas}
     Let $f$ satisfies \cref{assump:lipchitiz_hessian}. Then, under \cref{assump:bounded_conditionning,assump:bounded_epsilon,assump:projector_grad}, $\lambda_t^{(1)}$ and $\lambda_t^{(2)}$ in \cref{alg:accelerated_algo} are bounded by
     \begin{align}
         \lambda_t^{(1)} & \leq 30\cdot\frac{b_{t+1}^2}{B_t} \kappa_D\left(\delta\max\{4L,M_0\} + \max_{i=0\ldots t} \|(I-P_i)\nabla f(x_i)P_i)\|\right)\\
         \lambda_t^{(2)} & \leq \frac{L}{2} \delta + \max_{i=0\ldots t} \|(I-P_i)\nabla f(x_i)P_i)\|.
     \end{align}
\end{proposition}

\begin{proof}
    Since \cref{alg:accelerated_algo} doubles $\lambda_t^{(1)}$, $\lambda_t^{(2)}$ until $\phi_t^\star \geq f(x_{t+1})$, then by \cref{prop:recursion_phi}, both $\lambda_t^{(1)}$, $\lambda_t^{(2)}$ achieves at most
    \[
        \lambda^{(1)}_t \leq 2\cdot\frac{b_{t+1}^2}{B_t}M_{t+1}\left(\gamma_t+\|D_t\alpha_t\|\right),\quad \lambda^{(2)}_t \leq 2\cdot\frac{4}{\sqrt{3}}\frac{b_{t+1}^{3}}{B^2_t}M_{t+1}.
    \]
    There are three cases to distinguish:
    \begin{enumerate}
        \item The algorithm finishes with \texttt{ExitFlag = LargeStep},
        \item The algorithm finishes with \texttt{ExitFlag = SmallStep}.
    \end{enumerate}

    \paragraph{Case 1.} In this case, $\lambda^{(2)}_{t+1}$ may be updated. By proposition \cref{prop:termination_algo}, $M_{t}\leq 4L$ (unless $M_0 \geq 4L)$. Hence, $\lambda^{(2)}_t$ is bounded by
    \[
        \lambda^{(2)}_t \leq 2\cdot\frac{4}{\sqrt{3}}\frac{b_{t+1}^{3}}{B^2_t}\max\{M_0,4L\} \leq 5\frac{b_{t+1}^{3}}{B^2_t}\max\{M_0,4L\}.
    \]
    \paragraph{Case 2.} In this case, $\lambda^{(1)}_{t+1}$ may be updated. By \cref{prop:termination_algo}, and by \cref{assump:bounded_conditionning,assump:bounded_epsilon}, 
    \begin{align*}
        M_{t+1}\left(\gamma_t+\|D_t\alpha_t\|\right) & \leq \frac{\sqrt{3}+1}{\sqrt{3}-1}M_{t+1}\gamma_t\\
        & = \frac{\sqrt{3}+1}{\sqrt{3}-1}\frac{\kappa_{D_t}}{\|D_t\|}\left(\frac{3}{2}\|\varepsilon_t\|M_{t+1} + 2\|(I-P_t)G_t\|\right),\\
        & \leq \frac{\sqrt{3}+1}{\sqrt{3}-1}\left(\frac{3}{2}\delta\kappa_D\max\{4L,M_0\} + 2\kappa_D\frac{\|(I-P_t)G_t\|}{\|D_t\|}\right).
    \end{align*}
    In addition, by \cref{thm:bound_secant_alpha} and \cref{assump:bounded_epsilon},
    \begin{align*}
        \frac{\|(I-P_t)G_t\|}{\|D_t\|} & \leq \frac{\|(I-P_t)(G_t-\nabla f(x_t)D_t)\| + \|(I-P_t)\nabla f(x_t)D_t)\|}{\|D_t\|}\\
        & \leq \frac{\frac{L}{2} \|\varepsilon_t\| + \|(I-P_t)\nabla f(x_t)D_t)\|}{\|D_t\|},\\
        & = \frac{\frac{L}{2} \|\varepsilon_t\| + \|(I-P_t)\nabla f(x_t)P_tD_t)\|}{\|D_t\|},\\
        & \leq \frac{L}{2} \delta + \max_{i=0\ldots t} \|(I-P_i)\nabla f(x_i)P_i)\|.
    \end{align*}
    Hence,
    \begin{align*}
        & M_{t+1}\left(\gamma_t+\|D_t\alpha_t\|\right) \\
        & \leq \frac{\sqrt{3}+1}{\sqrt{3}-1}\left(\frac{3}{2}\delta\kappa_D\max\{4L,M_0\} + 2\kappa_D\left(\frac{L}{2} \delta + \max_{i=0\ldots t} \|(I-P_i)\nabla f(x_i)P_i)\|\right)\right),\\
        & \leq \frac{\sqrt{3}+1}{\sqrt{3}-1}\left(\frac{7}{4}\delta\kappa_D\max\{4L,M_0\} + 2\kappa_D\max_{i=0\ldots t} \|(I-P_i)\nabla f(x_i)P_i)\|\right).\\
        & \leq 7.5\kappa_D\left(\delta\max\{4L,M_0\} + \max_{i=0\ldots t} \|(I-P_i)\nabla f(x_i)P_i)\|\right).
    \end{align*}
    Therefore,
    \begin{align*}
        \lambda^{(1)}_t & \leq 2\cdot\frac{b_{t+1}^2}{B_t}M_{t+1}\left(\gamma_t+\|D_t\alpha_t\|\right)\\
        & \leq 30\cdot\frac{b_{t+1}^2}{B_t} \kappa_D\left(\delta\max\{4L,M_0\} + \max_{i=0\ldots t} \|(I-P_i)\nabla f(x_i)P_i)\|\right)
    \end{align*}
\end{proof}

\subsection{Missing proofs from Sections \ref{sec:detail_construction_algo} and \ref{sec:rate_convergence}}

\thmboundsecantalpha*
\begin{proof}

Using Cauchy-Schwartz with \cref{eq:ineq_secant} gives that, for all $v$,
\[
    v^T\left(\nabla f(y)-\nabla f(z)-\nabla^2f(z)(y-z)\right) \leq \frac{L\|v\|}{2} \|y-z\|^2.
 \]
Let $v=v_i$, $y = y_i$, and $z=z_i$. By the definition of $Y,\,Z,\,D,\,G$ in \cref{def:matrices},
\[
    v_i^T\left(g_i-\nabla^2f(z_i)d_i\right) \leq \frac{L\|v_i\|}{2} \|d_i\|^2.
\]
Introducing $\nabla^2 f(x)$ gives
\begin{align*}
    v_i^T\left(g_i-\nabla^2f(z_i)d_i\right) & = v_i^T\left(g_i-\nabla^2f(x)d_i\right) + v_i^T(\nabla^2f(z_i)-\nabla^2f(x))d_i.
\end{align*}
Since the Hessian is $L$-Lipchitz-continuous \cref{assump:lipchitiz_hessian}, $(\nabla^2f(z_i)-\nabla^2f(x))d_i \leq L\|d_i\|\|z_i-x\|$. Therefore, by the definition of $\varepsilon_i$,
\begin{align}
    v_i^T\left(g_i-\nabla^2f(x)d_i\right) & \leq \frac{L\|v_i\|\varepsilon_i}{2}. \label{eq:temp_bound_thm1}
\end{align}
Let $v_i=\sign(\alpha_i)w$. Summing all inequalities multiplied by $|\alpha_i|$ gives the first desired result:
\begin{align*}
    w^T\left(G-\nabla^2f(x)D\right)\alpha & \leq \frac{L\|w\|\sum_{i=1}^N\varepsilon_i|\alpha_i|}{2}.
\end{align*}
The second result is rather straightforward, since \cref{eq:temp_bound_thm1} with $v_i=w$ gives
\[ 
    w^T\left(g_i-\nabla^2f(x)d_i\right) \leq \frac{L\|w\|\varepsilon_i}{2}.
\]
Therefore,
\[ 
    \sqrt{\sum_{i=1}^N \left(w^T\left(g_i-\nabla^2f(x)d_i\right)\right)^2} \leq \|w\| \sqrt{\sum_{i=1}^N \left\|g_i-\nabla^2f(x)d_i\right\|^2}\leq \|w\| \sqrt{\sum_{i=1}^N L\varepsilon_i^2}  \leq \frac{L\|w\|\|\varepsilon\|}{2}.
\]

\end{proof}

\thmupperboundalgo*
\begin{proof}
    The inequality \cref{eq:type2_bound} is a direct consequence of \cref{eq:ineq_secant} (with $y=x_+$, $z=x$) combined with \cref{eq:bound_hessian_dalpha},
    \begin{align*}
        & \|\nabla f(x_+)-\nabla f(x)-\nabla^2f(x)D\alpha\| \leq \frac{L}{2} \|D\alpha\|^2\\
        \Leftrightarrow \;\; & w^T\left(\nabla f(x_+)-\nabla f(x)-\nabla^2f(x)D\alpha\right) \leq \frac{L\|w\|}{2} \|D\alpha\|^2\\
        \Leftrightarrow \;\; & w^T\nabla f(x_+) \leq \frac{L\|w\|}{2} \|D\alpha\|^2 + w^T\left(\nabla f(x)+\nabla^2f(x)D\alpha\right) \\
        \Leftrightarrow \;\; & w^T\nabla f(x_+) \overset{\cref{eq:bound_hessian_scalar}}{\leq} \frac{L\|w\|}{2} \left(\|D\alpha\|^2 + \sum_{i=1}^N |\alpha_i| \varepsilon_i\right)+ w^T\left(\nabla f(x)+G\alpha\right)\\
        \Leftrightarrow \;\; & w^T\nabla f(x_+) \leq \|w\|\left(\frac{L}{2} \left(\|D\alpha\|^2 + \sum_{i=1}^N |\alpha_i| \varepsilon_i\right)+ \|\nabla f(x)+G\alpha\|\right)
    \end{align*}
    Setting $w = \nabla f(x_+)$ gives \cref{eq:type2_bound}. 

    The inequality \cref{eq:type1_bound} instead comes from \cref{eq:ineq_function} combined with \cref{eq:bound_hessian_dalpha}. Indeed,
    \begin{align*}
        f(x_{+}) & \leq f(x) + \nabla f(x)D\alpha + \frac{1}{2}(D\alpha)^T\nabla^2 f(x)(D\alpha) + \frac{L}{6}\|D\alpha\|^3 \\
        & \overset{\cref{eq:bound_hessian_dalpha}}{\leq} f(x) + \nabla f(x)D\alpha + \frac{1}{2}\left((D\alpha)^TG\alpha +  \frac{L\|D\alpha\|}{2} \sum_{i=1}^N |\alpha_i| \varepsilon_i\right) + \frac{L}{6}\|D\alpha\|^3
    \end{align*}
    It remains to use the followings bounds:
    \begin{align*}
        \sum_{i=1}^N |\alpha_i| \varepsilon_i = \alpha^T (\sign(\alpha)\odot \varepsilon) & \leq \|\alpha\| \|\varepsilon\|,\\
        \|D\alpha\| & \leq \|D\| \|\alpha\|.
    \end{align*}
    All together,
    \[
        f(x_{+}) \leq f(x) + \nabla f(x)D\alpha + \frac{1}{2}(D\alpha)^TG\alpha + \frac{L}{4}\|\alpha\|^2 \|D\|\|\varepsilon\| + \frac{L}{6}\|D\alpha\|^3
    \]
    Finally, since $(D\alpha)^TG\alpha$ is a quadratic form, only the symmetric counterpart of $D^TG$ counts. That means, $(D\alpha)^TG\alpha = \alpha^T\frac{D^TG + G^TD}{2}\alpha$. Hence, writing $H = \frac{D^TG + G^TD}{2} + \mathrm{I}\frac{L}{2}\|D\|\|\varepsilon\|$ gives the desired result,
    \[
        f(x_{+}) \leq f(x) + \nabla f(x)D\alpha + \frac{\alpha^TH\alpha}{2} + \frac{L}{6}\|D\alpha\|^3.
    \]
\end{proof}

\thmforwardestimateproperties*
\begin{proof}
    The proof is done by recursion. 

    Since $D_1$ is simply a vector of norm one, $D_1^TD_1 = 1$. Moreover, $D_1 = \frac{\nabla f(x_0)}{\|\nabla f(x_0)\|}$. Hence, is obviously span $\nabla f(x_0)$.
    
    Assume that $D_{t-1}$ is orthogonal. Then, potentially removing one column does not affect its orthogonality. Therefore, $\tilde D_{t-1}$ is orthogonal. Now, consider the vector
    \begin{align*}
        v & = \frac{(I - \tilde D_{t-1}(\tilde D_{t-1}^T\tilde D_{t-1})^{-1}\tilde D_{t-1}^T) \nabla f(x_t)}{\|(I - \tilde D_{t-1}(\tilde D_{t-1}^T\tilde D_{t-1})^{-1}\tilde D_{t-1}^T) \nabla f(x_t)\|}\\
        & = \frac{(I - \tilde D_{t-1}\tilde D_{t-1}^T) \nabla f(x_t)}{\|(I - \tilde D_{t-1}\tilde D_{t-1}^T) \nabla f(x_t)\|},
    \end{align*}
    where the second equality is obtained using the orthogonality of $\tilde D_{t-1}$. This corresponds to a normalized orthogonal projection of the vector $\nabla f(x_t)$ onto the orthogonal columns span of $\tilde D_{t-1}$. Since
    \[
        \tilde D_{t-1}(I - \tilde D_{t-1}\tilde D_{t-1}^T) \nabla f(x_t) = 0,
    \]
    the matrix $D_t = [\tilde D_{t-1} v]$ is orthogonal:
    \[
        D_{t}^TD_t = \begin{bmatrix}
            I & 0 \\ 0 & 1
        \end{bmatrix}.
    \]
    Finally, the matrix $D_t$ indeed spans $\nabla f(x_t)$, since $P_t\nabla f(x_t) = \nabla f(x_t)$:
    \begin{align*}
        P_t\nabla f(x) = D_tD_t^T \nabla f(x_t) & = \tilde D_{t-1}\tilde D_{t-1}^T \nabla f(x_t) + vv^T \nabla f(x_t)\\
        & = \tilde D_{t-1}\tilde D_{t-1}^T \nabla f(x_t) +  \frac{(I - \tilde D_{t-1}\tilde D_{t-1}^T)\nabla f(x_t) \nabla f(x_t)^T(I - \tilde D_{t-1}\tilde D_{t-1}^T)}{\|(I - \tilde D_{t-1}\tilde D_{t-1}^T) \nabla f(x_t)\|^2} \nabla f(x_t)\\
        & = \tilde D_{t-1}\tilde D_{t-1}^T \nabla f(x_t) + (I - \tilde D_{t-1}\tilde D_{t-1}^T)\nabla f(x_t)  \frac{\nabla f(x_t)^T(I - \tilde D_{t-1}\tilde D_{t-1}^T)^2\nabla f(x_t)}{\|(I - \tilde D_{t-1}\tilde D_{t-1}^T) \nabla f(x_t)\|^2} \\
        & = \tilde D_{t-1}\tilde D_{t-1}^T \nabla f(x_t) + (I - \tilde D_{t-1}\tilde D_{t-1}^T)\nabla f(x_t)\\
        & = \nabla f(x_t).
    \end{align*}

\end{proof}

\thmminimaldecrease*
\begin{proof}
    Using \cref{eq:first_order_condition}, at each iteration, after the while loop, the first-order condition of the subroutine \cref{alg:type1} reads
    \begin{align} \label{eq:temp_first_order}
        D_t^T\nabla f(x_t) + H_t\alpha_{t+1} + \frac{M_{t+1}}{2}D_t^TD_t\alpha_{t+1} \|D_t\alpha_{t+1}\| = 0.
    \end{align}
    The subscript $t$ is dropped for clarity. After multiplying by $\alpha$,
    \[
        \nabla f(x_t)^TD\alpha + \alpha^TH\alpha + \frac{M}{2}\|D\alpha\|^3 = 0.
    \]
    In addition, multiplying both times by $\alpha$ the second-order condition \cref{eq:second_order_condition} gives 
    \[
        \alpha^TH\alpha \geq -\frac{M}{2}\|D\alpha\|^3.
    \]
    which gives, after replacing it in \cref{eq:temp_first_order},
    \begin{align} \label{eq:temp_ineq_first_order}
        \nabla f(x_t)^TD\alpha  \leq -\frac{M}{2}\|D\alpha\|^3 +\frac{M}{2}\|D\alpha\|^3 = 0.
    \end{align}
    Injecting \cref{eq:temp_first_order,eq:temp_ineq_first_order} into the while condition of \cref{alg:type1} gives the desired result:
    \begin{align}
        f(x_+) & \leq f(x) + \nabla f(x)^TD\alpha + \frac{1}{2} \alpha^TH\alpha + \frac{M\|D\alpha\|^3}{6},\label{eq:temp_ub_fun}\\
        & = f(x) -\frac{1}{2} \nabla f(x)^TD\alpha  - \frac{M\|D\alpha\|^3}{12} \nonumber\\
        & \leq f(x) - \frac{M\|D\alpha\|^3}{12}.\nonumber
    \end{align}
    Where \cref{eq:temp_ub_fun} is guaranteed if $M>L$. Therefore, in the worst case, $M<2L$. Finally, after $t$ iterations, the number of total gradient calls is bounded by $2t + \log_2\left( \frac{M_0}{L} \right)$ as shown in \citep{nesterov2006cubic}.
\end{proof}

\thmratenonconvex*
\begin{proof}
    The starting inequality is \cref{eq:lower_bound_gradient_plus}:
    \[
        \|\nabla f(x_+)\| \leq \frac{L+M}{2}\|D\alpha\|^2 + \|D\alpha\|\left( \frac{\|\varepsilon\|}{\|D\|} \left(\frac{L + M\kappa_D}{2} \right) \kappa_D + \|(I-P)\nabla^2 f(x)P\|  \right).
    \]
    The result is obtained by decomposing the inequality using a maximum,
    \begin{align*}
        & \|\nabla f(x_+)\| \\
        & \leq \max\left\{ (L+M)\|D\alpha\|^2 \;;\; 2\|D\alpha\|\left( \frac{\|\varepsilon\|}{\|D\|} \left(\frac{L + M\kappa_D}{2} \right) \kappa_D + \|(I-P)\nabla^2 f(x)P\|  \right)\right\}.
    \end{align*}
    In the first case,
    \begin{equation} \label{eq:temp_lower_bound_dalpha_1}
        \|D\alpha\|\geq \sqrt{\frac{\|\nabla f(x_+)\|}{L+M}},
    \end{equation}
    while in the second case,
    \[
        \|D\alpha\|\geq \frac{\|\nabla f(x_+)\|}{ \frac{\|\varepsilon\|}{\|D\|} \left(\frac{L + M\kappa_D}{2} \right) \kappa_D + \|(I-P)\nabla^2 f(x)P\| }.
    \]
    Let $C_t$ be defined as
    \[
        C_t =  \frac{\|\varepsilon_t\|}{\|D_t\|} \left(\frac{L + M_{t+1}\kappa_{D_t}}{2} \right) \kappa_{D_t} + \|(I-P_t)\nabla^2 f(x_t)P_t\|.
    \]
    Then, using \cref{assump:bounded_epsilon,assump:bounded_conditionning}, and since $M<2L$ by \cref{thm:minimal_decrease},
    \[
        C_t \leq C = \delta L\left(\frac{1 + 2\kappa}{2} \right) \kappa + \max_t \|(I-P_t)\nabla^2 f(x_t)P_t\| 
    \]
    Therefore,
    \begin{equation} \label{eq:temp_lower_bound_dalpha_2}
        \|D\alpha\|\geq \frac{\|\nabla f(x_+)\|}{C}.
    \end{equation}
    At each iteration $t$, combining \cref{eq:temp_lower_bound_dalpha_1,eq:temp_lower_bound_dalpha_2} into \cref{thm:minimal_decrease} gives
    \[
        f(x_{t})-f(x_{t+1}) \geq \frac{M_{t+1}}{12}\|\underbrace{x_{t+1}-x_t}_{=D_t\alpha_t}\|^3 \geq  \frac{M_{t+1}}{12} \min\left\{ \left(\frac{\|\nabla f(x_+)\|}{L+M_{t+1}}\right)^{3/2} \;;\; \left(\frac{\|\nabla f(x_+)\|}{C}\right)^{3} \right\}
    \]
    Therefore,
    \begin{align*}
        f(x_0)-f^\star & \geq f(x_0)-f(x_{t}) \\
        & = \sum_{i=0}^{t-1} f(x_{i})-f(x_{i+1})  \\
        & \geq \sum_{i=0}^{t-1} \left(\frac{M_{i+1}}{12}\|x_{i+1}-x_i\|^3\right) \\
        & \geq  \sum_{i=0}^{t-1} \min_t \frac{M_{i+1}}{12} \left\{ \left(\frac{\|\nabla f(x_{i+1})\|}{L+M_{i+1}}\right)^{3/2} \;;\; \left(\frac{\|\nabla f(x_{i+1})\|}{C}\right)^{3} \right\}\\
        & \geq  t \min_{i\in[0,t-1]}\frac{M_{i+1}}{12} \min\left\{ \left(\frac{\|\nabla f(x_{i+1})\|}{L+M_{i+1}}\right)^{3/2} \;;\; \left(\frac{\|\nabla f(x_{i+1})\|}{C}\right)^{3} \right\}\\
        & \geq  t  \frac{M_{\min}}{12} \min\left\{ \min_{i\in[1,t]}\left(\frac{\|\nabla f(x_{i})\|}{3L}\right)^{3/2} \;;\; \min_{i\in[1,t]}\left(\frac{\|\nabla f(x_{i})\|}{C}\right)^{3} \right\}
    \end{align*}
    After analyzing separately each case of the minimum, either
    \[
        \left(\frac{\min\limits_{i\in[1,t]}\|\nabla f(x_{i})\|}{3L}\right)^{3/2} \leq 12\frac{f(x_0)-f^\star}{tM_{\min}} \quad \text{or} \quad \left(\frac{\min\limits_{i\in[1,t]}\|\nabla f(x_{t+1})\|}{C}\right)^{3} \leq 12\frac{f(x_0)-f^\star}{tM_{\min}}.
    \]
    It remains to simplify to obtain the desired result,
    \[
        \min_{i=1\ldots t}\|\nabla f(x_{i})\| \leq \max\left\{\frac{3L}{t^{2/3}} \left(12\frac{f(x_0)-f^\star}{M_{\min}}\right)^{2/3} \; ; \; \left(\frac{C}{t^{1/3}}\right)\left(12\frac{f(x_0)-f^\star}{M_{\min}}\right)^{1/3} \right\}.
    \]
\end{proof}

\thmratestarconvex*
\begin{proof}
    Starting from the inequality in \cref{prop:global_lower_bound_xplus},
    \begin{align*}
        f(x_{t+1}) \leq f(y) + \frac{M_{t+1}+L}{6}\|y-x_t\|^3  + \frac{\|y-x_t\|^2}{2} C_2^{(t)},
    \end{align*}
    where 
    \[
        C_2^{(t)} = \|\nabla^2 f(x_t)- P_t\nabla^2 f(x_t)P_t\| + \delta\frac{L \kappa +  M_{t+1}\kappa^2}{2},
    \]
    and setting $y = (1-\beta_t) x_t + \beta_t x^\star $ and $f(x^\star)=f^\star$ gives
    \begin{align*}
        f(x_{t+1})-f^\star \leq f((1-\beta_t) x_t + \beta_t x^\star)-f^\star + \frac{M_{t+1}+L}{6}\beta_t^3\|x_t-x^\star\|^3  + \frac{\beta_t^2\|x_t-x^\star\|^2}{2} C_2^{(t)}.
    \end{align*}
    Because the function is star-convex,
    \begin{align*}
        f(x_{t+1})-f^\star \leq (1-\beta_t) (f(x_t)-f^\star) + \frac{M_{t+1}+L}{6}\beta_t^3\|x_t-x^\star\|^3  + \frac{\beta_t^2\|x_t-x^\star\|^2}{2} C_2^{(t)}.
    \end{align*}
    Since \cref{alg:type1} ensure a decrease in the function value, the iterate $x_t$ satisfies 
    \[
        x_t \in\{x:f(x\leq f(x_0))\},
    \]
    and therefore, $\|x_t-x^\star\|\leq R$ by \cref{assump:bounded_radius}. In addition, $M<2L$ by \cref{thm:minimal_decrease}. The inequality now becomes
    \begin{align}
        (f(x_{t+1})-f^\star) \leq (1-\beta_t) (f(x_t)-f^\star) + \beta_t^3\frac{LR^3}{2}  + \beta_t^2\frac{R^2C_2^{(t)}}{2} . \label{eq:temp_recursion_fx}
    \end{align}
    Finally, since $M<2L$, the scalar $C_2^t$ is bounded over time by $C_2$:
    \[
        C_2^{(t)} \leq C_2 \defas \delta L\frac{ \kappa +  2\kappa^2}{2} +\max_t \|\nabla^2 f(x_t)- P_t\nabla^2 f(x_t)P_t\| .
    \]
    Now, let 
    \begin{itemize}
        \item $B_t = \frac{t(t+1)(t+2)}{6}$,
        \item $b_t:B_{t}=B_{t-1}+b_t$, hence $b_t = \frac{t(t+1)}{2}$, and
        \item $\beta_t = \frac{b_{t+1}}{B_{t+1}}$.
    \end{itemize}
    Therefore, for $t\geq 1$,
    \[
        1=\frac{B_t}{B_{t}} = \frac{B_{t-1}}{B_{t}}+\frac{b_t}{B_{t}} =\frac{B_{t-1}}{B_{t}}+\beta_{t-1} \quad \Rightarrow\quad 1-\beta_{t-1} = \frac{B_{t-1}}{B_{t}}.
    \]
    Injecting those relations in \cref{eq:temp_recursion_fx} gives
    \begin{align*}
        (f(x_{t+1})-f^\star) \leq \frac{B_{t}}{B_{t+1}} (f(x_t)-f^\star) + \left( \frac{b_{t+1}}{B_{t+1}}\right)^3\frac{LR^3}{2}  + \left( \frac{b_{t+1}}{B_{t+1}}\right)^2\frac{R^2C_2}{2},
    \end{align*}
    hence the recursion
    \begin{align*}
        B_{t+1}(f(x_{t+1})-f^\star) & \leq B_t (f(x_t)-f^\star) + \frac{b^{3}_{t+1}}{B^2_{t+1}}\frac{LR^3}{2}  + \frac{b^2_{t+1}}{B_{t+1}}\frac{R^2C_2}{2}\\
        & \leq B_0 (f(x_t)-f^\star) + \sum_{i=0}^{t}\frac{b^{3}_{i+1}}{B^2_{i+1}}\frac{LR^3}{2}  + \sum_{i=0}^{t}\frac{b^2_{i+1}}{B_{i+1}}\frac{R^2C_2}{2}.
    \end{align*}
    \begin{align*}
        (f(x_{t+1})-f^\star) \leq \frac{B_0}{B_{t+1}} (f(x_t)-f^\star) + \frac{\sum_{i=0}^{t}\frac{b^{3}_{i+1}}{B^2_{i+1}}}{B_{t+1}}\frac{LR^3}{2}  + \frac{\sum_{i=0}^{t}\frac{b^2_{i+1}}{B_{i+1}}}{B_{t+1}}\frac{R^2C_2}{2}.
    \end{align*}
    Therefore, the rate reads
    By the definition of $b_t$ and $B_t$, 
    \begin{align*}
        \frac{b^{3}_{i+1}}{B^2_{i+1}} & = \frac{36}{8} \frac{(i+1)^3(i+2)^3}{(i+1)^2(i+2)^2(i+3)^2} = \frac{9}{2} \frac{(i+1)(i+2)}{(i+3)^2} \leq \frac{9}{2},\\
        \frac{b^{2}_{i+1}}{B_{i+1}} & = \frac{6}{4} \frac{(i+1)^2(i+2)^2}{(i+1)(i+2)(i+3)} = \frac{3}{2} \frac{(i+2)}{(i+3)}(i+1) \leq \frac{3}{2}(i+1).
    \end{align*}
    Hence,
    \begin{align*}
         \frac{\sum_{i=0}^{t}\frac{b^{3}_{i+1}}{B^2_{i+1}}}{B_{t+1}} & \leq \frac{\frac{9}{2}(t+1)}{\frac{(t+1)(t+2)(t+3)}{6}}\leq \frac{27}{(t+2)(t+3)},\\
         \frac{\sum_{i=0}^{t}\frac{b^2_{i+1}}{B_{i+1}}}{B_{t+1}} & \leq \frac{\sum_{i=0}^t\frac{3}{2}(i+1)}{\frac{(t+1)(t+2)(t+3)}{6}} =  \frac{\frac{3}{4}(t+2)(t+1)}{\frac{(t+1)(t+2)(t+3)}{6}}= \frac{9}{2(t+3)}.
    \end{align*}
    Shifting from $t+1$ tp $t$ gives the desired result,
    \[
        (f(x_{t})-f^\star) \leq 6\frac{f(x_t)-f^\star}{t(t+1)(t+2)} + \frac{1}{(t+1)(t+2)}\frac{L(3R)^3}{2}  + \frac{1}{t+2}\frac{C_2(3R)^2}{4}.
    \]
\end{proof}

\thmraterandomconvex*
\begin{proof}
    The proof technique is similar to \cite{hanzely2020stochastic}. Starting from \cref{prop:global_lower_bound_xplus_stoch} with $x=x_t$,
    \begin{align*}
        \mathbb{E}f(x_{t+1}) \leq & \left(1-\frac{N}{d}\right)f(x_t) + \frac{N}{d} f(y)  + \frac{N}{d}\frac{(M_{t+1}+L)}{6}\|y-x_t\|^3 \\
        & + \frac{N}{d}\frac{\|y-x_t\|^2}{2} \left(\delta \frac{L \kappa +  M_{t+1}\kappa^2}{2} +\frac{(d-N)}{d}\|\nabla^2 f(x_t)\|\right),
    \end{align*}
    where the expectation is taken with $D_{0},\ldots, D_{t-1}$ fixed. Using the inequality $M_{t+1}\leq 2L$ gives
    \begin{align*}
        \mathbb{E}f(x_{t+1}) \leq & \left(1-\frac{N}{d}\right)f(x_t) + \frac{N}{d}\left( f(y) + \frac{\|y-x_t\|^2}{2} C_3 + \frac{L}{2}\|y-x_t\|^3\right) 
    \end{align*}
    where
    \[
        C_3\defas \left(\delta L \frac{\kappa +  2\kappa^2}{2} +\frac{(d-N)}{d}\max_{i\in[0,t]}\|\nabla^2 f(x_i)\|\right).
    \]
    Let $y = \beta_t x^\star + (1-\beta_t)x_t$, $\beta_t \in[0,1]$. After using \cref{assump:convex} and \cref{assump:bounded_radius},
    \begin{align*}
        \mathbb{E}f(x_{t+1}) & \leq  \left(1-\frac{N}{d}\right)f(x_t) + \frac{N}{d}\left( f\Big(\beta_t x^\star + (1-\beta_t)x_t\big) + \beta_t^2\frac{C_3R^2}{2}  + \beta_t^3\frac{LR^3}{2}\right) \\ 
        & \leq  \left(1-\frac{N}{d}\right)f(x_t) + \frac{N}{d}\left( \beta_t f(x^\star) + (1-\beta_t) f(x_t) + \beta_t^2\frac{C_3R^2}{2}  + \beta_t^3\frac{LR^3}{2}\right) \\ 
        & =  \left(1-\frac{N}{d}\right)f(x_t) + \frac{N}{d}\left( \beta_t f(x^\star) + (1-\beta_t) f(x_t) + \beta_t^2\frac{C_3R^2}{2}  + \beta_t^3\frac{LR^3}{2}\right) , \\ 
        & =  \left(1-\beta_t\frac{N}{d}\right)f(x_t) + \frac{N}{d}\left( \beta_t f(x^\star) + \beta_t^2\frac{C_3R^2}{2}  + \beta_t^3\frac{LR^3}{2}\right) .
    \end{align*}
    Hence, the recursion
    \[
        \left(\mathbb{E}f(x_{t+1}) - f^\star\right) \leq  \left(1-\beta_t\frac{N}{d}\right)(f(x_t)-f^\star) + \frac{N}{d}\left( \beta_t^2\frac{C_3R^2}{2}  + \beta_t^3\frac{LR^3}{2}\right).
    \]
    Now, define
    \begin{align*}
        b_t & = t^2,\\
        B_t & = B_0 + \sum_{i=0}^t b_i, \;\; B_0 = \frac{4}{3}\left(\frac{d}{N}\right)^3\\
        \beta_t & = \frac{d}{N} \frac{b_{t+1}}{B_{t+1}} \;\; \Rightarrow \; 1-\frac{N}{d}\beta_t = \frac{B_t}{B_{t+1}}.
    \end{align*}
    Replacing those relations in the recursion gives
    \begin{align*}
        & B_{t+1}\left(\mathbb{E}f(x_{t+1}) - f^\star\right) \\
        \leq &  B_t(f(x_t)-f^\star) + \frac{N}{dB_{t+1}}\left(  \left(\frac{d}{N} \frac{b_{t+1}}{B_{t+1}}\right)^2\frac{C_3R^2}{2}  + \left(\frac{d}{N} \frac{b_{t+1}}{B_{t+1}}\right)^3\frac{LR^3}{2}\right)\\
        = & B_t(f(x_t)-f^\star) +   \frac{d}{N} \frac{b^2_{t+1}}{B_{t+1}}\frac{C_3R^2}{2}  + \frac{d^2}{N^2} \frac{b^3_{t+1}}{B^2_{t+1}}\frac{LR^3}{2}
    \end{align*}
    Expanding the inequality gives
    \begin{align*}
        & B_{t+1}\left(\mathbb{E}f(x_{t+1}) - f^\star\right) \leq B_0(f(x_0)-f^\star) +  \frac{d}{N} \sum_{t=0}^{t+1}\frac{b^2_{i+1}}{B_{i+1}}\frac{C_3R^2}{2}  + \frac{d^2}{N^2} \sum_{t=0}^{t+1}\frac{b^3_{i+1}}{B^2_{i+1}}\frac{LR^3}{2}
    \end{align*}
    Since
    \begin{align*}
        B_t & = B_0 + \sum_{i=1}^t \geq B_0 + \int_0^t x^2 \mathrm{d} x = B_0 + \frac{t^3}{3}\\
        \sum_{i=0}^t \frac{b_t^2}{B_t} & \leq \sum_{i=0}^t \frac{i^4}{B_0+i^3/3} \leq 3t^2,\\
        \sum_{i=0}^t \frac{b_t^3}{B^2_t} & \leq \sum_{i=0}^t \frac{i^6}{(B_0+i^3/3)^2} \leq 9t,\\
    \end{align*}
    the bound becomes
    \begin{align*}
        B_{t+1}\left(\mathbb{E}f(x_{t+1}) - f^\star\right) \leq B_0(f(x_0)-f^\star) +  \frac{d}{N}3t^2\frac{C_3R^2}{2}  + \frac{d^2}{N^2} 9t\frac{LR^3}{2}
    \end{align*}
    Dividing both sides by $B_{t+1}$ gives
    \begin{align*}
        \mathbb{E}f(x_{t+1}) - f^\star \leq \frac{B_0}{B_0+\frac{(t+1)^3}{3}}(f(x_0)-f^\star) +  \frac{d}{N}\frac{3(t+1)^2}{B_0+\frac{(t+1)^3}{3}}\frac{C_3R^2}{2}  + \frac{d^2}{N^2} \frac{9(t+1)}{B_0+\frac{(t+1)^3}{3}}\frac{LR^3}{2}.
    \end{align*}
    After the following simplifications,
    \begin{align*}
         \frac{B_0}{B_0+(t+1)^3/3} & =  \frac{1}{1+\frac{(t+1)^3}{3B_0}}= \frac{1}{1+\frac{1}{4}\left(\frac{N}{d}(t+1)\right)^3},\\
         \frac{3(t+1)^2}{B_0+(t+1)^3/3} & = \frac{3}{B_0}\frac{(t+1)^3}{1+\frac{(t+1)^3}{3B_0}} \frac{1}{t+1} \leq \frac{3}{B_0} 3B_0\frac{1}{t+1} = \frac{9}{t+1}, \\
         \frac{9(t+1)}{B_0+\frac{(t+1)^3}{3}} & = \frac{9}{B_0}\frac{(t+1)^3}{\frac{(t+1)^3}{3B_0}} \frac{1}{(t+1)^2} \leq \frac{9}{B_0}3B_0 \frac{1}{(t+1)^2} = \frac{27}{(t+1)^2},
    \end{align*}
    the inequality finally becomes (after shifting from $t+1$ to $t$),
    \begin{align*}
        \mathbb{E}f(x_{t}) - f^\star \leq \frac{1}{1+\frac{1}{4}\left[\frac{N}{d}t\right]^3}(f(x_0)-f^\star) + \frac{1}{\left[\frac{N}{d}t\right]^2} \frac{L(3R)^3}{2} +  \frac{1}{\left[\frac{N}{d}t\right]}\frac{C_3(3R)^2}{2}.
    \end{align*}
\end{proof}

\thmrateaccsketch*
\begin{proof}
    By construction of $\phi_t(x)$, from \cref{prop:recursion_phi} and \cref{assump:bounded_radius},
    \begin{align}
        B_t f(x_t) &\leq \min_x \phi_t(x)\\
        &\leq \phi_t(x^\star)\\
        &\leq B_t f(x^\star) + \frac{\lambda_t^{(1)}+\tilde\lambda^{(1)}}{2}\|x^\star-x_0\|^2 + \frac{\lambda_t^{(2)}+\tilde\lambda^{(2)}}{6}\|x^\star-x_0\|^3\\
        &\leq B_t f(x^\star) + \frac{\lambda_t^{(1)}+\tilde\lambda^{(1)}}{2}R^2 + \frac{\lambda_t^{(2)}+\tilde\lambda^{(2)}}{6}R^3\\
        \Rightarrow f(x_t)-f^\star & \leq  \frac{\lambda_t^{(1)}+\tilde\lambda^{(1)}}{2B_t}R^2 + \frac{\lambda_t^{(2)}+\tilde\lambda^{(2)}}{6B_t}R^3.
    \end{align}
    By \cref{prop:bound_lambdas}, the following bounds holds:
    \begin{align*}
         \lambda_t^{(1)} & \leq 30\cdot\frac{b_{t+1}^2}{B_t} \kappa_D\left(\delta\max\{4L,M_0\} + \max_{i=0\ldots t} \|(I-P_i)\nabla f(x_i)P_i)\|\right),\\
         \lambda_t^{(2)} & \leq 5\frac{b_{t+1}^{3}}{B^2_t}\max\{M_0,4L\}.
     \end{align*}
     Since $\frac{b_{t+1}}{B_t} = \frac{3}{(t+3)}$,
     \begin{align}
         \frac{b_{t+1}^{3}}{B^3_t} = \frac{3^3}{(t+3)^3}, \qquad
         \frac{b_{t+1}^{2}}{B^2_t} = \frac{3^2}{(t+3)^2}.
     \end{align}
     Therefore,
     \begin{align*}
        f(x_t)-f^\star \leq & 30\cdot\kappa_D\left(\delta\max\{4L,M_0\} + \max_{i=0\ldots t} \|(I-P_i)\nabla f(x_i)P_i)\|\right) \frac{(3R)^2}{(t+3)^2} \\
        & + 5\max\{M_0,4L\}\left(\frac{3R}{t+3}\right)^3 \\
        & + \frac{\frac{\tilde\lambda^{(1)}R^2}{2} + \frac{\tilde \lambda^{(2)}R^{3}}{6}}{(t+1)^3}. 
     \end{align*}
\end{proof}

\end{document}